\documentclass[a4paper]{article}

\usepackage{url}
\usepackage{hyperref}										
\usepackage{color}
\usepackage{makeidx}										

\makeindex														

\usepackage{amsmath, amsthm, amssymb}			
\usepackage{stmaryrd}										
\usepackage{dsfont}											
\usepackage{amsfonts}										
\usepackage[arrow, matrix, curve]{xy}

\newtheorem{theorem}{Theorem}[section]							
\newtheorem{lemma}[theorem]{Lemma}
\newtheorem{proposition}[theorem]{Proposition}
\newtheorem{corollary}[theorem]{Corollary}
\newtheorem*{theoremunnum}{Theorem}

\theoremstyle{definition}
\newtheorem{definition}[theorem]{Definition}

\newtheorem{remark}[theorem]{Remark}

\newcommand{\partderx}[1]{\frac{\partial^{#1}}{\partial x^{#1}}}  			
\newcommand{\partdery}[1]{\frac{\partial^{#1}}{\partial y^{#1}}}

\newcommand{\partderxInd}[1]{\frac{\partial}{\partial x_{#1}}}

\newcommand{\oneton}{\left\{1, \ldots, n\right\}}

\newcommand{\zeroinfty}{\left[0,\infty\left[\right.\right.}
\newcommand{\Natzeroinfty}{\Natural_0\cup\left\{\infty\right\}}

\newcommand{\interclcl}[2]{\left[#1,#2\right]}
\newcommand{\interopcl}[2]{\left.\right]#1,#2\left.\right]}
\newcommand{\interclop}[2]{\left[\right.#1,#2\left[\right.}

\newcommand{\contsemiE}{\mathcal{P}_E}
\newcommand{\contsemiF}{\mathcal{P}_F}
\newcommand{\Weights}{\mathcal{W}}
\newcommand{\Real}{\mathds{R}}
\newcommand{\Natural}{\mathds{N}}
\newcommand{\cpsub}[1]{\mathcal{K}(#1)}
\newcommand{\defeq}{\stackrel{def}{=}}
\newcommand{\projectlim}{\underleftarrow{\lim}\:}

\newcommand{\Ht}{Hausdorff topological }
\newcommand{\Hlc}{Hausdorff locally convex }

\DeclareMathOperator{\ev}{ev}
\DeclareMathOperator{\supp}{supp}
\DeclareMathOperator{\id}{id}
\DeclareMathOperator{\Evol}{Evol}
\DeclareMathOperator{\evol}{evol}

\begin{document}
\title{\textbf{Exponential laws for weighted function spaces and regularity of weighted mapping groups}}
\author{Natalie Nikitin}
\date{}
\maketitle

\begin{abstract}
Let $E$ be a locally convex space, $U\subseteq\Real^n$ as well as $V\subseteq\Real^m$
be open and $k,l\in\Natzeroinfty$. Locally convex spaces $C^{k,l}(U\times V,E)$ of functions 
with different degrees of differentiability in the $U$- and $V$-variable were recently studied by H.Alzaareer, 
who established
an exponential law of the form $C^{k,l}(U\times V,E)\cong C^k(U,C^l(V,E))$. We establish an
analogous exponential law 
$C^{k,l}_{\Weights_1\otimes\Weights_2}(U\times V,E)\cong C^k_{\Weights_1}(U,C^l_{\Weights_2}(V,E))$
for suitable spaces of weighted $C^{k,l}$-maps, as well as an analogue for spaces of weighted continuous 
functions on locally compact spaces. The results entail that certain Lie groups $C^l_\Weights(U,H)$ of
weighted mappings introduced by B.Walter are $C^k$-regular, for each $C^k$-regular Lie group $H$ modeled
on a locally convex space and a suitable set of weights $\Weights$.
\end{abstract}

\section{Introduction}
For many purposes in Mathematical Analysis it is very natural to control the growth of a continuous (or differentiable)
function by means of weight functions. Prime examples ate the Schwartz spaces od rapidly decreasing smooth 
functions in the theory of (tempered) distributions.

\medskip
Spaces of weighted continuous functions on topological spaces were introduced and studied by L.Nachbin \cite{Nachbin} 
and W.H. Summers \cite{Summers} (in the scalar-valued case), K.-D. Bierstedt \cite{Bierstedt} and 
Prolla J.B. \cite{Prolla} (in the vector-valued case), and many others. Many results
concerning spaces of weighted differentiable functions can be found in \cite{Garnir-DeWilde-SchmetsIII} and
\cite{BWalter}. The aim of this article is the development of an Exponential Law for spaces of weighted differentiable functions with values in locally convex spaces, which says that under some conditions we have

\begin{equation*}
	C_{\Weights_1}^k(U,C_{\Weights_2}^l(V,E))\cong C_{\Weights_1\otimes\Weights_2}^{k,l}(U\times V,E),
\end{equation*}

where $U\subseteq\Real^n$, $V\subseteq\Real^m$ are open subsets, $\Weights_1$, $\Weights_2$ are sets of weights on $U$ and $V$, respectively, 
and $k,l\in\Natzeroinfty$.

\medskip
Starting with a simpler situation, we consider a subspace $C_\Weights(X,E)$ of $C(X,E)$, where $X$ is a \Ht space, $E$ is a \Hlc space and
$\Weights$ is a set of functions $f:X\to\zeroinfty$, called \textit{weights}, and endow it with a locally convex topology (see Definition \ref{def:weighted-function-space} for details).
Using the Exponential Law for spaces of continuous functions (Proposition \ref{prop:classical-exp-law}, as discussed in 
\cite[Appendix B]{GlExpLaws}), we prove Theorem \ref{thm:weighted-exp-law-comb}, which states:

\begin{theoremunnum}[\textbf{Exponential Law for spaces of weighted continuous functions}]
Let $X_1$, $X_2$ be locally compact spaces and $E$ be a \Hlc space. Let $\Weights_1$ and $\Weights_2$ be sets
of weights on $X_1$ and $X_2$, respectively, such that

\begin{enumerate}
	\item [(i)] $\Weights_1$, $\Weights_2$ satisfy the $o$-condition,
	\item [(ii)] all weights $f\in\Weights_1$, $g\in\Weights_2$ are bounded on compact subsets of $X_1$ and $X_2$, respectively,
	\item [(iii)] for each compact subset $K\subseteq X_1$ there exists a weight $f\in\Weights_1$ such that $\inf_{x\in K}f(x)>0$, and 
	likewise for $\Weights_2$.
\end{enumerate}

Then the linear map

\begin{equation*}
	\Psi:C_{\Weights_1}(X_1, C_{\Weights_2}(X_2,E))\to C_\Weights(X_1\times X_2, E),\quad\gamma\mapsto\gamma^\wedge,
\end{equation*}

where $\Weights=\Weights_1\otimes\Weights_2$, is a homeomorphism.
\end{theoremunnum}

Further, after recalling the concept of differentiability in locally convex spaces, 
we pass on to spaces of weighted differentiable functions introduced
in Definitions \ref{def:weighted-Ck-function-space} and \ref{def:weighted-Ckl-function-space}. 
Using the Exponential Law for spaces of differentiable functions (see Proposition 
\ref{cor:part-Ckl-classical-exponential-law}, as proven in \cite[Theorem 94]{Alz}), we prove the following
Theorem \ref{thm:weighted-Ckl-exp-law-comb}:

\begin{theoremunnum}[\textbf{Exponential Law for spaces of weighted differentiable functions}]
Let $E$ be a \Hlc space, $U\subseteq\Real^n$ and $V\subseteq\Real^m$ be open subsets, and $k,l\in\Natzeroinfty$.
For the set of weights $\Weights_1\subseteq C^k(U,\zeroinfty)$ on $U$ we assume that

\begin{itemize}
	\item [(i)]	$\Weights_1$ satisfies the $o$-condition,
	\item [(ii)]	for each $f\in\Weights_1$ and $\alpha\in\Natural_0^n$ with $|\alpha|\leq k$ there exists 
					$g\in\Weights_1$ such that
					
					\begin{equation*}
						\left|\partderx{\alpha}f(x)\right|\leq g(x)
					\end{equation*}
					
					for all $x\in U$,
\end{itemize}

and likewise for the set of weights $\Weights_2\subseteq C^l(V,\zeroinfty)$ on $V$. Then the linear map 

\begin{equation*}
	\Psi:C_{\Weights_1}^k(U,C_{\Weights_2}^l(V,E))\to C_\Weights^{k,l}(U\times V,E),\quad\gamma\mapsto\gamma^\wedge,
\end{equation*}

where $\Weights=\Weights_1\otimes\Weights_2$, is a homeomorphism.
\end{theoremunnum}

In the proof, we crucially use that
the space $C_c^{k,l}(U\times V,E)$ is dense in
$C_\Weights^{k,l}(U\times V,E)$  (see Proposition 
\ref{prop:smooth-compact-supp-dense-in-weighted-Ckl-assertion}). The proof of this proposition varies 
the proof of the density of $C_c^\infty(U,\Real)$ in the space $C_\Weights^k(U,\Real)$ by H.G. Garnir,
M. De Wilde and J. Schmets in \cite{Garnir-DeWilde-SchmetsIII}. An immediate consequence (Corollary 
\ref{cor:weighted-Ckl-exp-law}) of the Exponential Law is that we have

\begin{equation*}
	C_{\Weights_1}^k(U,C_{\Weights_2}^l(V,E))\cong C_{\Weights_2}^l(V,C_{\Weights_1}^k(U,E))
\end{equation*}

for all $k,l\in\Natzeroinfty$ and suitable sets of weights $\Weights_1$, $\Weights_2$. Moreover, after some 
modifications in the proof of Proposition \ref{prop:smooth-compact-supp-dense-in-weighted-Ckl-assertion},
we obtain that

\begin{equation*}
	C_{\Weights_1}^k(U,C_{\Weights_2}^l(K,E))\cong C_{\Weights_2}^l(K,C_{\Weights_1}^k(U,E)),
\end{equation*}

where $U\subseteq\Real^n$ is open, $K\subseteq\Real^m$ is convex and compact, and $\Weights_1$, $\Weights_2$ 
are appropriate
sets of weights on $U$ and $K$ (see Corollary \ref{cor:weighted-Ckl-exp-law-cp-open}).

As a special case, constructing a set of weights $\Weights_1$ on a subset $U\subseteq\Real^n$ such that

\begin{equation*}
	C_{\Weights_1}^k(U,E)=C^k(U,E)
\end{equation*}

as topological vector spaces for each $k\in\Natzeroinfty$, we get the results

\begin{equation*}
	C^k(U,C_{\Weights_2}^l(V,E))\cong C_{\Weights_2}^l(V,C^k(U,E))
\end{equation*}

and

\begin{equation*}
	C^k(K,C_{\Weights_2}^l(V,E))\cong C_{\Weights_2}^l(V,C^k(K,E)),
\end{equation*}

where $U\subseteq\Real^n$ is open, $K\subseteq\Real^m$ is convex and compact, $\Weights_2$ is a suitable set of weights 
on an open subset $V\subseteq\Real^m$, and $l\in\Natzeroinfty$ (see Remark \ref{rem:weighted-Ck-space-eq-Ck-space} for details).

\medskip
The last section deals with regularity of Lie groups of weighted Lie group-valued functions. Precisely, applying Corollary 
\ref{cor:weighted-Ckl-exp-law-cp-open}, we show in Theorem \ref{thm:weighted-Lie-group-Ck-regular}:

\begin{theoremunnum}
Let $E$ be a \Hlc space and $H$ be a $C^k$-regular Lie group modeled on $E$, with $k\in\Natzeroinfty$. Let 
$\Weights\subseteq C^l(U,\zeroinfty)$ be a set of weights on an open subset $U\subseteq\Real^n$ such that

\begin{enumerate}
	\item [(i)] 	$\mathds{1}_U\in\Weights$,
	\item [(ii)]	$\Weights$ satisfies the $o$-condition,
	\item [(iii)]	for each $f\in\Weights$ and $\alpha\in\Natural_0^n$ with $|\alpha|\leq l$ there exists $g\in\Weights$ such that
					
					\begin{equation*}
						\left|\partderx{\alpha}f(x)\right|\leq g(x)
					\end{equation*}
					
					for all $x\in U$.
\end{enumerate}

Then the Lie group $G:=C_\Weights^l(U,H)$ is $C^k$-regular for each $l\in\Natzeroinfty$.
\end{theoremunnum}

For the theory of Lie groups modeled on locally convex spaces, the reader is referred to \cite{Milnor}, \cite{Neeb}
and \cite{GlNeeb}. It is known (cf. \cite{Milnor} and \cite{GlLieGrStr}) that if $H$ is a locally convex Lie group and $K$ is
a smooth compact manifold, then $C^l(K,H)$ is a locally convex Lie group modeled on the space $C^l(K,L(H))$ for
each $l\in\Natzeroinfty$. 
In Definition \ref{def:lie-group-Ck-reg}, we recall the concept of $C^k$-regularity of Lie groups, which goes back to J.Milnor 
(see \cite{Milnor}),
who works with $C^\infty$-regularity (simply called regularity). If $H$ is a $C^k$-regular Lie group, then
the Lie group $C^l(K,H)$ 
is $C^k$-regular for each 
$l\in\Natzeroinfty$
(as proven by H. Gl\"ockner in  \cite{GlSemireg}). 
The construction of 
Lie groups of weighted 
Lie group-valued functions is discussed in \cite{BWalter} (generalizing the seminal work
of H. Boseck, G. Czichowski and K.-P. Rudolph \cite{Boseck-Czichowski-Rudolph}). Moreover, 
B.Walter shows in \cite{BWalter} that if $U$ 
is an open subset of a normed space, $\mathds{1}_U\in\Weights$ and $H$ is a Banach Lie group, then the Lie group $C_\Weights^l(U,H)$ is regular.

\medskip
Exponential laws for function spaces related to infinite-dimensional Lie groups have also been established
in the recent work \cite{Kriegl-Michor-Rainer} by
A.Kriegl, P.W.Michor and A.Rainer, in the setting of convenient differential calculus. For the most part, the
results are complementary. Taking $U=\Real^n$, $V=\Real^m$, $k=l=\infty$ and $\Weights_1, \Weights_2$ 
as the respective sets of all squares of polynomial functions in Theorem \ref{thm:weighted-Ckl-exp-law-comb},
we obtain an exponential law $\mathcal{S}(\Real^{n+m},E)\cong\mathcal{S}(\Real^n(\mathcal{S}(\Real^m,E))$
for Schwartz spaces of vector-valued rapidly decreasing smooth functions, which (as a bornological isomorphism)
is also covered by \cite{Kriegl-Michor-Rainer}.

\medskip
The results presented here are based on the author's master's thesis \cite{NikMT} advised by Helge Gl\"ockner (Paderborn).

\medskip
All of the topological vector spaces will be $\mathds{K}$-vector spaces, with $\mathds{K}\in\left\{\Real,\mathds{C}\right\}$. 
Further, we denote
the set of all compact subsets of a topological space $X$ by $\cpsub{X}$, and the set of all continuous seminorms on a 
locally convex
space $E$ will be denoted by $\contsemiE$.

\section{Spaces of weighted continuous functions and the Exponential Law}

In this section we present the construction of the space of weighted continuous functions with values in a locally 
convex space and the corresponding topology. We study some important properties of such spaces and finally we
establish an Exponential Law for spaces of weighted continuous functions.

\begin{definition}\label{def:weighted-function-space}
Let $X$ be a \Ht space and $E$ be a \Hlc space. We denote by $\Weights$ a nonempty set of maps $f:X\to\zeroinfty$ 
such that for each $x\in X$ there exists $f_x\in\Weights$ with $f_x(x)>0$ (and call the elements of $\Weights$ \textit{weights}\index{weight}).

For a continuous function $\gamma:X\to E$, a seminorm $q\in\mathcal{P}_E$ and $f\in\Weights$ we define

\begin{equation*}
	\left\|\gamma\right\|_{f,q}:=\sup_{x\in X}f(x)q(\gamma(x))\in\interclcl{0}{\infty}.
\end{equation*}

Furthermore, we define the vector space of \textit{weighted continuous functions}\index{weighted continuous function}\index{function!weighted continuous}

\begin{equation*}
	C_\Weights(X,E):=\left\{\gamma\in C(X,E) : \left(\forall f\in\Weights\right)\left(\forall q\in\contsemiE\right)\left\|\gamma\right\|_{f,q}<\infty\right\}
\end{equation*}

and endow it with the locally convex topology induced by the seminorms

\begin{equation*}
	\left\|\cdot\right\|_{f,q}:C_\Weights(X,E)\to\zeroinfty.
\end{equation*}

For a subset $U\subseteq X$ we write

\begin{equation*}
	C_\Weights(U,E):=C_{\Weights\left|_U\right.}(U,E),
\end{equation*}

where

\begin{equation*}
	\Weights\left|_U\right.:=\left\{f\left|_U\right. :f\in\Weights\right\}.
\end{equation*}

Finally, for a subset $V\subseteq E$ we define

\begin{equation*}
	C_\Weights(X,V):=\left\{\gamma\in C_\Weights(X,E) : \gamma(X)\subseteq V\right\}.
\end{equation*}
\end{definition}

\begin{remark}\label{rem:weighted-function-space-Hausdorff}
The point evaluation

\begin{equation*}
	\ev_x:C_\Weights(X,E)\to E,\quad\ev_x(\gamma):=\gamma(x)
\end{equation*}

is continuous for all $x\in X$, since

\begin{equation*}
	q(\gamma(x))=\frac{1}{f_x(x)}f_x(x)q(\gamma(x))\leq\frac{1}{f_x(x)}\left\|\gamma\right\|_{f_x,q},
\end{equation*}

for each seminorm $q\in\contsemiE$ and a certain weight $f_x\in\Weights$ with $f_x(x)>0$. Thus, 
the topology on $C_\Weights(X,E)$ is Hausdorff.
\end{remark}

\begin{remark}\label{rem:sum-of-weights-in-set}
Let $X$ be a \Ht space and $E$ be a \Hlc space. If $\Weights_1$ is a set of weights on $X$, then for the set of weights

\begin{equation*}
	\Weights_2:=\left\{\sum_{i=1}^n r_i f_i : f_i\in\Weights_1,r_i>0\mbox{ for }i\in\oneton,n\in\Natural\right\}
\end{equation*}

on $X$ we have

\begin{equation*}
	C_{\Weights_1}(X,E) = C_{\Weights_2}(X,E)
\end{equation*}

as topological vector spaces. In fact, since $\Weights_1\subseteq\Weights_2$, we have 
$C_{\Weights_2}(X,E)\subseteq C_{\Weights_1}(X,E)$ and the inclusion map $C_{\Weights_2}(X,E)\to C_{\Weights_1}(X,E)$ is continuous. Conversely, if $\gamma\in C_{\Weights_1}(X,E)$, $f\in\Weights_2$ (that is $f=r_1f_1+\cdots+r_nf_n$ for some weights
$f_1,\ldots,f_n\in\Weights_1$ and $r_1,\ldots,r_n>0$) and $q\in\contsemiE$, then

\begin{equation*}
	\left\|\gamma\right\|_{f,q}=\left\|\gamma\right\|_{r_1f_1+\cdots+r_nf_n,q}\leq r_1\left\|\gamma\right\|_{f_1,q}+\cdots+r_n\left\|\gamma\right\|_{f_n,q}<\infty,
\end{equation*}

thus $\gamma\in C_{\Weights_2}(X,E)$ and the inclusion map $C_{\Weights_1}(X,E)\to C_{\Weights_2}(X,E)$ is continuous. 

Therefore, we can always assume that for a set of weights $\Weights$ on $X$ we have

\begin{equation*}
	\left(\forall f_1,\ldots,f_n\in\mathcal{W}\right)\left(\forall r_1,\ldots,r_n>0\right)r_1f_1+\cdots+r_nf_n\in\mathcal{W}.
\end{equation*}
\end{remark}

\begin{remark}\label{rem:cp-open-top-and-top-of-cp-conv}
Recall that a basis for the \textit{compact-open topology}\index{topology!compact-open}\index{compact-open topology} 
on $C(X,Y)$, where $X$, $Y$ are \Ht spaces, is given by the sets 

\begin{equation*}
	\left\lfloor K_1,U_1\right\rfloor\cap\ldots\cap\left\lfloor K_n,U_n\right\rfloor,
\end{equation*}

where $n\in\mathds{N}$, $K_1,\ldots,K_n\in\cpsub{X}$, $U_1,\ldots,U_n\subseteq Y$ are open sets and

\begin{equation*}
	\left\lfloor K_i,U_i\right\rfloor:=\left\{\gamma\in C(X,Y) : \gamma(K_i)\subseteq U_i\right\}
\end{equation*}

for each $i\in\oneton$. We always endow the space $C(X,Y)$ with the compact-open topology.
Further, if $E$ is a \Hlc space, then the compact-open topology on $C(X,E)$ coincides with the
locally convex topology induced by the seminorms

\begin{equation*}
	\left\|\cdot\right\|_{K,q}:C(X,E)\to\zeroinfty,\quad\left\|\gamma\right\|_{K,q}:=\sup_{x\in K}q(\gamma(x)),
\end{equation*}

where $K\in\cpsub{X}$ and $q\in\contsemiE$. (This topology is known as the 
\textit{topology of uniform convergence on compact sets}\index{topology!of uniform convergence on compact sets}.
For details, see, for example, \cite{BourbGenTop510}.
\end{remark}

Now we can easily show the continuity of the following inclusion map:

\begin{lemma}\label{lem:inclusion-map-continuous}
Let $X$ be a \Ht space, $E$ be a \Hlc space and $\Weights$ be a set of weights on $X$. Assume that for each compact subset $K\subseteq X$ 
there exists a weight $f_K\in\Weights$ such that $\inf_{x\in K}f_K(x)>0$. Then the inclusion map

\begin{equation*}
	i:C_\Weights(X,E)\to C(X,E),
\end{equation*}

is linear and continuous.
\end{lemma}

\begin{proof}
The linearity of the inclusion map $i$ is clear. Now, for a compact subset $K\subseteq X$ we have $\varepsilon:=\inf_{x\in K}f_K(x)>0$ for a suitable weight $f_K\in\Weights$. Then

\begin{equation*}
	\varepsilon\mathds{1}_K \leq f_K,
\end{equation*} 

and hence for a seminorm $q\in\contsemiE$ we get

\begin{equation*}
	\begin{split}
		\left\|\gamma\right\|_{K,q} 
		&\defeq \sup_{x\in K}q(\gamma(x)) = \sup_{x\in X}\mathds{1}_K(x)q(\gamma(x))\\
		&\leq \frac{1}{\varepsilon}\sup_{x\in X}f_K(x)q(\gamma(x)) = \frac{1}{\varepsilon}\left\|\gamma\right\|_{f_K,q},
	\end{split}
\end{equation*}
for each $\gamma\in C_\Weights(X,E)$. Thus the map $i$ is continuous.
\end{proof}

\begin{remark}\label{rem:weights-cont-then-inf-geq-0}
If all given weights on $X$ are continuous, then the condition in Lemma \ref{lem:inclusion-map-continuous} is satisfied,
that is, for each compact subset $K\subseteq X$ there is a weight $f_K\in\Weights$ such that $\inf_{x\in K}f_K(x)>0$.
In fact, for each weight $f\in\Weights$ we define the set

\begin{equation*}
	U_f:=\left\{x\in X : f(x)>0\right\},
\end{equation*}

which is an open subset of $X$, since $f$ is continuous. By definition of $\Weights$, for each 
$x\in K$ there is a weight $f_x\in\Weights$ such that $f_x(x)>0$, thus $x\in U_{f_x}$. Since $K$ is compact, there exist $x_1,\ldots,x_n\in K$ such that 

\begin{displaymath}
	K\subseteq\bigcup_{i=1}^n U_{f_{x_i}},
\end{displaymath}

that is, for each $x\in K$ we have $x\in U_{f_{x_i}}$ for some $i\in\oneton$. We set

\begin{displaymath}
	f_K:=f_{x_1}+\cdots+f_{x_n},
\end{displaymath}

which is a weight on $X$ (see Remark \ref{rem:sum-of-weights-in-set}). Then we see that

\begin{equation*}
	(\forall x\in K)(\exists i\in\oneton) f_K(x)\geq f_{x_i}(x)>0,
\end{equation*}

and, using that the minimum is attained as $K$ is compact, we obtain

\begin{equation*}
	\inf_{x\in K}f_K(x)>0,
\end{equation*}

as required.
\end{remark}

Also superposition operators $C_\Weights(X,\lambda)$ are continuous.

\begin{lemma}\label{lem:superposition-operator-cont}
Let $E$, $F$ be \Hlc spaces and $\lambda:E\to F$ be a continuous linear function. Let $X$ be a \Ht space and $\Weights$ be a set of weights on $X$. 
If $\gamma\in C_\Weights(X,E)$, then

\begin{equation*}
	\lambda\circ\gamma\in C_\Weights(X,F)
\end{equation*}

and the map

\begin{equation*}
	C_\Weights(X,\lambda):C_\Weights(X,E)\to C_\Weights(X,F),\quad\gamma\mapsto\lambda\circ\gamma
\end{equation*}

is continuous and linear.
\end{lemma}

\begin{proof}
If $q\in\contsemiF$, then $q\circ\lambda\in\contsemiE$. Therefore, for a weight $f\in\Weights$ we see that

\begin{equation*}
	\left\|\lambda\circ\gamma\right\|_{f,q}\defeq\sup_{x\in X}f(x)q(\lambda(\gamma(x)))=\left\|\gamma\right\|_{f,q\circ\lambda}<\infty,
\end{equation*}

since $\gamma\in C_\Weights(X,E)$. Hence $\lambda\circ\gamma\in C_\Weights(X,F)$.

The linearity of the map $C_\Weights(X,\lambda)$ is clear, and we have

\begin{equation*}
	\left\|C_\Weights(X,\lambda)(\gamma)\right\|_{f,q}=\left\|\lambda\circ\gamma\right\|_{f,q}=\left\|\gamma\right\|_{f,q\circ\lambda},
\end{equation*}

thus the continuity follows.
\end{proof}

\begin{remark}\label{rem:k-space}
We recall that a topological space $X$ is called a \textit{$k$-space}\index{space!$k$-space}\index{$k$-space} if each subset $A\subseteq X$ is closed if and only if
$A\cap K$ is closed in $K$ for each subset $K\in\cpsub{X}$. In this case, a function $\gamma:X\to Y$ to a 
topological space $Y$ is continuous if and only if $\gamma\big|_K$ is continuous for each $K\in\cpsub{X}$. For
example, each locally convex space and each metrizable space is a $k$-space. (Details can be found in 
\cite{KelleyGenTop}).
\end{remark}

We show that in the following case the space $C_\Weights(X,E)$ is complete.

\begin{proposition}\label{prop:weighted-space-complete}
Let $X$ be a $k$-space and $E$ be a complete \Hlc space. If $\Weights$ is a set of weights on $X$ such that for 
each compact set 
$K\subseteq X$ there exists a weight $f_K\in\Weights$ such that $\inf_{x\in K}f_K(x)>~0$, then the space $C_\Weights(X,E)$ is 
complete.
\end{proposition}

\begin{proof}
Let $(\gamma_a)_{a\in A}$ be a Cauchy net in $C_\Weights(X,E)$. Since the inclusion map $i:C_\Weights(X,E)\to C(X,E)$
is continuous (see Lemma \ref{lem:inclusion-map-continuous}), $(\gamma_a)_{a\in A}$ is a Cauchy net in $C(X,E)$. But the space $C(X,E)$ is complete, by \cite[7, Thm.12]{KelleyGenTop}, whence $(\gamma_a)_{a\in A}$ converges to a
$\gamma\in C(X,E)$.
It remains to show that $\gamma\in C_\Weights(X,E)$ and the Cauchy net $(\gamma_a)_{a\in A}$ converges to $\gamma$ in $C_\Weights(X,E)$. To this end, let $f\in\Weights$, $q\in\contsemiE$ and $\varepsilon>0$. There exists an index $a_\varepsilon\in A$ 
such that

\begin{equation*}
	(\forall a_1, a_2\geq a_\varepsilon)\left\|\gamma_{a_1}-\gamma_{a_2}\right\|_{f,q}\leq\varepsilon,
\end{equation*}

that is

\begin{equation*}
	(\forall a_1, a_2\geq a_\varepsilon, x\in X) f(x)q(\gamma_{a_1}(x)-\gamma_{a_2}(x))\leq\varepsilon,
\end{equation*}

by definition of the seminorm $\left\|\cdot\right\|_{f,q}$. Passing to the limit in $a_2$, we obtain

\begin{equation}\label{eq:conv-of-the-Cauchy-net}
	(\forall a_1\geq a_\varepsilon, x\in X) f(x)q(\gamma_{a_1}(x)-\gamma(x))\leq\varepsilon.
\end{equation}

Thus we see that for all $x\in X$

\begin{align*}
	f(x)q(\gamma(x))	& = f(x)q(\gamma_{a_1}(x)+\gamma(x)-\gamma_{a_1}(x))\\
							& \leq f(x)q(\gamma_{a_1}(x))+ f(x)q(\gamma(x)-\gamma_{a_1}(x))\\
							& \leq f(x)q(\gamma_{a_1}(x))+\varepsilon,
\end{align*}

hence

\begin{equation*}
	\left\|\gamma\right\|_{f,q}\leq\left\|\gamma_{a_1}\right\|_{f,q}+\varepsilon<\infty.
\end{equation*}

Thus $\gamma\in C_\Weights(X,E)$ and \eqref{eq:conv-of-the-Cauchy-net} shows that $(\gamma_a)_{a\in A}$ converges to $\gamma$ in $C_\Weights(X,E)$, as required. 
\end{proof}

We construct a set of weights for maps on products using sets of weights on the two factors.

\begin{definition}\label{def:double-weight}
Let $X_1$ and $X_2$ be \Ht spaces and $\Weights_1$, $\Weights_2$ be sets of weights on $X_1$ and $X_2$, respectively. For $f_1\in\Weights_1$ and $f_2\in\Weights_2$ we define the map

\begin{equation*}
	f_1\otimes f_2:X_1\times X_2\to\zeroinfty,\quad(x_1,x_2)\mapsto f_1(x_1)f_2(x_2)
\end{equation*}

and obtain a set of weights on $X_1\times X_2$ via

\begin{equation*}
	\Weights_1\otimes\Weights_2:=\left\{f_1\otimes f_2 : f_1\in\Weights_1, f_2\in\Weights_2\right\}.
\end{equation*}
\end{definition}

The following lemma will be useful.

\begin{lemma}\label{lem:cont-lin-inj-then-top-embedding}
Let $E$, $F$ be \Hlc spaces and $\lambda:E\to F$ be a continuous, linear and injective function. If for each seminorm $q_1\in\contsemiE$ there exists a seminorm $q_2\in\contsemiF$ such that $q_2(\lambda(x))=q_1(x)$ for all $x\in E$, then $\lambda$ is a topological embedding.
\end{lemma}

\begin{proof}
We need to show that the map

\begin{equation*}
	\left(\lambda\big|^{\lambda(E)}\right)^{-1}:\lambda(E)\to E
\end{equation*}

is continuous. Let $q_1\in\contsemiE$. By assumption, there exists a seminorm $q_2\in\contsemiF$ such that
$q_1=q_2\circ\lambda$. Then

\begin{equation*}
	q_1(\lambda^{-1}(y))=q_2(\lambda(\lambda^{-1}(y)))=q_2(y),
\end{equation*}

for all $y\in\lambda(E)$. Hence $\lambda$ is a topological embedding.
\end{proof}

Before proving the first part of the Exponential Law for spaces of weighted continuous functions,
let us recall the Exponential Law for spaces of continuous functions, which can be found
in \cite[Proposition B.15]{GlExpLaws}.

\begin{proposition}[\textbf{Exponential Law for spaces of continuous functions}]\label{prop:classical-exp-law}
Let $X_1$, $X_2$, $Y$ be \Ht spaces. If $\gamma:X_1\times X_2\to Y$ is a continuous map, then also the map

\begin{equation*}
	\gamma^\vee:X_1\to C(X_2,Y),\quad x\mapsto\gamma^\vee(x):=\gamma(x,\bullet)
\end{equation*}

is continuous. Moreover, the map

\begin{equation*}
	\Phi:C(X_1\times X_2,Y)\to C(X_1,C(X_2,Y)),\quad\gamma\mapsto\gamma^\vee
\end{equation*}

is a topological embedding.

If $X_2$ is locally compact or $X_1\times X_2$ is a k-space, then $\Phi$ is a homeomorphism.
\end{proposition}

\begin{theorem}\label{thm:weighted-exp-top-emb}
Let $E$ be a \Hlc space and $X_1$, $X_2$ be \Ht spaces such that $X_2$ is locally compact or $X_1\times X_2$ is a k-space. Let $\Weights_1$ and $\Weights_2$ be sets of weights on $X_1$ and $X_2$, respectively. We assume that for each compact subset $K\subseteq X_1$ there exists a weight $f\in\Weights_1$ such that 
$\inf_{x\in K}f(x)>0$, and likewise for $\Weights_2$. 
If $\gamma\in C_{\Weights_1}(X_1, C_{\Weights_2}(X_2,E))$, then

\begin{equation*}
	\gamma^\wedge\in C_\Weights(X_1\times X_2, E),
\end{equation*}

where $\gamma^\wedge$ is the map

\begin{equation*}
	\gamma^\wedge:X_1\times X_2\to E,\quad\gamma^\wedge(x_1,x_2):=\gamma(x_1)(x_2)
\end{equation*}

and $\Weights=\Weights_1\otimes\Weights_2$. Furthermore, the map

\begin{equation*}
	\Psi:C_{\Weights_1}(X_1, C_{\Weights_2}(X_2,E))\to C_\Weights(X_1\times X_2, E),\quad\gamma\mapsto\gamma^\wedge
\end{equation*}

is linear and a topological embedding.
\end{theorem}

\begin{proof}
By Lemma \ref{lem:superposition-operator-cont}, the map

\begin{equation*}
	C_{\Weights_1}(X_1, i):C_{\Weights_1}(X_1, C_{\Weights_2}(X_2,E))\to C_{\Weights_1}(X_1, C(X_2,E))
\end{equation*}

is continuous and linear, where $i$ is the continuous and linear inclusion map

\begin{equation*}
	i:C_{\Weights_2}(X_2,E)\to C(X_2,E),
\end{equation*}

as in Lemma \ref{lem:inclusion-map-continuous}. Also the inclusion map

\begin{equation*}
	j:C_{\Weights_1}(X_1, C(X_2,E))\to C(X_1, C(X_2,E))
\end{equation*}

is continuous and linear, by Lemma \ref{lem:inclusion-map-continuous}. We assumed that $X_1\times X_2$ is a $k$-space or $X_2$ is locally compact, thus the map $\Phi$ in Proposition
\ref{prop:classical-exp-law} is a homeomorphism. Hence, using the inverse map

\begin{equation*}
	\Phi^{-1}:  C(X_1,C(X_2,E))\to C(X_1\times X_2,E),\quad\gamma\mapsto\gamma^\wedge
\end{equation*}

we set

\begin{equation*}
	\Theta:=\Phi^{-1}\circ j\circ C_{\Weights_1}(X_1, i)
\end{equation*}
 
and obtain the continuous linear map

\begin{equation*}
	\begin{split}
		\Theta:C_{\Weights_1}(X_1, C_{\Weights_2}(X_2,E))	&\to C(X_1\times X_2,E), \\
																		\gamma	&\mapsto\gamma^\wedge.
	\end{split}
\end{equation*}

To show that $\gamma^\wedge\in C_\Weights(X_1\times X_2, E)$ for each $\gamma\in C_{\Weights_1}(X_1, C_{\Weights_2}(X_2,E))$, let $q\in\contsemiE$ and 
$f\in\Weights$, that is $f=f_1\otimes f_2$ for some weights $f_1\in\Weights_1$, $f_2\in\Weights_2$. Then

\begin{align*}
	\left\|\gamma^\wedge\right\|_{f,q}	&\stackrel{def}{=} \sup_{(x_1,x_2)\in X_1\times X_2}f_1\otimes f_2(x_1,x_2)q(\gamma^\wedge(x_1,x_2))\\
													&= \sup_{(x_1,x_2)\in X_1\times X_2}f_1(x_1)f_2(x_2)q(\gamma(x_1)(x_2))\\
													&= \sup_{x_1\in X_1}f_1(x_1)\sup_{x_2\in X_2}f_2(x_2)q(\gamma(x_1)(x_2)).
\end{align*}

Since $\gamma(x_1)\in C_{\Weights_2}(X_2,E)$, we see that

\begin{equation*}
	\sup_{x_2\in X_2}f_2(x_2)q(\gamma(x_1)(x_2))=\left\|\gamma(x_1)\right\|_{f_2,q}<\infty,
\end{equation*}

and hence

\begin{align*}
	\sup_{x_1\in X_1}f_1(x_1)\sup_{x_2\in X_2}f_2(x_2)q(\gamma(x_1)(x_2))	&=\sup_{x_1\in X_1}f_1(x_1)\left\|\gamma(x_1)\right\|_{f_2,q}\\
																											&=\left\|\gamma\right\|_{f_1,\left\|\cdot\right\|_{f_2,q}}<\infty,
\end{align*}

since $\gamma\in C_{\Weights_1}(X_1, C_{\Weights_2}(X_2,E))$, whence $\gamma^\wedge\in C_\Weights(X_1\times X_2, E)$, as required. Therefore, we have $\Theta\left(C_{\Weights_1}(X_1, C_{\Weights_2}(X_2,E))\right)\subseteq  C_\Weights(X_1\times X_2, E)$ and the map

\begin{equation*}
	\begin{split}
		\Psi:=\Theta\left|^{C_\Weights(X_1\times X_2, E)}\right.:C_{\Weights_1}(X_1, C_{\Weights_2}(X_2,E))	&\to C_\Weights(X_1\times X_2, E),\\
																																			\gamma	&\mapsto\gamma^\wedge
	\end{split}
\end{equation*}

is continuous, linear and, obviously, injective. Thus, $\Psi$ is a topological embedding, by Lemma 
\ref{lem:cont-lin-inj-then-top-embedding}, since 

\begin{equation*}
	\left\|\Psi(\gamma)\right\|_{f,q}=\left\|\gamma^\wedge\right\|_{f,q}=\left\|\gamma\right\|_{f_1,\left\|.\right\|_{f_2,q}},
\end{equation*}

for each $f_1\in\Weights_1$, $f_2\in\Weights_2$, $q\in\contsemiE$ and $\gamma\in C_{\Weights_1}(X_1, C_{\Weights_2}(X_2,E))$.
\end{proof}

We want to find conditions ensuring that the topological embedding defined in the preceding theorem will be bijective.
We will work with spaces of continuous compactly supported functions:

\begin{definition}\label{def:cp-supp}
Let $X$ be a \Ht space and $E$ be a Hausdorff locally convex space. For a compact subset $K\subseteq X$ we define the space

\begin{equation*}
	C_K(X,E):=\left\{\gamma\in C(X,E) : \supp (\gamma)\subseteq K\right\}
\end{equation*}

(where $\supp(\gamma)$ denotes the support of $\gamma$) and endow it with the locally convex topology defined by the seminorms

\begin{equation*}
	\left\|\cdot\right\|_q:C_K(X,E)\to\zeroinfty,\quad\left\|\gamma\right\|_q:=\sup_{x\in K}q(\gamma(x)),
\end{equation*}

where $q\in\contsemiE$.

Additionally, we define the space

\begin{equation*}
	C_c(X,E):=\left\{\gamma\in C(X,E) : \supp (\gamma) \mbox{ is compact}\right\}.
\end{equation*}

We obviously have

\begin{equation*}
	C_c(X,E)=\bigcup_{K\in\cpsub{X}}C_K(X,E).
\end{equation*}
\end{definition}

Applying the classical Exponential Law (Proposition \ref{prop:classical-exp-law}), we get the following result:

\begin{lemma}\label{lem:gamma-check-cp-supp}
Let $X_1$, $X_2$ be \Ht spaces and $E$ be a \Hlc space. If
$\gamma\in C_{K_1\times K_2}(X_1\times X_2,E)$ for some compact subsets $K_1\subseteq X_1$, $K_2\subseteq X_2$, then

\begin{equation*}
	\gamma_x:=\gamma(x,\bullet)\in C_{K_2}(X_2,E)
\end{equation*}

for all $x\in X_1$ and 

\begin{equation*}
	\gamma^\vee\in C_{K_1}(X_1,C_{K_2}(X_2,E)),
\end{equation*}

where $\gamma^\vee$ is the map

\begin{equation*}
	\gamma^\vee:X_1\to C_{K_2}(X_2,E),\quad\gamma^\vee(x):=\gamma_x.
\end{equation*}
\end{lemma}

\begin{proof}
The map $\gamma_{x_1}$ is continuous for each $x_1\in X_1$. Further, if 
$x_2\in X_2\backslash K_2$, then we have

\begin{equation*}
	\gamma_{x_1}(x_2)=\gamma(x_1,x_2)=0,
\end{equation*}

whence $\supp(\gamma_{x_1})\subseteq K_2$. Thus we have 

\begin{equation*}
	\gamma_{x_1}\in C_{K_2}(X_2,E).
\end{equation*}

That is, $\gamma^\vee(X_1)\subseteq C_{K_2}(X_2,E)$ and using Proposition \ref{prop:classical-exp-law}, we see that

\begin{equation*}
	\gamma^\vee:X_1\to C_{K_2}(X_2,E),\quad\gamma^\vee(x):=\gamma_x
\end{equation*}

is continuous. For $x_1\in X_1\backslash K_1$ and $x_2\in X_2$ we have

\begin{equation*}
	\gamma^\vee(x_1)(x_2)=\gamma(x_1,x_2)=0,
\end{equation*}

whence $\supp(\gamma^\vee)\subseteq K_1$, and therefore $\gamma^\vee\in C_{K_1}(X_1,C_{K_2}(X_2,E))$, as required.
\end{proof}

The following lemma will be helpful for an important result.

\begin{lemma}\label{lem:K-supp-in-weighted-incl-cont}
Let $X$ be a \Ht space, $E$ be a \Hlc space and $K\subseteq X$ be a compact subset. If $\Weights$ is a set of weights on $X$ such that
each weight $f\in\Weights$ is bounded on $K$, then

\begin{equation*}
	C_K(X,E)\subseteq C_\Weights(X,E)
\end{equation*}

and the inclusion map

\begin{equation*}
	i:C_K(X,E)\to C_\Weights(X,E)
\end{equation*}

is continuous and linear.
\end{lemma}

\begin{proof}
Let $\gamma\in C_K(X,E)$, $f\in\Weights$ and $q\in\contsemiE$. Since $f$ is bounded on $K$, we set

\begin{equation*}
	r:=\sup_{x\in K}f(x)<\infty,
\end{equation*}

and obtain

\begin{equation*}
	\left\|\gamma\right\|_{f,q} \defeq \sup_{x\in X}f(x)q(\gamma(x)) = \sup_{x\in K}f(x)q(\gamma(x))
	\leq r\left\|\gamma\right\|_{q}<\infty,
\end{equation*}

whence $\gamma\in C_\Weights(X,E)$. 

The inclusion map $i:C_K(X,E)\to C_\Weights(X,E)$ is obviously linear and since

\begin{equation*}
	\left\|\gamma\right\|_{f,q}\leq r\left\|\gamma\right\|_{q},
\end{equation*}

the map $i$ is continuous.
\end{proof}

\begin{proposition}\label{prop:cp-supp-in-image}
Let $X_1$, $X_2$ be \Ht spaces, $E$ be a \Hlc space and $\Weights_1$, $\Weights_2$ be sets of weights on $X_1$ and $X_2$, respectively.
If each of the following conditions is satisfied:

\begin{enumerate}
	\item [(i)] $X_2$ is locally compact or $X_1\times X_2$ is a $k$-space,
	\item [(ii)] all weights $f\in\Weights_1$, $g\in\Weights_2$ are bounded on compact subsets of $X_1$ and $X_2$, respectively,
	\item [(iii)] for each compact subset $K\subseteq X_1$ there exists a weight $f\in\Weights_1$ such that $\inf_{x\in K}f(x)>0$, and 
	likewise for $\Weights_2$,
\end{enumerate}

then $C_c(X_1\times X_2,E)\subseteq im(\Psi)$, where $\Psi$ is the topological embedding 
	
\begin{equation*}
	\Psi:C_{\Weights_1}(X_1, C_{\Weights_2}(X_2,E))\to C_\Weights(X_1\times X_2, E),\quad\gamma\mapsto\gamma^\wedge
\end{equation*}
defined in Theorem \ref{thm:weighted-exp-top-emb}.
\end{proposition}

\begin{proof}
Let $\gamma\in C_c(X_1\times X_2,E)$. We show that $\gamma^\vee\in C_{\Weights_1}(X_1,C_{\Weights_2}(X_2,E))$. (Then 
$\Psi(\gamma^\vee)=(\gamma^\vee)^\wedge=\gamma$, and the proof is finished.) Consider the projections

\begin{equation*}
	\pi_1:X_1\times X_2\to X_1,\quad\pi_2:X_1\times X_2\to X_2
\end{equation*}

onto the first and second component, respectively. For $K:=\supp(\gamma)\subseteq X_1\times X_2$ we define $K_1:=\pi_1(K)$ and 
$K_2:=\pi_2(K)$. Since $K$ is compact and the projection maps are continuous, the sets $K_1$ and $K_2$ are compact
and we have $K\subseteq K_1\times K_2$, whence $\gamma\in C_{K_1\times K_2}(X_1\times X_2,E)$. From Lemma 
\ref{lem:gamma-check-cp-supp} we conclude that $\gamma^\vee\in C_{K_1}(X_1,C_{K_2}(X_2,E))$. The inclusion map

\begin{equation*}
	i:C_{K_2}(X,E)\to C_{\Weights_2}(X_2,E)
\end{equation*}

is continuous and linear by Lemma \ref{lem:K-supp-in-weighted-incl-cont}, thus we have 

\begin{equation*}
\gamma^\vee=i\circ\gamma^\vee\in C_{K_1}(X_1,C_{\Weights_2}(X_2,E))\subseteq C_{\Weights_1}(X_1,C_{\Weights_2}(X_2,E)),
\end{equation*}

as required.
\end{proof}

We introduce the following notation:

\begin{definition}\label{def:little-o}
Let $X$ be a \Ht space. For two functions $f,g:X\to\zeroinfty$ we write 

\begin{equation*}
	f=o(g)
\end{equation*}

if for each $\varepsilon>0$ there exists a compact subset $K_\varepsilon\subseteq X$ such that $f(x)\leq\varepsilon g(x)$ for all $x\in X\backslash K_\varepsilon$. 

We say that a set of weights $\Weights$ on $X$ \textit{satisfies the $o$-condition}\index{o-condition} if for each weight $f\in\Weights$ there exists a weight 
$g\in\Weights$ such that $f=o(g)$.
\end{definition}

\begin{lemma}\label{lem:double-weight-little-o}
Let $X_1$, $X_2$ be \Ht spaces. If the sets of weights $\Weights_1$ and $\Weights_2$ on $X_1$ and $X_2$, respectively, satisfy the $o$-condition, then also the set of weights $\Weights_1\otimes\Weights_2$ on $X_1\times X_2$ satisfies the $o$-condition.
\end{lemma}

\begin{proof}
Let $f\in\Weights_1\otimes\Weights_2$, that is $f=f_1\otimes f_2$, where $f_1\in\Weights_1$ and $f_2\in\Weights_2$. By assumption, $f_1=o(g_1)$ and $f_2=o(g_2)$ for some
weights $g_1\in\Weights_1$, $g_2\in\Weights_2$. By Remark \ref{rem:sum-of-weights-in-set}, we can define the weights $h_1:=f_1+g_1\in\Weights_1$, 
$h_2:=f_2+g_2\in\Weights_2$, so that 

\begin{align*}
	f_1\leq h_1,\quad	&	f_1=o(h_1),\\
	f_2\leq h_2,\quad	&	f_2=o(h_2).
\end{align*}

Now, given $\varepsilon>0$, there exist compact subsets $K_1\subseteq X_1$, $K_2\subseteq X_2$ such that for $x_1\in X_1\backslash K_1$, $x_2\in X_2$

\begin{equation*}
	f_1(x_1)f_2(x_2)\leq\varepsilon h_1(x_1)f_2(x_2)\leq\varepsilon h_1(x_1)h_2(x_2),
\end{equation*}

and for $x_1\in X_1$, $x_2\in X_2\backslash K_2$

\begin{equation*}
	f_1(x_1)f_2(x_2)\leq f_1(x_1)\varepsilon h_2(x_2)\leq h_1(x_1)\varepsilon h_2(x_2).
\end{equation*}

Hence, since the subset $K_1\times K_2\subseteq X_1\times X_2$ is compact, we have

\begin{equation*}
	f=o(h),
\end{equation*}

where $h:=h_1\otimes h_2\in\Weights_1\otimes\Weights_2$.
\end{proof}

The next lemma enables us to prove the Exponential Law for spaces of weighted continuous functions.

\begin{lemma}\label{lem:cp-supp-dense-in-weighted}
Let $X$ be a Hausdorff locally compact space and $E$ be a \Hlc space. If $\Weights$ is a set of weights on $X$ such that 

\begin{enumerate}
	\item [(i)] each weight $f\in\Weights$ is bounded on compact subsets of $X$,
	\item [(ii)] the $o$-condition is satisfied,
\end{enumerate}

then $C_c(X,E)$ is dense in $C_\Weights(X,E)$.
\end{lemma}

\begin{proof}
Using Lemma \ref{lem:K-supp-in-weighted-incl-cont}, we conclude that $C_c(X,E)\subseteq C_\Weights(X,E)$. By assumption, $\Weights$ satisfies the $o$-condition,
thus for a weight $f\in\Weights$ there is a weight $g\in\Weights$ such that $f=o(g)$. Let $\gamma\in C_\Weights(X,E)$, $q\in\contsemiE$
and $\varepsilon>0$. We choose a $\delta>0$ such that

\begin{equation}\label{eq:delta-seminorm-less-epsilon}
	\delta\left\|\gamma\right\|_{g,q}<\varepsilon.
\end{equation}

Since $f=o(g)$, there is a compact subset $K_\delta\subseteq X$ with 

\begin{equation}\label{eq:f-leq-delta-g}
	f(x)\leq\delta g(x)
\end{equation}

for all $x\in X\backslash K_\delta$. The space $X$ is assumed locally compact, thus there is an open subset $U_\delta\subseteq X$
such that $\overline{U_\delta}$ is compact and $K_\delta\subseteq U_\delta$. By Urysohn's Lemma, there is a continuous function 
$h:X\to\interclcl{0}{1}$ such that $h\big|_{K_\delta}=1$ and $h\big|_{X\backslash U_\delta}=0$. We set

\begin{equation*}
	\eta:=h\cdot\gamma.
\end{equation*}

Then $\eta\in C_c(X,E)$ because $\supp(\eta)\subseteq\overline{U_\delta}$. Further, 

\begin{equation*}
	f(x)q(\eta(x)-\gamma(x))=0
\end{equation*}

for all $x\in K_\delta$. Now, if $x\in X\backslash K_\delta$, then we have

\begin{align*}
	f(x)q(\eta(x)-\gamma(x)) 
	&= f(x)q(h(x)\gamma(x)-\gamma(x))
	= f(x)\underbrace{|h(x)-1|}_{\leq1}q(\gamma(x))\\
	&\leq f(x)q(\gamma(x))\leq\delta g(x)q(\gamma(x)),
\end{align*}

using \eqref{eq:f-leq-delta-g}. Therefore

\begin{equation*}
	\left\|\eta-\gamma\right\|_{f,q} \leq \delta\left\|\gamma\right\|_{g,q}<\varepsilon,
\end{equation*}

by the choice of $\delta$ in \eqref{eq:delta-seminorm-less-epsilon}. Thus $C_c(X,E)$ is dense in $C_\Weights(X,E)$.
\end{proof}

\begin{theorem}[\textbf{Exponential Law for spaces of weighted continuous functions}]\label{thm:weighted-exp-law-comb}
Let $X_1$, $X_2$ be locally compact spaces and $E$ be a \Hlc space. Let $\Weights_1$ and $\Weights_2$ be sets
of weights on $X_1$ and $X_2$, respectively, such that

\begin{enumerate}
	\item [(i)] $\Weights_1$, $\Weights_2$ satisfy the $o$-condition,
	\item [(ii)] all weights $f\in\Weights_1$, $g\in\Weights_2$ are bounded on compact subsets of $X_1$ and $X_2$, respectively,
	\item [(iii)] for each compact subset $K\subseteq X_1$ there exists a weight $f\in\Weights_1$ such that $\inf_{x\in K}f(x)>0$, and 
	likewise for $\Weights_2$.
\end{enumerate}

Then the linear map

\begin{equation*}
	\Psi:C_{\Weights_1}(X_1, C_{\Weights_2}(X_2,E))\to C_\Weights(X_1\times X_2, E),\quad\gamma\mapsto\gamma^\wedge,
\end{equation*}

where $\Weights=\Weights_1\otimes\Weights_2$, is a homeomorphism.
\end{theorem}

\begin{proof}
\underline{Step $1$.} First we assume that the space $E$ is complete. Since $X$ is a $k$-space, being locally compact,
we conclude from Proposition \ref{prop:weighted-space-complete} that $C_{\Weights_2}(X_2,E)$ is complete, whence 
also $C_{\Weights_1}(X_1, C_{\Weights_2}(X_2,E))$ is complete. By Theorem \ref{thm:weighted-exp-top-emb}, the map $\Psi$ is a topological embedding, thus $im(\Psi)$ is complete, hence it
is closed in $C_\Weights(X_1\times X_2, E)$. We know by Proposition \ref{prop:cp-supp-in-image} that 
$C_c(X_1\times X_2,E)\subseteq im(\Psi)$. Moreover, by Lemma \ref{lem:cp-supp-dense-in-weighted}, the space
$C_c(X_1\times X_2,E)$ is dense in $C_\Weights(X_1\times X_2, E)$ (because $X_1\times X_2$ is locally compact and 
$\Weights=\Weights_1\otimes\Weights_2$ satisfies the $o$-condition, by Lemma \ref{lem:double-weight-little-o}). Thus $im(\Psi)$
is dense in $C_\Weights(X_1\times X_2, E)$. Consequently, we have

\begin{equation*}
	im(\Psi)=\overline{im(\Psi)}=C_\Weights(X_1\times X_2, E),
\end{equation*}

which shows that the topological embedding $\Psi$ is surjective, hence a homeomorphism.

\underline{Step $2$.} Now we show the surjectivity of $\Psi$ in the general case. To this end, let 
$\gamma\in C_\Weights(X_1\times X_2, E)\subseteq C_\Weights(X_1\times X_2, \widetilde{E})$, 
where $\widetilde{E}$ is the completion of $E$. The map

\begin{equation*}
	\widetilde{\Psi}:C_{\Weights_1}(X_1, C_{\Weights_2}(X_2,\widetilde{E}))\to C_\Weights(X_1\times X_2, \widetilde{E}),
	\quad\eta\mapsto\eta^\wedge
\end{equation*}

is bijective, by Step $1$, thus there exists $\eta\in C_{\Weights_1}(X_1, C_{\Weights_2}(X_2,\widetilde{E}))$
such that

\begin{equation*}
	\widetilde{\Psi}(\eta)=\eta^\wedge=\gamma.
\end{equation*}

But this means that

\begin{equation*}
	\eta(x_1)(x_2)=\eta^\wedge(x_1,x_2)=\gamma(x_1,x_2)\in E
\end{equation*}

for all $(x_1,x_2)\in X_1\times X_2$, whence $\eta\in C_{\Weights_1}(X_1, C_{\Weights_2}(X_2,E))$. Therefore, we have
 $\Psi(\eta)=\gamma$, and $\Psi$ is surjective, as asserted.
\end{proof}

\section{Spaces of weighted differentiable functions and the Exponential Law}

In this section we establish an Exponential Law for spaces of weighted differentiable functions. First of all
we recall some concepts of differentiability of maps between locally convex spaces (the calculus of maps defined
on open subsets of locally convex spaces goes back to A.Bastiani \cite{Bastiani} and is also known as
Keller's $C_c^k$-calculus \cite{Keller}). SInce we are also interested in differentiability on sets which are
not necessarily open (for example, on the interval $\interclcl{0}{1}$), we recall the following concept 
(cf. \cite{GlNeeb}):

\begin{definition}\label{def:locally-convex-subset}\
Let $E$ be a \Hlc space. We call a subset $U\subseteq E$ \textit{locally convex}\index{locally convex subset}\index{subset!locally convex}
 if each element $x\in U$ has a convex neighborhood $V\subseteq U$.
\end{definition}

Note that each open set and each convex set with nonempty interior satisfies this condition.

\begin{definition}\label{def:Ck-map}
Let $E$, $F$ be \Hlc spaces
\begin{itemize}
	\item [(a)] 
	Let $\gamma:U\to F$ be a map on an open subset $U\subseteq E$. 
	For $x\in U$ and $h\in E$ the \textit{derivative of $\gamma$ at $x$ in the direction $h$}\index{derivative!directional} is defined as 

	\begin{equation*}
		D_h\gamma(x):=d^{(1)}\gamma(x,h):=d\gamma(x,h):=\lim_{t\to 0} \frac{\gamma(x+th)-\gamma(x)}{t},
	\end{equation*}

	whenever the limit exists.

	For $k\in\Natural$ the map $\gamma$ is called a \textit{$C^k$-map}\index{map!$C^k$} if $\gamma$ is continuous, the \textit{iterated directional derivatives}

	\begin{equation*}
		d^{(j)}\gamma(x,h_1,\ldots,h_j):=(D_{h_j}\cdots D_{h_1}\gamma)(x)
	\end{equation*}

	exist for all $j\in\Natural$ with $j\leq k$, $(x,h_1,\ldots,h_j)\in U\times E^j$, and define continuous functions

	\begin{equation*}
		d^{(j)}\gamma:U\times E^j\to F.
	\end{equation*}
	
	If $\gamma$ is $C^k$ for each $k\in\Natural$, then $\gamma$ is called $C^\infty$ or \textit{smooth}\index{map!smooth}; continuous maps are called $C^0$ 
	and $d^{(0)}\gamma:=\gamma$.
	
	\item[(b)]
	Let $U\subseteq E$ be a locally convex subset with dense interior. 
	A continuous map $\gamma:U\to F$ is called a \textit{$C^k$-map}\index{map!$C^k$} (for $k\in\Natzeroinfty$) if 
	$\gamma\big|_{U^\circ}:U^\circ\to F$ is $C^k$ and for each $j\in\Natural_0$ with $j\leq k$ the map 
	$d^{(j)}\left(\gamma\big|_{U^\circ}\right):U^\circ\times E^j\to F$ admits a (unique) continuous extension
	$d^{(j)}\gamma:U\times E^j\to F$.
\end{itemize}
\end{definition}

In \cite{Alz}, a notion of differentiability on products of locally convex subsets is presented:

\begin{definition}\label{def:C-alpha-map}
Let $E_1$, \ldots, $E_n$ and $F$ be \Hlc spaces.
\begin{itemize}
	\item [(a)] 
	Let $U_i\subseteq E_i$ be an open subset for each $i\in\oneton$ and $\alpha\in(\Natzeroinfty)^n$.
	A continuous function $\gamma:U_1\times\cdots\times U_n\to F$ is called a \textit{$C^\alpha$-map}\index{map!$C^\alpha$} if for all 
	$\beta\in\Natural_0^n$ with
	$\beta\leq\alpha$ the iterated directional derivatives
	
	\begin{equation*}
		d^\beta\gamma(x, h_1,\ldots,h_n):=(\breve{D}_1\cdots\breve{D}_n\gamma)(x),
	\end{equation*}
	
	where
	
	\begin{equation*}
		\breve{D}_i\gamma(x):=(D_{(h_i)_{\beta_i}^*}\cdots D_{(h_i)_1^*}\gamma)(x)
	\end{equation*}
	
	exist for all $x:=(x_1,\ldots, x_n)\in U_1\times\cdots\times U_n$, $h_i:=((h_i)_1,\ldots,(h_i)_{\beta_i})\in E_i^{\beta_i}$,
	$h_i^*:=((h_i)_1^*,\ldots,(h_i)_{\beta_i}^*)\in (\{0\}^{i-1}\times E_i\times\{0\}^{n-1})^{\beta_i}$ and the maps
	
	\begin{equation*}
		d^{\beta}\gamma:U_1\times\cdots\times U_n\times E_1^{\beta_1}\times\cdots\times E_n^{\beta_n}\to F
	\end{equation*}
	
	are continuous.
	
	\item [(b)] 
	Let $U_i\subseteq E_i$ be a locally convex subset with dense interior for each $i\in\oneton$ and $\alpha\in(\Natzeroinfty)^n$.
	A continuous map $\gamma:U_1\times\cdots\times U_n\to F$ is called a \textit{$C^\alpha$-map}\index{map!$C^\alpha$} if
	$\gamma\big|_{U_1^\circ\times\cdots\times U_n^\circ}:U_1^\circ\times\cdots\times U_n^\circ\to F$ is a $C^\alpha$-map
	and for all $\beta\in\Natural_0^n$ with $\beta\leq\alpha$ the maps
	
	\begin{equation*}
		d^\beta(\gamma\big|_{U_1^\circ\times\cdots\times U_n^\circ}):U_1^\circ\times\cdots\times U_n^\circ\times E_1^{\beta_1}\times\cdots\times E_n^{\beta_n}\to F
	\end{equation*}
	
	extend to (unique) continuous maps
	
	\begin{equation*}
		d^{\beta}\gamma:U_1\times\cdots\times U_n\times E_1^{\beta_1}\times\cdots\times E_n^{\beta_n}\to F.
	\end{equation*}
\end{itemize}
\end{definition}

Now, consider the following concept of differentiability for functions defined on subsets of $\Real^n$:

\begin{definition}\label{def:partially-Ck-map}
Let $E$ be a \Hlc space.

\begin{itemize}
	\item[(a)] 
	Let $U\subseteq\Real^n$ be an open subset, $\gamma:U\to E$ be a continuous function, $x\in U$ and $i\in\oneton$.
	We define the \textit{partial derivative \index{derivative!partial} of $\gamma$ with respect to the $i$-th variable} via
	
	\begin{equation*}
		\frac{\partial^1}{\partial x_i}\gamma(x):=\partderxInd{i}\gamma(x):=\lim_{t\to 0}\frac{\gamma(x+te_i)-\gamma(x)}{t},
	\end{equation*}

	whenever the limit exists. We call $\gamma$ \textit{partially $C^1$}\index{map!partially $C^1$} if $\gamma$ is continuous, 
	the partial
	derivative $\partderxInd{i}\gamma(x)$ exists for each $x\in U$
	and $i\in\oneton$ and each of the maps
	
	\begin{equation*}
		\partderxInd{i}\gamma:U\to E
	\end{equation*}
	
	is continuous.
	For $k\in\mathds{N}\cup\{\infty\}$ we call $\gamma$ a \textit{partially $C^k$-map}\index{map!partially $C^k$} if 
	$\gamma$ is 
	partially $C^1$ and each of the partial derivatives $\partderxInd{i}\gamma$ is partially $C^{k-1}$. In this case we 
	denote the higher partial derivatives of $\gamma$ by

	\begin{equation*}
		\partderx{\alpha}\gamma(x):=
		\frac{\partial^{\alpha_1}}{\partial x_1^{\alpha_1}}\cdots\frac{\partial^{\alpha_n}}{\partial x_n^{\alpha_n}}
		\gamma(x)
	\end{equation*}

	for $x\in U$ and $\alpha\in\Natural_0^n$ with $|\alpha|\leq k$, which define continuous functions

	\begin{equation*}
		\partderx{\alpha}\gamma:U\to E.
	\end{equation*}
	
	Finally, we call continuous maps \textit{partially $C^0$}.
	
	\item[(b)] 
	If $U\subseteq\Real^n$ is a locally convex subset with dense interior, then we call a continuous map 
	$\gamma:U\to E$ \textit{partially $C^k$}\index{map!partially $C^k$} (for $k\in\Natzeroinfty$) if 
	$\gamma\big|_{U^\circ}:~U^\circ\to ~E$ is partially $C^k$ and the 
	maps $\partderx{\alpha}\left(\gamma\big|_{U^\circ}\right):U^\circ\to E$ extend to (unique) continuous maps
	$\partderx{\alpha}\gamma:U\to E$ for all $\alpha\in\Natural_0^n$ with $|\alpha|\leq k$.
\end{itemize}
\end{definition}

\begin{remark}\label{rem:Ck-eq-C-alpha-eq-partially-Ck}
If $E$ is a \Hlc space, $U:=U_1\times\cdots\times U_n\subseteq\Real^n$  is a locally convex subset with dense interior,
and $k\in\Natzeroinfty$, then for a continuous function $\gamma:U\to E$ the following assertions are equivalent:

\begin{itemize}
	\item[(i)]
	$\gamma$ is partially $C^k$,
	
	\item[(ii)]
	$\gamma$ is $C^\alpha$ for each $\alpha\in(\Natzeroinfty)^n$ with $|\alpha|\leq k$,
	
	\item[(iii)]
	$\gamma$ is $C^k$.
\end{itemize}

Moreover, the topology on the space $C^k(U,E)$ (consisting of all $C^k$-functions $\gamma:U\to E$) defined
in \cite{Alz} coincides with the locally convex topology defined by the 
seminorms

\begin{equation*}
	\left\|\cdot\right\|_{K,\alpha,q}:C^k(U,E)\to\zeroinfty,
	\quad\left\|\gamma\right\|_{K,\alpha,q}:=\sup_{x\in K}q\left(\partderx{\alpha}\gamma(x)\right),
\end{equation*}

where $\alpha\in\Natural_0^n$ with $|\alpha|\leq k$, $q\in\contsemiE$ and $K\subseteq U$ is compact.
\end{remark}

We pass on to the definition of the space of weighted differentiable functions:

\begin{definition}\label{def:weighted-Ck-function-space}
Let $E$ be a \Hlc space, $U\subseteq\Real^n$ be a locally convex subset with dense interior and 
$\gamma:U\to E$ be partially $C^k$ with $k\in\Natzeroinfty$. Let $\Weights$ be a set of weights on $U$. For a weight $f\in\Weights$, a seminorm $q\in\contsemiE$ and $\alpha\in\Natural_0^n$ with $|\alpha|\leq k$ we define

\begin{equation*}
	\left\|\gamma\right\|_{f,\alpha,q}:=\sup_{x\in U}f(x)q\left(\partderx{\alpha}\gamma(x)\right)\in\interclcl{0}{\infty}.
\end{equation*}

Further, we define the vector space of \textit{weighted $C^k$-functions}

\begin{equation*}
	\begin{split}
	C_\Weights^k(U,E):=\{\gamma\in C^k(U,E) : &(\forall f\in\Weights)(\forall q\in\contsemiE)\\
		&(\forall\alpha\in\Natural_0^n, |\alpha|\leq k)\left\|\gamma\right\|_{f,\alpha,q}<\infty\}
	\end{split}
\end{equation*}

and endow it with the locally convex topology induced by the seminorms

\begin{equation*}
	\left\|\cdot\right\|_{f,\alpha,q}:C_\Weights^k(U,E)\to\zeroinfty
\end{equation*}

where $f\in\Weights$, $q\in\contsemiE$ and $\alpha\in\Natural_0^n$ with $|\alpha|\leq k$.

For a subset $V\subseteq E$ we additionally define

\begin{equation*}
	C_\Weights^k(U,V):=\{\gamma\in C_\Weights^k(U,E):\gamma(U)\subseteq V\}.
\end{equation*}
\end{definition}

\begin{remark}\label{rem:initial-topology}
We recall that for a set $X$, the \textit{initial topology on X with respect to the family $(\gamma_i)_{i\in I}$}\index{topology!initial}\index{initial topology} of mappings
$\gamma_i:X\to Y_i$ into topological spaces $Y_i$ is defined as the coarsest topology making each $\gamma_i$ continuous.
\end{remark}

\begin{remark}\label{rem:top-on-weighted-Ck-space}
The topology on $C_\Weights^k(U,E)$ is Hausdorff, since the point evaluation map 
$\ev_x:C_\Weights^k(U,E)\to E, \gamma\mapsto\gamma(x)$ is continuous for each $x\in U$, this can be 
shown as in Remark \ref{rem:weighted-function-space-Hausdorff}.

Further, the topology on $C_\Weights^k(U,E)$ for $k\in\Natural$ is initial with respect to the maps

\begin{equation*}
	\partderx{\alpha}:C_\Weights^k(U,E)\to C_\Weights(U,E)
\end{equation*}

for $\alpha\in\Natural_0^n$ with $|\alpha|\leq k$. Moreover, using the transitivity of initial topologies 
(cf. \cite{GlNeeb}), 
one can show that this topology is initial with respect to the maps

\begin{equation*}
	\begin{split}
		j:C_\Weights^k(U,E)&\to C_\Weights(U,E),\\
		\partderxInd{i}:C_\Weights^k(U,E)&\to C_\Weights^{k-1}(U,E)
	\end{split}
\end{equation*}	
where $j$ is the inclusion map and $i\in\oneton$. That is, the map

\begin{equation*}
	\Theta:C_\Weights^k(U,E)\to C_\Weights(U,E)\times C_\Weights^{k-1}(U,E)^n,
	\quad\gamma\mapsto(\gamma,\partderxInd{1}\gamma,\ldots\partderxInd{n}\gamma)
\end{equation*}

is a (linear) topological embedding.

Further, the topology on $C_\Weights^\infty(U,E)$ is initial with respect to the inclusion maps

\begin{equation*}
	C_\Weights^\infty(U,E)\to C_\Weights^k(U,E)
\end{equation*}

for $k\in\Natural_0$; we even have

\begin{equation*}
	C_\Weights^\infty(U,E)=\projectlim C_\Weights^k(U,E).
\end{equation*}
\end{remark}

The following lemma can be proven similarly to Lemma \ref{lem:inclusion-map-continuous}.

\begin{lemma}\label{lem:k-inclusion-map-continuous}
Let $E$ be a \Hlc space and $U\subseteq\Real^n$ be a locally convex subset with dense interior. Let $\Weights$ be a set of weights on $U$ such that for each compact subset 
$K\subseteq U$ there is a weight $f_K\in\Weights$ with $\inf_{x\in K}f_K(x)>0$. Then the inclusion map

\begin{equation*}
	i:C_\Weights^k(U,E)\to C^k(U,E)
\end{equation*}

is linear and continuous for each $k\in\Natzeroinfty$.
\end{lemma}

The next two lemmas can be proven with simple induction.

\begin{lemma}\label{lem:lcx-part-Ck-in-image}
Let $F$ be a \Hlc space, $U\subseteq\Real^n$ be a locally convex subset with dense interior and $\gamma:U\to F$ be
a partially $C^k$-map for $k\in\Natural\cup\left\{\infty\right\}$. If $E$ is a subspace of $F$, then $\gamma\in C^k(U,E)$ if
and only if $\partderx{\alpha}\gamma(U)\subseteq E$ for each $\alpha\in\Natural_0^n$ with $|\alpha|\leq k$.
\end{lemma}

\begin{lemma}\label{lem:lcx-partCk-lin-operator}
Let $E$, $F$ be \Hlc spaces and $\lambda:E\to F$ be a continuous linear map. Let $U\subseteq \Real^n$ be a locally
convex subset with dense interior and $\gamma:U\to E$ be a partially $C^k$-map with $k\in\Natural\cup\left\{\infty\right\}$. 
The map $\lambda\circ\gamma:U\to F$ is partially $C^k$ and

\begin{equation*}
	\partderx{\alpha}(\lambda\circ\gamma)=\lambda\circ\partderx{\alpha}\gamma
\end{equation*}

for each $\alpha\in\Natural_0^n$ with $|\alpha|\leq k$.
\end{lemma}

The continuity of superposition operators $C_\Weights^k(U,\lambda)$ can be proven similarly
to Lemma \ref{lem:superposition-operator-cont}, using the preceding fact.

\begin{lemma}\label{lem:k-superposition-operator-cont}
Let $E$, $F$ be \Hlc spaces and $\lambda:E\to F$ be a continuous linear function. Let $U\subseteq\Real^n$ be a locally convex subset with dense interior and $\Weights$ be a set of weights on $U$. If $\gamma\in C_\Weights^k(U,E)$, then

\begin{equation*}
	\lambda\circ\gamma\in C_\Weights^k(U,F)
\end{equation*}

for each $k\in\Natzeroinfty$

Moreover, the map

\begin{equation*}
	C_\Weights^k(U,\lambda):C_\Weights^k(U,E)\to C_\Weights^k(U,F),\quad\gamma\mapsto\lambda\circ\gamma
\end{equation*}

is continuous and linear.
\end{lemma}

For further work we recall a notion of integrability of curves with values in a locally convex space. Many important 
results (for example, the corresponding Fundamental Theorem of Calculus) can be found in \cite{GlNeeb}.

\begin{remark}\label{rem:facts-about-weak-integral}
Let $E$ be a \Hlc space, $\gamma:I\to E$ be a continuous curve on an interval $I\subseteq\Real$ and $a,b\in I$. If there exists an element $z\in E$ such that

\begin{equation*}
	\lambda(z)=\int_a^b \lambda\left(\gamma(t)\right)dt
\end{equation*}

for each $\lambda\in E'$, then $z$ is called the \textit{weak integral of $\gamma$ from $a$ to $b$}\index{weak integral} and we
write

\begin{equation*}
	\int_a^b \gamma(t)dt:=z.
\end{equation*}
Note that, since the space $E'$ separates points on $E$ by the Hahn-Banach Theorem, the weak integral is uniquely determined if it exists. Further, it is known that the weak integral of a curve $\gamma:I\to E$ always exists if $E$ is sequentially 
complete (see, for example \cite{GlNeeb}). Since every locally convex space $F$ can be completed, the weak integral 
$z:=\int_a^b\eta(t)dt$ of a continuous curve $\eta:I\to F$ always exists in the completion $\widetilde{F}$ of $F$. One can show that the
weak integral $w:=\int_a^b\eta(t)dt$ in $F$ exists if and only if $z\in F$, in which case we have $z=w$.
\end{remark}

We will deduce the completeness of $C_\Weights^k(U,E)$ from the next
proposition:

\begin{proposition}\label{prop:der-top-emb-closed-image}
Let $E$ be a \Hlc space and $U\subseteq\Real^n$ be a locally convex subset with dense interior. Let $\Weights$ be a set of weights on 
$U$ such that for each compact subset $K\subseteq U$ there is a weight $f_K\in\Weights$ with $\inf_{x\in K}f_K(x)>0$. Then the map

\begin{equation*}
	\Theta:C_\Weights^k(U,E)\to C_\Weights(U,E)\times C_\Weights^{k-1}(U,E)^n,
		\quad\gamma\mapsto(\gamma,\partderxInd{1}\gamma,\ldots,\partderxInd{n}\gamma)
\end{equation*}

is a linear topological embedding with closed image, for each $k\in\Natural$.
\end{proposition}

\begin{proof}
We know that $\Theta$ is a topological embedding, by Remark \ref{rem:top-on-weighted-Ck-space}, and
it is easy to show that $\Theta$ is linear.

Now, let $(\gamma_a)_{a\in A}$ be a net in $C_\Weights^k(U,E)$ such that $(\gamma_a)_{a\in A}$ converges to $\gamma$ in $C_\Weights(U,E)$ and each of the nets 
$\left(\partderxInd{i}\gamma_a\right)_{a\in A}$ converges to $\gamma^i$ in $C_\Weights^{k-1}(U,E)$, for each $i\in\oneton$. We have to show that 
$(\gamma,\gamma^1,\ldots,\gamma^n)\in im(\Theta)$, that is, $\gamma\in C_\Weights^k(U,E)$ and $\partderxInd{i}\gamma=\gamma^i$. 

For $x\in U^\circ$, $e_i\in\Real^n$ and $0<t\leq1$ with $K_{x,t,i}:=\left\{x+tue_i : u\in\interclcl{0}{1}\right\}\subseteq U^\circ$ we have

\begin{equation*}
	\frac{\gamma(x+te_i)-\gamma(x)}{t} = \lim_{a\in A}\frac{\gamma_a(x+te_i)-\gamma_a(x)}{t}.
\end{equation*}

We apply
the Mean Value Theorem (cf. \cite{GlNeeb}) and obtain

\begin{align*}
	\lim_{a\in A}\frac{\gamma_a(x+te_i)-\gamma_a(x)}{t}
	&=\lim_{a\in A}\frac{1}{t}\int_0^1 d\gamma_a(x+tue_i,te_i)du\\ 
	&=\lim_{a\in A}\frac{1}{t}\int_0^1 td\gamma_a(x+tue_i,e_i)du\\
	&= \lim_{a\in A}\frac{1}{t}\int_0^1\partderxInd{i}\gamma_a(x+tue_i)du.
\end{align*}

By the hypothesis on $\Weights$, the net $\left(\partderxInd{i}\gamma_a\right)_{a\in A}$con\-verges to $\gamma^i$ uniformly on the segment
$K_{x,t,i}$ (see Lemma \ref{lem:k-inclusion-map-continuous}),
hence we can use the fact, that weak integrals and uniform limits can be interchanged (see \cite{GlNeeb})
and get

\begin{equation*}
	\lim_{a\in A}\int_0^1\partderxInd{i}\gamma_a(x+tue_i)du = \int_0^1\gamma^i(x+tue_i)du.
\end{equation*}

The map

\begin{equation*}
	(t,u)\mapsto\gamma^i(x+tue_i)
\end{equation*}

is continuous, whence the parameter-dependent integral

\begin{equation*}
	t\mapsto\int_0^1\gamma^i(x+tue_i)du
\end{equation*}

is continuous
(also at $t=0$, see \cite{GlNeeb}) and we have

\begin{equation*}
	\lim_{t\to0}\frac{\gamma(x+te_i)-\gamma(x)}{t} = \lim_{t\to0}\int_0^1\gamma^i(x+tue_i)du = \int_0^1\gamma^i(x)du = \gamma^i(x).
\end{equation*}

Thus we see that $\gamma\big|_{U^\circ}$ is partially $C^1$ with partial derivatives $\partderxInd{i}\left(\gamma\big|_{U^\circ}\right)=\gamma^i\big|_{U^\circ}$ for each $i\in\oneton$. But $\gamma^i\big|_{U^\circ}$ continuously extends to the map $\gamma^i:U\to E$. Hence $\partderxInd{i}\left(\gamma\big|_{U^\circ}\right)$ continuously extends to

\begin{equation*}
	\partderxInd{i}\gamma:U\to E,\quad x\mapsto\gamma^i(x),
\end{equation*}

whence $\gamma$ is partially $C^1$. Moreover, each of the maps $\partderxInd{i}\gamma=\gamma^i$ is partially $C^{k-1}$, thus $\gamma$ is partially $C^k$.

It remains to show that $\gamma\in C_\Weights^k(U,E)$. To this end, let $f\in\Weights$, $q\in\contsemiE$ and $\alpha\in\Natural_0^n$ with $|\alpha|\leq k$. If 
$|\alpha|=0$, then we have

\begin{equation*}
	\left\|\gamma\right\|_{f,\alpha,q} = \left\|\gamma\right\|_{f,q} < \infty.
\end{equation*}

Otherwise, we write $\alpha=\beta+e_i$ for suitable $\beta\in\Natural_0^n$ and $i\in\oneton$ and obtain

\begin{align*}
	\left\|\gamma\right\|_{f,\alpha,q} 
	&\defeq \sup_{x\in U}f(x)q\left(\partderx{\alpha}\gamma(x)\right)\\
	&= \sup_{x\in U}f(x)q\left(\partderx{\beta}\partderxInd{i}\gamma(x)\right)
	= \left\|\partderxInd{i}\gamma\right\|_{f,\beta,q}<\infty,
\end{align*}

since $\partderxInd{i}\gamma=\gamma^i\in C_\Weights^{k-1}(U,E)$, and the assertion is proven.
\end{proof}

\begin{proposition}\label{prop:weighted-Ck-space-complete}
Let $E$ be a complete \Hlc space and $U\subseteq~\Real^n$ be a locally convex subset with dense interior. If $\Weights$
is a set of weights on $U$ such that for each compact subset $K\subseteq U$ there is a weight $f_K\in\Weights$ with 
$\inf_{x\in K}f_K(x)>0$, then the space $C_\Weights^k(U,E)$ is complete for each $k\in\Natzeroinfty$.
\end{proposition}

\begin{proof}
First we proof the assertion for $k<\infty$ by induction.

\textit{The case $k=0$.} Follows from Proposition \ref{prop:weighted-space-complete}, since every metrizable space
is a $k$-space.

\textit{Induction step.} By the induction hypothesis, the space $C_\Weights(U,E)\times C_\Weights^{k-1}(U,E)^n$ is complete.
From Proposition \ref{prop:der-top-emb-closed-image} we conclude that $C_\Weights^k(U,E)\cong im(\Theta)$ is complete,
since $im(\Theta)$ is closed (hence complete) in $C_\Weights(U,E)\times C_\Weights^{k-1}(U,E)^n$.

Finally, from Remark \ref{rem:top-on-weighted-Ck-space} we deduce that $C_\Weights^\infty(U,E)$ is complete,
since projective limits of complete topological vector spaces are complete (see \cite[5.3]{SchTopVec}).
\end{proof}

\begin{definition}\label{def:part-Ckl-maps}
Let $E$ be a \Hlc space and $k,l\in\Natzeroinfty$.

\begin{itemize}
	\item [(a)] Let $U\subseteq\Real^n$, $V\subseteq\Real^m$ be open subsets. We call a map 
	$\gamma:U\times V\to E$ a \textit{partially $C^{k,l}$-map}\index{map!partially $C^{k,l}$}, if for each $\alpha\in\Natural_0^n$
	 with $|\alpha|\leq k$, $\beta\in\Natural_0^m$ with $|\beta|\leq l$ the partial derivative
	
	\begin{equation*}
		\partderx{\alpha}\partdery{\beta}\gamma(x,y)
	\end{equation*}
	
	exists in $E$ for all  $(x,y)\in U\times V$, and the map
	
	\begin{equation*}
		\partderx{\alpha}\partdery{\beta}\gamma:U\times V\to E.
	\end{equation*}
	
	is continuous.
	\item [(b)] If $U\subseteq\Real^n$, $V\subseteq\Real^m$ are locally convex with dense interior, then we say that a
	continuous map $\gamma:U\times V\to E$ is \textit{partially $C^{k,l}$}\index{map!partially $C^{k,l}$} if 
	$\gamma\big|_{U^\circ\times V^\circ}$ is partially $C^{k,l}$
	and for each $\alpha,\beta$ as in (a) the map
	
	\begin{equation*}
		\partderx{\alpha}\partdery{\beta}\left(\gamma\big|_{U^\circ\times V^\circ}\right):U^\circ\times V^\circ\to E
	\end{equation*}
	
	admits a (unique) continuous extension
	
	\begin{equation*}
		\partderx{\alpha}\partdery{\beta}\gamma:U\times V\to E.
	\end{equation*}	
\end{itemize}
\end{definition}

\begin{remark}\label{rem:C-alpha-beta-eq-partially-Ckl}
A continuous map $\gamma:U\times V\to E$ as in Definition \ref{def:part-Ckl-maps} is partially $C^{k,l}$ if and only
if $\gamma$ is $C^{(\delta,\varepsilon)}$ for all $\delta\in(\Natzeroinfty)^n$ with $|\delta|\leq k$ and $\varepsilon\in(\Natzeroinfty)^m$
with $|\varepsilon|\leq l$. Thus, using the Schwarz' Theorem for $C^\alpha$-mappings (\cite[Prop.69]{Alz}) we see that

\begin{equation*}
	\partderx{\alpha}\partdery{\beta}\gamma(x,y)=\partdery{\beta}\partderx{\alpha}\gamma(x,y)
\end{equation*}

for all $(x,y)\in U\times V$.
Moreover, $\gamma:U\times V\to E$ is partially $C^{k,l}$ if and only if the map

\begin{equation*}
	V\times U\to E,\quad(y,x)\mapsto\gamma(x,y)
\end{equation*}

is partially $C^{l,k}$.

The topology on $C^{k,l}(U\times V,E)$ (the space of all partially $C^{k,l}$-maps $\gamma:U\times V\to E$)
defined in \cite{Alz} coincides with
the locally convex topology defined by the seminorms

\begin{equation*}
	\begin{split}
		\left\|\cdot\right\|_{K,(\alpha,\beta),q}:C^{k,l}(U\times V,E)&\to\zeroinfty,\\
		\left\|\gamma\right\|_{K,(\alpha,\beta),q}&:=\sup_{(x,y)\in K}q\left(\partderx{\alpha}\partdery{\beta}\gamma(x,y)\right)
	\end{split}
\end{equation*}

with $\alpha\in\Natural_0^n$, $\beta\in\Natural_0^m$ such that $|\alpha|\leq k$, $|\beta|\leq l$,
$K\subseteq U\times V$ a compact subset and $q\in\contsemiE$.
\end{remark}

\begin{remark}[\textbf{Product Rule}]\label{rem:part-Ck-part-Ckl-product-rules}
Let $E$ be a \Hlc space, $U\subseteq\Real^n$ and $V\subseteq\Real^m$ be locally convex subsets with dense interior, and
$k,l\in\Natzeroinfty$. We can show the following Product Rules for differentiable functions:

\begin{itemize}
	\item[(i)] If the maps $f:U\to\Real$ and $\gamma:U\to E$ are partially $C^k$, then the map $f\cdot\gamma$ is
	partially $C^k$ and the partial derivatives are given by
	
	\begin{equation*}
		\partderx{\alpha}(f\cdot\gamma)(x)=\sum_{\tau\leq\alpha}\binom{\alpha}{\tau}
		\partderx{\alpha-\tau}f(x)\partderx{\tau}\gamma(x)
	\end{equation*}
	
	for each $x\in U$ and $\alpha\in\Natural_0^n$ with $|\alpha|\leq k$.
	
	\item[(ii)] If the maps $f:U\times V\to\Real$ and $\gamma:U\times V\to E$ are partially $C^{k,l}$, then the map
	$f\cdot\gamma:U\times V\to E$ is partially $C^{k,l}$ with partial derivatives

	\begin{equation*}
		\partderx{\alpha}\partdery{\beta}(f\cdot\gamma)(x,y)
		=\sum_{\tau\leq\alpha}\sum_{\kappa\leq\beta}\binom{\alpha}{\tau}\binom{\beta}{\kappa}
		\partderx{\alpha-\tau}\partdery{\beta-\kappa}f(x,y)\partderx{\tau}\partdery{\kappa}\gamma(x,y)
	\end{equation*} 

	for $\alpha\in\Natural_0^n$ with $|\alpha|\leq k$, $\beta\in\Natural_0^m$ with $|\beta|\leq l$ and $(x,y)\in U\times V$.
\end{itemize}
\end{remark}

\begin{definition}\label{def:weighted-Ckl-function-space}
Let $E$ be a \Hlc space, $U\subseteq\Real^n$, $V\subseteq\Real^m$ be locally convex subsets with dense interior,
and $k,l\in\Natzeroinfty$. Let $\Weights_1$ and $\Weights_2$ be sets of weights on $U$ and $V$, respectively, and
$\Weights:=\Weights_1\otimes\Weights_2$. For $\gamma\in C^{k,l}(U\times V,E)$, $f=f_1\otimes f_2\in\Weights$,
$q\in\contsemiE$ and $\alpha\in\Natural_0^n$ with $|\alpha|\leq k$, $\beta\in\Natural_0^m$ with $|\beta|\leq l$ we define

\begin{equation*}
	\left\|\gamma\right\|_{f,(\alpha,\beta),q}:=\sup_{(x,y)\in U\times V}f_1(x)f_2(y)q\left(\partderx{\alpha}\partdery{\beta}\gamma(x,y)\right)\in\interclcl{0}{\infty}
\end{equation*}

and endow the vector space

\begin{multline*}
		C_\Weights^{k,l}(U\times V,E):=\{\gamma\in C^{k,l}(U\times V,E) : (\forall f\in\Weights)(\forall q\in\contsemiE)\\
		(\forall \alpha\in\Natural_0^n,|\alpha|\leq k)
		(\forall \beta\in\Natural_0^m,|\beta|\leq l)\left\|\gamma\right\|_{f,(\alpha,\beta),q}<\infty\}
\end{multline*}

with the locally convex topology defined by the seminorms

\begin{equation*}
	\left\|\cdot\right\|_{f,(\alpha,\beta),q}:C_\Weights^{k,l}(U\times V,E)\to\zeroinfty.
\end{equation*}
\end{definition}

Using the fact, that metrizable spaces are $k$-spaces, we can rephrase the Exponential Law for spaces
of differentiable functions (which is proven in \cite[Theorem 94]{Alz}) as follows:

\begin{proposition}[\textbf{Exponential Law for spaces of differentiable functions}]\label{cor:part-Ckl-classical-exponential-law}
Let $E$ be a \Hlc space and $U\subseteq\Real^n$, $V\subseteq\Real^m$ be locally convex subsets with dense interior. 
If $\gamma\in C^{k,l}(U\times V,E)$ for some $k,l\in\Natzeroinfty$, then

\begin{equation*}
	\gamma_x:=\gamma(x,\bullet):V\to E,\quad\gamma_x(y):=\gamma(x,y)
\end{equation*}

is partially $C^l$ for each $x\in U$, and

\begin{equation*}
	\gamma^\vee:U\to C^l(V,E),\quad\gamma^\vee(x):=\gamma_x
\end{equation*} 

is partially $C^k$. 

Moreover, the linear map

\begin{equation*}
	\Phi:C^{k,l}(U\times V,E)\to C^k(U,C^l(V,E)),\quad\gamma\mapsto\gamma^\vee
\end{equation*}

is a homeomorphism.
\end{proposition}

Now we obtain the following intermediate result:

\begin{theorem}\label{thm:weighted-Ck-exp-top-emb}
Let $E$ be a \Hlc space and $U\subseteq\Real^n$, $V\subseteq\Real^m$ be locally convex subsets with dense interior. Let $\Weights_1$ and $\Weights_2$ be sets of weights on $U$ and $V$, respectively. We assume that for each compact subset $K\subseteq U$ there is a weight $f_K\in\Weights_1$ such that $\inf_{x\in K}f_K(x)>0$, and likewise for $\Weights_2$. If 
$\gamma\in C_{\Weights_1}^k(U,C_{\Weights_2}^l(V,E))$ for some $k,l\in\Natzeroinfty$, then

\begin{equation*}
	\gamma^\wedge\in C_\Weights^{k,l}(U\times V,E),
\end{equation*}

where $\gamma^\wedge$ is the map

\begin{equation*}
	\gamma^\wedge:U\times V\to E,\quad\gamma^\wedge(x,y):=\gamma(x)(y)
\end{equation*}

and $\Weights=\Weights_1\otimes\Weights_2$.

Moreover, the map

\begin{equation*}
	\Psi:C_{\Weights_1}^k(U,C_{\Weights_2}^l(V,E))\to C_\Weights^{k,l}(U\times V,E),\quad\gamma\mapsto\gamma^\wedge
\end{equation*}

is a topological embedding.
\end{theorem}

\begin{proof}
We know by Lemma \ref{lem:k-inclusion-map-continuous} that the inclusion maps

\begin{equation*}
	i:C_{\Weights_1}^k(U,C^l(V,E))\to C^k(U,C^l(V,E))
\end{equation*}

and

\begin{equation*}
	j:C_{\Weights_2}^l(V,E))\to C^l(V,E)
\end{equation*}

are continuous and linear. Thus, by Lemma \ref{lem:k-superposition-operator-cont}, the map

\begin{equation*}
	C_{\Weights_1}^k(U,j):C_{\Weights_1}^k(U,C_{\Weights_2}^l(V,E))\to C_{\Weights_1}^k(U,C^l(V,E))
\end{equation*}

is continuous and linear. Now, we use the inverse map

\begin{equation*}
	\Phi^{-1}:C^k(U,C^l(V,E))\to C^{k,l}(U\times V,E),\quad\gamma\mapsto\gamma^\wedge
\end{equation*}

of the homeomorphism $\Phi$ from Proposition \ref{cor:part-Ckl-classical-exponential-law} and define the continuous linear map

\begin{equation*}
	\begin{split}
		\Theta:=\Phi^{-1}\circ i\circ C_{\Weights_1}^k(U,j):C_{\Weights_1}^k(U,C_{\Weights_2}^l(V,E))	&\to C^{k,l}(U\times V,E)\\
																																\gamma	&\mapsto\gamma^\wedge.
	\end{split}
\end{equation*}

To show that $\Theta(\gamma)=\gamma^\wedge\in C_\Weights^{k,l}(U\times V,E)$ for 
$\gamma\in C_{\Weights_1}^k(U,C_{\Weights_2}^l(V,E))$, let $q\in\contsemiE$, $f\in\Weights$ (that is $f=f_1\otimes f_2$ for some weights 
$f_1\in\Weights_1$, $f_2\in\Weights_2$) and $\alpha\in\Natural_0^n$ with $|\alpha|\leq k$, $\beta\in\Natural_0^m$
with $|\beta|\leq l$. We have

\begin{align}
	\left\|\gamma^\wedge\right\|_{f,(\alpha,\beta),q}
	&\defeq\sup_{(x,y)\in U\times V}f_1(x)f_2(y)q\left(\partderx{\alpha}\partdery{\beta}\gamma^\wedge(x,y)\right)\nonumber\\
	&=\sup_{x\in U}f_1(x)\sup_{y\in V}f_2(y)q\left(\partdery{\beta}\left(\partderx{\alpha}\gamma^\wedge(x,\bullet)\right)(y)\right)\nonumber\\
	&=\sup_{x\in U}f_1(x)\sup_{y\in V}f_2(y)q\left(\partdery{\beta}\left(\partderx{\alpha}\gamma(x)\right)(y)\right).
	\label{eq:gamma-wedge-in-weighted-k-l-step-one}
\end{align}

We have $\partderx{\alpha}\gamma(x)\in C_{\Weights_2}^l(V,E)$ for each $x\in U$, hence

\begin{equation}\label{eq:gamma-wedge-in-weighted-k-l-step-two}
	\sup_{y\in V}f_2(y)q\left(\partdery{\beta}\left(\partderx{\alpha}\gamma(x)\right)(y)\right) = \left\|\partderx{\alpha}\gamma(x)\right\|_{f_2,\beta,q} <\infty.
\end{equation}

Applying \eqref{eq:gamma-wedge-in-weighted-k-l-step-two} to \eqref{eq:gamma-wedge-in-weighted-k-l-step-one}, we obtain

\begin{equation}\label{eq:gamma-wedge-in-weighted-k-l-step-three}
	\begin{split}
		\left\|\gamma^\wedge\right\|_{f,(\alpha,\beta),q}
		&=\sup_{x\in U}f_1(x)\sup_{y\in V}f_2(y)q\left(\partdery{\beta}\left(\partderx{\alpha}\gamma(x)\right)(y)\right)\\
		&=\sup_{x\in U}f_1(x)\left\|\partderx{\alpha}\gamma(x)\right\|_{f_2,\beta,q}\\
		&=\left\|\gamma\right\|_{f_1,\alpha,\left\|.\right\|_{f_2,\beta,q}} <\infty,
	\end{split}
\end{equation}

whence $\gamma^\wedge\in C_\Weights^{k,l}(U\times V,E)$.

Thus, we can define the map

\begin{equation*}
	\Psi:=\Theta\big|^{C_\Weights^{k,l}(U\times V,E)}:C_{\Weights_1}^k(U,C_{\Weights_2}^l(V,E))\to C_\Weights^{k,l}(U\times V,E),\quad\gamma\mapsto\gamma^\wedge,
\end{equation*}

which is continuous, linear and injective, by construction. Moreover, for each $f_1\in\Weights_1$, $f_2\in\Weights_2$, $q\in\contsemiE$, and $\alpha$, $\beta$ as above, we have

\begin{equation*}
	\left\|\Psi(\gamma)\right\|_{f_1\otimes f_2,(\alpha,\beta),q} \defeq \left\|\gamma^\wedge\right\|_{f_1\otimes f_2,(\alpha,\beta),q} = 
	\left\|\gamma\right\|_{f_1,\alpha,\left\|.\right\|_{f_2,\beta,q}}
\end{equation*}

for all $\gamma\in C_{\Weights_1}^k(U,C_{\Weights_2}^l(V,E))$, by \eqref{eq:gamma-wedge-in-weighted-k-l-step-three}. Thus, by Lemma \ref{lem:cont-lin-inj-then-top-embedding}, the map $\Psi$ is a topological embedding, which completes the proof.
\end{proof}

For the further work, we need spaces of differentiable maps with compact support.

\begin{definition}\label{def:part-Ck-cp-supp}
Let $E$ be a \Hlc space and $U\subseteq\Real^n$,  $V\subseteq\Real^m$ be locally convex subsets with dense interior. For a compact subset $K\subseteq U$ and $k\in\Natzeroinfty$ we define the space

\begin{equation*}
	C_K^k(U,E):=\left\{\gamma\in C^k(U,E) : \supp(\gamma)\subseteq K\right\}
\end{equation*}

and endow it with the locally convex topology defined by the seminorms

\begin{equation*}
	\left\|\cdot\right\|_{\alpha,q}:C_K^k(U,E)\to\zeroinfty,
	\quad\left\|\gamma\right\|_{\alpha,q}:=\sup_{x\in K}q\left(\partderx{\alpha}\gamma(x)\right)
\end{equation*}

with $q\in\contsemiE$ and $\alpha\in\Natural_0^n$ such that $|\alpha|\leq k$.

Similarly, for a compact subset $L\subseteq U\times V$ and $k,l\in\Natzeroinfty$, the space

\begin{equation*}
	C_L^{k,l}(U\times V,E):=\{\gamma\in C^{k,l}(U\times V,E) : \supp(\gamma)\subseteq L\}
\end{equation*}

is endowed with the locally convex topology defined by the seminorms

\begin{equation*}
	\left\|\cdot\right\|_{(\alpha,\beta),q}:C_L^{k,l}(U\times V,E)\to\zeroinfty,
	\quad\left\|\gamma\right\|_{(\alpha,\beta),q}:=\sup_{(x,y)\in L}q\left(\partderx{\alpha}\partdery{\beta}\gamma(x,y)\right)
\end{equation*}

with $q\in\contsemiE$, $\alpha\in\Natural_0^n$ such that $|\alpha|\leq k$ and $\beta\in\Natural_0^m$ such that $|\beta|\leq l$.

Additionally, we define the spaces

\begin{equation*}
\begin{split}
	C_c^k(U,E):&=\left\{\gamma\in C^k(U,E) : \supp(\gamma)\subseteq U\mbox{ is compact}\right\}\\
	&=\bigcup\{C_K^k(U,E) : K\subseteq U\mbox{ is compact}\}.
\end{split}
\end{equation*}

and

\begin{equation*}
\begin{split}
	C_c^{k,l}(U\times V,E):&=\left\{\gamma\in C^{k,l}(U\times V,E) : \supp(\gamma)\subseteq U\times V\mbox{ is compact}\right\}\\
	&=\bigcup\{C_L^{k,l}(U\times V,E) : L\subseteq U\times V\mbox{ is compact}\}.
\end{split}
\end{equation*}
\end{definition}

\begin{lemma}\label{lem:gamma-check-kl-cp-supp}
Let $E$ be a \Hlc space and $U\subseteq\Real^n$, $V\subseteq\Real^m$ be locally convex subsets with dense interior.
Let $K_1\subseteq U$ and $K_2\subseteq V$ be compact subsets. If $\gamma\in C_{K_1\times K_2}^{k,l}(U\times V,E)$, then

\begin{equation*}
	\gamma_x:=\gamma(x,\bullet)\in C_{K_2}^l(V,E)
\end{equation*}

for each $x\in U$, and

\begin{equation*}
	\gamma^\vee\in C_{K_1}^k(U,C_{K_2}^l(V,E))
\end{equation*}

holds for the map

\begin{equation*}
	\gamma^\vee:U\to C_{K_2}^l(V,E),\quad x\mapsto\gamma_x.
\end{equation*}
\end{lemma}

\begin{proof}
We know from Proposition \ref{cor:part-Ckl-classical-exponential-law} that the map $\gamma_x$ is partially $C^l$ for each
$x\in U$. If $y\in V\backslash K_2$, then

\begin{equation*}
	\gamma_x(y)=\gamma(x,y)=0,
\end{equation*}

whence $\supp(\gamma_x)\subseteq K_2$. Thus we have

\begin{equation*}
	\gamma_x\in C_{K_2}^l(V,E).
\end{equation*}

Further, the map

\begin{equation*}
	\gamma^\vee:U\to C^l(V,E),\quad x\mapsto\gamma_x
\end{equation*}

is partially $C^k$ (by Proposition \ref{cor:part-Ckl-classical-exponential-law}). But we have

\begin{equation*}
	\partderx{\alpha}\gamma^\vee(x)=\partderx{\alpha}\gamma_x\in C_{K_2}^l(V,E)
\end{equation*}

for each $x\in U$ and $\alpha\in\Natural_0^n$ with $|\alpha|\leq k$, whence $\gamma^\vee\in C^k(U,C_{K_2}^l(V,E))$,
by Lemma \ref{lem:lcx-part-Ck-in-image}. Moreover, if $x\in U\backslash K_1$ and $y\in V$, then we have

\begin{equation*}
	\gamma^\vee(x)(y)=\gamma(x,y)=0,
\end{equation*}

hence $\gamma^\vee\in C_{K_1}^k(U,C_{K_2}^l(V,E))$, as asserted.
\end{proof}

The next two lemmas can be proven similarly to Lemma \ref{lem:K-supp-in-weighted-incl-cont}

\begin{lemma}\label{lem:Ck-K-supp-in-part-Ck-weighted-incl-cont}
Let $E$ be a \Hlc space, $U\subseteq\Real^n$ be a locally convex subset with dense interior and $K\subseteq U$
be compact. If $\Weights$ is a set of weights on $U$ such that each weight $f\in\Weights$ is bounded on $K$, then

\begin{equation*}
	C_K^k(U,E)\subseteq C_\Weights^k(U,E)
\end{equation*}

and the inclusion map

\begin{equation*}
	i:C_K^k(U,E)\to C_\Weights^k(U,E)
\end{equation*}

is continuous and linear for each $k\in\Natzeroinfty$.
\end{lemma}

\begin{lemma}\label{lem:Ckl-K-supp-in-part-Ckl-weighted-incl-cont}
Let $E$ be a \Hlc space, $U\subseteq\Real^n$ and $V\subseteq\Real^m$ be locally convex subsets with dense interior,
and $k,l\in\Natzeroinfty$. Let $\Weights_1$ and $\Weights_2$ be sets of weights on $U$ and $V$, respectively, consisting
of functions that are bounded on compact sets. Then

\begin{equation*}
	C_K^{k,l}(U\times V,E)\subseteq C_{\Weights_1\otimes\Weights_2}^{k,l}(U\times V,E)
\end{equation*}

for each compact subset $K\subseteq U\times V$, and the inclusion map

\begin{equation*}
	i:C_K^{k,l}(U\times V,E)\to C_{\Weights_1\otimes\Weights_2}^{k,l}(U\times V,E)
\end{equation*}

is continuous and linear.
\end{lemma}

The next proposition will be essential for the proof of the Exponential Law:

\begin{proposition}\label{prop:smooth-compact-supp-dense-in-weighted-Ckl-assertion}
Let $E$ be a \Hlc space, $U\subseteq\Real^n$ and $V\subseteq~\Real^m$ be open subsets, and $k,l\in\Natzeroinfty$. Assume that for a set of weights 
$\Weights_1\subseteq C^k(U,\zeroinfty)$ on $U$ we have:

\begin{itemize}
	\item [(i)]	$\Weights_1$ satisfies the $o$-condition,
	\item [(ii)]	for each $f\in\Weights_1$ and $\alpha\in\Natural_0^n$ with $|\alpha|\leq k$ there exists 
					$g\in\Weights_1$ such that
					
					\begin{equation*}
						\left|\partderx{\alpha}f(x)\right|\leq g(x)
					\end{equation*}
					
					for all $x\in U$,
\end{itemize}

and likewise for a set of weights $\Weights_2\subseteq C^l(V,\zeroinfty)$ on $V$. 
Then the space $C_c^{k,l}(U\times V,E)$ is dense in $C_\Weights^{k,l}(U\times V,E)$, where 
$\Weights=\Weights_1\otimes\Weights_2$.
\end{proposition}

The proof (which can be found after Lemma \ref{lem:estimate-der-rho}) varies the proof of  the density
of $C_c^\infty(U,\Real)$ in the space $C_\Weights^k(U,\Real)$ by H.G. Garnir,
M. De Wilde and J. Schmets in \cite{Garnir-DeWilde-SchmetsIII}. In addition, several auxiliary results will be necessary. 
We introduce some useful constructions.

\begin{lemma}\label{lem:useful-weighted-Ckl-maps}
Let $E$ be a \Hlc space, $U\subseteq\Real^n$, $V\subseteq\Real^m$ be locally convex subsets with dense interior. Let 
$\Weights_1$ and $\Weights_2$ be sets of weights on $U$ and $V$, respectively. For a $C^{k,l}$-map 
$\gamma:U\times V\to E$ (with $k,l\in\Natzeroinfty$) we define the maps

\begin{equation*}
	\gamma_{\beta,y,f_2}:=f_2(y)\partdery{\beta}\gamma(\bullet,y):U\to E
\end{equation*}

for some weight $f_2\in\Weights_2$, $\beta\in\Natural_0^m$ with $|\beta|\leq l$ and $y\in V$, and

\begin{equation*}
	\gamma_{\alpha,x,f_1}:=f_1(x)\partderx{\alpha}\gamma(x,\bullet):V\to E
\end{equation*}

for some weight $f_1\in\Weights_1$, $\alpha\in\Natural_0^n$ with $|\alpha|\leq k$ and $x\in U$. 

If $\gamma\in C_{\Weights_1\otimes\Weights_2}^{k,l}(U\times V,E)$, then $\gamma_{\beta,y,f_2}\in C_{\Weights_1}^k(U,E)$ and 
$\gamma_{\alpha,x,f_1}\in C_{\Weights_2}^l(V,E)$.
\end{lemma}

\begin{proof}
We prove the assertion for the map $\gamma_{\beta,y,f_2}$, the proof for $\gamma_{\alpha,x,f_1}$ will be similar. , 
From \cite[Lemma 28]{Alz} follows that the map $\partdery{\beta}\gamma(\bullet,y)$ is partially $C^k$, whence 
also the map 
$f_2(y)\partdery{\beta}\gamma(\bullet,y)=\gamma_{\beta,y,f_2}$ is partially $C^k$. 
Now, if $f_1\in\Weights_1$, $q\in\contsemiE$, and $\alpha\in\Natural_0^n$ with $|\alpha|\leq k$, then we have

\begin{align*}
	\left\|\gamma_{\beta,y,f_2}\right\|_{f_1,\alpha,q}
	&\defeq \sup_{x\in U}f_1(x)q\left(\partderx{\alpha}\gamma_{\beta,y,f_2}(x)\right)\\
	&=\sup_{x\in U}f_1(x)q\left(\partderx{\alpha}\left(f_2(y)\partdery{\beta}\gamma(\bullet,y)\right)(x)\right)\\
	&=\sup_{x\in U}f_1(x)f_2(y)q\left(\partderx{\alpha}\partdery{\beta}\gamma(x,y)\right)
	   =\left\|\gamma\right\|_{f_1\otimes f_2,(\alpha,\beta),q}<\infty.
\end{align*}

Hence $\gamma_{\beta,y,f_2}\in C_{\Weights_1}^k(U,E)$.
\end{proof}

The next result will be very helpful:

\begin{lemma}\label{lem:f-o-one-then-to-zero}
Let $f:U\to\Real$ be a map on an open subset $U\subseteq\Real^n$. If $f=o(1)$, then $f(x)\to0$ as $\left\|x\right\|_\infty\to\infty$ or
$x\to\overline{x}\in\partial U$.
\end{lemma}

\begin{proof}
First, let $(x_m)_{m\in\Natural}$ be a sequence in $U$ such that $\left\|x_m\right\|_\infty\to\infty$ as $m\to\infty$, and let 
$\varepsilon>0$. Since $f=o(1)$, there is a compact subset $K_\varepsilon\subseteq U$ such that $\left|f(x)\right|\leq\varepsilon$ for each
$x\in U\backslash K_\varepsilon$. Further, there exists $N\in\Natural$ such that for each $m\geq N$ we have 
$\left\|x_m\right\|_\infty>\max_{x\in K_\varepsilon}\left\|x\right\|_\infty$, that is $x_m\in U\backslash K_\varepsilon$. Thus
$\left|f(x_m)\right|\leq\varepsilon$ for all $m\geq N$, as required.

Now let $(x_m)_{m\in\Natural}$ be a sequence in $U$ which converges to some $\overline{x}\in\partial U$, $\varepsilon>0$ and
$K_\varepsilon$ as above. The set $\Real^n\backslash K_\varepsilon$ is open in $\Real^n$ and 
$\overline{x}\in\Real^n\backslash K_\varepsilon$, thus there exists $N\in\Natural$ such that for all $m\geq N$ we have
$x_m\in\Real^n\backslash K_\varepsilon$ (more precisely, we have $x_m\in U\backslash K_\varepsilon$), whence 
$f(x_m)\leq\varepsilon$ for all $m\geq N$.
\end{proof}

Let us consider some cases, in which the products of weights and weighted maps
tend to zero.

\begin{remark}\label{rem:o-condition-then-derivatives-to-zero}
We recall from Definition \ref{def:little-o} that if the set of weights $\Weights$ on an open subset $U\subseteq\Real^n$ satisfies 
the $o$-condition, then for each weight $f\in\Weights$ there is a weight $g\in\Weights$ such that $f=o(g)$. (We can always assume
that $g(x)=0$ only if $f(x)=0$ and set $\frac{f(x)}{g(x)}:=0$ in this case.)  Therefore, we have $\frac{f}{g}=o(1)$, which yields that

\begin{equation*}
	\frac{f(x)}{g(x)}\to0
\end{equation*}

as $\left\|x\right\|_\infty\to\infty$ or $x\to\overline{x}\in\partial U$ (by Lemma \ref{lem:f-o-one-then-to-zero}). Thus for each 
$\gamma\in C_\Weights^k(U,E)$ and $\alpha\in\Natural_0^n$ with $|\alpha|\leq k$ we have

\begin{equation}\label{eq:derivatives-to-zero}
	f(x)\partderx{\alpha}\gamma(x)\to 0
\end{equation}

as $\left\|x\right\|_\infty\to\infty$ or $x\to\overline{x}\in\partial U$. In fact, if $q\in\contsemiE$ and $f=o(g)$, then

\begin{equation*}
	f(x)q\left(\partderx{\alpha}\gamma(x)\right) = \frac{f(x)}{g(x)}g(x)q\left(\partderx{\alpha}\gamma(x)\right) \leq \frac{f(x)}{g(x)}\left\|\gamma\right\|_{g,\alpha,q}\to 0.
\end{equation*}

If, additionally, $\Weights\subseteq C^k(U,\interclop{0}{\infty})$ and for each weight $f\in\Weights$ and 
$\alpha\in\Natural_0^n$ with $|\alpha|\leq k$ there exists a weight $h_{f,\alpha}\in\Weights$ such that

\begin{equation*}
	\left|\partderx{\alpha}f(x)\right|\leq h_{f,\alpha}(x)
\end{equation*}

for all $x\in U$, then for all $\beta,\tau\in\Natural_0^n$ with $|\beta|+|\tau|\leq k$ we have

\begin{equation*}
	\partderx{\beta}\left(f\cdot\partderx{\tau}\gamma\right)(x)\to 0
\end{equation*}

as $\left\|x\right\|_\infty\to\infty$ or $x\to\overline{x}\in\partial U$. To see this, we use the Product Rule 
(see Remark \ref{rem:part-Ck-part-Ckl-product-rules} (i)) and obtain

\begin{align*}
	q\left(\partderx{\beta}\left(f\cdot\partderx{\tau}\gamma\right)(x)\right)
	&= q\left(\sum_{\kappa\leq\beta}\binom{\beta}{\kappa}\partderx{\beta-\kappa}f(x)\partderx{\kappa+\tau}\gamma(x)\right)\\
	&\leq\sum_{\kappa\leq\beta}\binom{\beta}{\kappa}\left|\partderx{\beta-\kappa}f(x)\right|q\left(\partderx{\kappa+\tau}\gamma(x)\right)\\
	&\leq\sum_{\kappa\leq\beta}\binom{\beta}{\kappa}h_{f,\beta-\kappa}(x)q\left(\partderx{\kappa+\tau}\gamma(x)\right)\to 0,
\end{align*}

by \eqref{eq:derivatives-to-zero}.
\end{remark}

\begin{lemma}\label{lem:o-condition-then-Ckl-derivatives-to-zero}
Let $E$ be a \Hlc space, $U\subseteq\Real^n$ and $V\subseteq\Real^m$ be locally convex subsets with dense interior,
and $k,l\in\Natzeroinfty$. Let $\Weights_1$ and $\Weights_2$ be sets of weights on $U$ and $V$, respectively, and
$\alpha\in\Natural_0^n$ with $|\alpha|\leq k$, $\beta\in\Natural_0^m$ with $|\beta|\leq l$.
\begin{enumerate}
	\item[(i)] If $U$ is open and $\Weights_1$ satisfies the $o$-condition, then for each
	$f_1\in\Weights_1$, $f_2\in\Weights_2$ and $\gamma\in C_{\Weights_1\otimes\Weights_2}^{k,l}(U\times V,E)$
	we have
	
	\begin{equation*}
		f_1(x)f_2(y)\partderx{\alpha}\partdery{\beta}\gamma(x,y)\to 0
	\end{equation*}
	
	uniformly in $y$, as $\left\|x\right\|_\infty\to\infty$ or $x\to\overline{x}\in\partial U$.
	\item[(ii)] If $V$ is open and $\Weights_2$ satisfies the $o$-condition, then for each
	$f_1\in\Weights_1$, $f_2\in\Weights_2$ and $\gamma\in C_{\Weights_1\otimes\Weights_2}^{k,l}(U\times V,E)$
	we have
	
	\begin{equation*}
		f_1(x)f_2(y)\partderx{\alpha}\partdery{\beta}\gamma(x,y)\to 0
	\end{equation*}
	
	uniformly in $x$, as $\left\|y\right\|_\infty\to\infty$ or $y\to\overline{y}\in\partial V$.
\end{enumerate}
\end{lemma}

\begin{proof}
Assume that $U$ is open and $\Weights_1$ satisfies the $o$-condition. Then there exists a weight
$g_1\in\Weights_1$ such that $f_1=o(g_1)$. For a seminorm $q\in\contsemiE$, $x\in U$ and $y\in V$ we have

\begin{align*}
	f_1(x)f_2(y)q\left(\partderx{\alpha}\partdery{\beta}\gamma(x,y)\right)
	&=\frac{f_1(x)}{g_1(x)}g_1(x)f_2(y)q\left(\partderx{\alpha}\partdery{\beta}\gamma(x,y)\right)\\
	&\leq\frac{f_1(x)}{g_1(x)}\left\|\gamma\right\|_{g_1\otimes f_2,(\alpha,\beta),q}\to 0
\end{align*}

as $\left\|x\right\|_\infty\to\infty$ or $x\to\overline{x}\in\partial U$, see Remark \ref{rem:o-condition-then-derivatives-to-zero}.
Thus $(i)$ holds. To prove $(ii)$, we proceed similarly.
\end{proof}

We will also need the following facts concerning differentiable extensions of mappings:

\begin{lemma}\label{lem:C1-extension-on-interval}
Let $E$ be a \Hlc space and $\gamma:\interclcl{a}{b}\to E$ be a continuous curve. If $\gamma\big|_{\interopcl{a}{b}}$ is $C^1$ and the derivative 
$\left(\gamma\big|_{\interopcl{a}{b}}\right)'$ admits a continuous extension $\eta:\interclcl{a}{b}\to E$, 
then $\gamma$ is $C^1$ and $\gamma'= \eta$.
\end{lemma}

\begin{proof}
For each $t\in\interopcl{a}{b}$ the curve $\gamma\big|_{\interopcl{t}{b}}$ is $C^1$, thus

\begin{equation}\label{eq:gamma-eq-integral}
	\gamma(t) = \gamma(b)-\int_t^b \gamma'(s)ds = \gamma(b)-\int_t^b \eta(s)ds,
\end{equation}

by the first part of the Fundamental Theorem of Calculus 
(see \cite{GlNeeb}). By Remark \ref{rem:facts-about-weak-integral}, 
the weak integral $\int_a^b \eta(s)ds$ exists in the completion $\widetilde{E}$ of $E$, thus we can define the continuous curve

\begin{equation*}
	\xi:\interclcl{a}{b}\to \widetilde{E}\quad\xi(t):=\gamma(b)-\int_t^b \eta(s)ds.
\end{equation*}

Since $\gamma$ is continuous and $\gamma\big|_{\interopcl{a}{b}} = \xi\big|_{\interopcl{a}{b}}$ (by \eqref{eq:gamma-eq-integral}), we conclude that $\gamma = \xi$. Hence $\xi(\interclcl{a}{b})\subseteq E$, that is $\int_a^b \eta(s)ds$ exists in $E$. The second part of the Fundamental Theorem of Calculus 
(can also be found in \cite{GlNeeb}) yields that $\gamma$ is $C^1$ and $\gamma' = \eta$, as required.
\end{proof}

\begin{proposition}\label{prop:partially-Ck-extensions-partially-Ck}
Let $E$ be a \Hlc space and $\gamma:U\to E$ be a partially $C^k$-map on an open subset $U\subseteq\Real^n$ for $k\in\Natural\cup\left\{\infty\right\}$. Assume that 
$\gamma$ admits the continuous extension

\begin{equation}\label{eq:cont-extension-of-gamma}
	\eta:\Real^n\to E,\quad\eta(x):=\left\{
															\begin{array}{ll}
																\gamma(x)	& x\in U\\
																0	& x\notin U\\
															\end{array}
														\right.
\end{equation}

and the partial derivatives $\partderx{\alpha}\gamma$ admit the continuous extensions

\begin{equation}\label{eq:cont-extension-of-partder-gamma}
	\eta_{\alpha}:\Real^n\to E,\quad\eta_{\alpha}(x):=\left\{
																						\begin{array}{ll}
																							\partderx{\alpha}\gamma(x)	& x\in U\\
																							0	& x\notin U\\
																						\end{array}
																					\right.
\end{equation}

for each $\alpha\in\Natural_0^n$ with $|\alpha|\leq k$. Then the map $\eta$ is partially $C^k$ with the partial 
derivatives $\partderx{\alpha}\eta = \eta_{\alpha}$ for all $\alpha$.
\end{proposition}

\begin{proof}

We may assume that $k$ is finite and prove the assertion by induction.

\textit{The case $k=1$.}\\ \underline{Step $1$: $n=1$.} Assume that $\gamma:U\to E$ is a $C^1$-map on an open subset $U\subseteq\Real$ such that there are continuous extensions

\begin{equation*}
	\eta:\Real\to E,\quad\eta(t):=\left\{\begin{array}{ll}\gamma(t)	& t\in U\\
																					0	& t\notin U\\
													\end{array}\right.	
\end{equation*}

of $\gamma$ and

\begin{equation*}
	\eta_1:\Real\to E,\quad\eta_1(t):=\left\{\begin{array}{ll}\gamma'(t)	& t\in U\\
																							0	& t\notin U\\
													\end{array}\right.	
\end{equation*}

of the derivative $\gamma'$. For $t\in U\cup\left(\Real\backslash\overline{U}\right)$ we obviously have

\begin{equation*}
	\eta'(t)=\eta_1(t).
\end{equation*}

Now, we want to show that the right derivative 
$\eta_+'(t):=\lim_{s\to 0_+}\frac{\eta(t+s)-\eta(t)}{s}$ exists for each $t\in\partial U$ and we have

\begin{equation}\label{eq:right-derivative-of-eta-is-zero}
	\eta_+'(t)=\eta_1(t)=0.
\end{equation}

The proof for the left derivative $\eta_-'(t):=\lim_{s\to 0_-}\frac{\eta(t+s)-\eta(t)}{s}$ is similar. 

Let $\left(t_m\right)_{m\in\Natural}\subseteq\left.\right]t,\infty\left[\right.$ be a sequence such that $t_m\to t$ as $m\to\infty$. To prove 
\eqref{eq:right-derivative-of-eta-is-zero}, it suffices to show that

\begin{equation*}
	\lim_{m\to\infty}q\left(\frac{\eta(t_m)-\eta(t)}{t_m-t}\right)=0
\end{equation*}

for each seminorm $q\in\contsemiE$.

If $t_m\notin U$ for all $m\in\Natural$, then the assertion is clear. Otherwise, for $m\in\Natural$ with $t_m\in U$ we define

\begin{equation*}
	s_m:=\min\left\{s\in\left[t,t_m\left[\right.\right. : \left.\left.\right]s,t_m\right]\subseteq U\right\},
\end{equation*}

and obtain

\begin{align*}
	q\left(\frac{\eta(t_m)-\eta(t)}{t_m-t}\right)	
	&= q\left(\frac{\eta(t_m)}{t_m-t}\right) = \frac{1}{t_m-t}q\left(\eta(t_m)\right)\\
	&\leq \frac{1}{t_m-s_m}q\left(\eta(t_m)\right) = \frac{1}{t_m-s_m}q\left(\eta(t_m)-\eta(s_m)\right).
\end{align*}

In the last step we used the fact that $s_m\notin U$, whence $\eta(s_m)=0$. Since $\left.\left.\right]s_m,t_m\right]\subseteq U$, we have 

\begin{equation*}
	\eta\big|_{\left.\left.\right]s_m,t_m\right]}=\gamma\big|_{\left.\left.\right]s_m,t_m\right]},
\end{equation*}

thus $\eta\big|_{\left.\left.\right]s_m,t_m\right]}$ is $C^1$ and $\left(\eta\big|_{\left.\left.\right]s_m,t_m\right]}\right)' = \left(\gamma\big|_{\left.\left.\right]s_m,t_m\right]}\right)'$ admits the continuous extension $\eta_1\big|_{\left[s_m,t_m\right]}$. Lemma 
\ref{lem:C1-extension-on-interval} yields that $\eta\big|_{\left[s_m,t_m\right]}$ is a $C^1$-curve with the derivative 
$\left(\eta\big|_{\left[s_m,t_m\right]}\right)' = \eta_1\big|_{\left[s_m,t_m\right]}$. Using the first part of the 
Fundamental Theorem of Calculus 
(see \cite{GlNeeb}), we obtain

\begin{equation*}
	\frac{1}{t_m-s_m}q\left(\eta(t_m)-\eta(s_m)\right) = \frac{1}{t_m-s_m}q\left(\int_{s_m}^{t_m}\eta_1(u)du\right).
\end{equation*}

Further,
we have

\begin{equation*}
	\frac{1}{t_m-s_m}q\left(\int_{s_m}^{t_m}\eta_1(u)du\right) 
	\leq \frac{t_m-s_m}{t_m-s_m}\min_{u\in\interclcl{s_m}{t_m}}q(\eta_1(u))
	=q\left(\eta_1(S_m)\right),
\end{equation*}

for a suitable $S_m\in\left[s_m,t_m\right]\subseteq\left[t,t_m\right]$ (see again \cite{GlNeeb}). 
But for $m\to\infty$ we have $S_m\to t$, whence

\begin{equation*}
	\lim_{m\to\infty}q\left(\eta_1(S_m)\right) = q\left(\eta_1(t)\right)=0,
\end{equation*}

as required. Consequently, the map $\eta$ is $C^1$ and $\eta'=\eta_1$.

\underline{Step $2$: $n\in\Natural$.} Now, let $U\subseteq\Real^n$ be an open subset and $\gamma:U\to E$ be partially $C^1$ such that $\gamma$ admits the continuous extension

\begin{equation*}
	\eta:\Real^n\to E,\quad\eta(x):=\left\{\begin{array}{ll}\gamma(x)	& x\in U\\
																						0	& x\notin U\\
													\end{array}\right.
\end{equation*}

and for each $i\in\oneton$ the partial derivative $\partderxInd{i}\gamma$ admits the continuous extension

\begin{equation*}
	\eta_{e_i}:\Real^n\to E,\quad\eta_{e_i}(x):=\left\{\begin{array}{ll}\partderxInd{i}\gamma(x)	& x\in U\\
																															0	& x\notin U.\\
																\end{array}\right.
\end{equation*}

For $x\in\Real^n$ and $i\in\oneton$ we define the continuous maps

\begin{align*}
	\xi:\Real\to E,&\quad\xi(t):=\eta(x+te_i)
	\intertext{and}
	\xi_i:\Real\to E,&\quad\xi_i(t):=\eta_{e_i}(x+te_i).
\end{align*}

Further, we define the set

\begin{equation*}
	V:=\left\{t\in\Real : x+te_i\in U\right\},
\end{equation*}

which is open in $\Real$. For $t\in V$ we obtain

\begin{equation*}
	\xi(t) \defeq \eta(x+te_i) = \gamma(x+te_i),
\end{equation*}

thus $\xi\big|_V$ is $C^1$ with the derivative

\begin{equation*}
	\left(\xi\big|_V\right)'(t)=\partderxInd{i}\gamma(x+te_i) = \eta_{e_i}(x+te_i) = \xi_i(t).
\end{equation*}

If $t\notin V$, then we have

\begin{align*}
	\xi(t) &\defeq \eta(x+te_i)=0
	\intertext{and}
	\xi_i(t) &\defeq \eta_{e_i}(x+te_i)=0.
\end{align*}

Thus, using the result of the first step of the proof, we know that $\xi$ is a $C^1$-map with the derivative $\xi'=\xi_i$. Hence we have

\begin{equation*}
	\lim_{t\to 0}\frac{\eta(x+te_i)-\eta(x)}{t} = \lim_{t\to 0}\frac{\xi(t)-\xi(0)}{t} = \xi'(0) = \xi_i(0) = \eta_{e_i}(x).
\end{equation*}

Therefore, the map $\eta$ is partially $C^1$ with $\partderxInd{i}\eta=\eta_{e_i}$ for each $i\in\oneton$.

\textit{Induction step.} We assume that $\gamma:U\to E$ is a partially $C^k$ map on an open subset $U\subseteq\Real^n$ for $n\in\Natural$ and $k\geq 2$, and that the continuous extensions $\eta$ and $\eta_{\alpha}$ defined in 
\eqref{eq:cont-extension-of-gamma} and \eqref{eq:cont-extension-of-partder-gamma} exist for all 
$\alpha\in\Natural_0^n$ with $|\alpha|\leq k$. Since $\gamma$ is partially $C^1$, the map $\eta$ is partially $C^1$ with the partial derivatives 
$\partderxInd{i}\eta=\eta_{e_i}$ for all $i\in\oneton$. By assumption, there exist continuous maps

\begin{equation*}
	\eta_{\beta+e_i}:\Real^n\to E,\quad\eta_{\beta+e_i}(x):=\left\{\begin{array}{ll}\partderx{\beta+e_i}\gamma(x)=\partderx{\beta}\partderxInd{i}\gamma(x)	& x\in U\\
														0	& x\notin U\\
																		\end{array}\right.	
\end{equation*}

for each $i\in\oneton$ and $\beta\in\Natural_0^n$ with $|\beta|\leq k-1$ (we may assume that $|\alpha|\neq 0$). Since each
$\partderxInd{i}\gamma$ is partially $C^{k-1}$ (by definition) with the continuous extension $\eta_{e_i}$ and the extensions $\eta_{\beta+e_i}$ of its partial derivatives, we conclude that $\eta_{e_i}$ is partially $C^{k-1}$ with $\partderx{\beta}\eta_{e_i} = \eta_{\beta+e_i}$, by induction. Thus, $\eta$ is partially $C^k$, since $\eta$ is partially $C^1$ and each 
$\eta_{e_i} = \partderxInd{i}\eta$ is partially $C^{k-1}$, and we have

\begin{equation*}
	\partderx{\alpha}\eta = \partderx{\beta+e_i}\eta = \partderx{\beta}\partderxInd{i}\eta = \partderx{\beta}\eta_{e_i} = \eta_{\beta+e_i} = \eta_{\alpha}
\end{equation*}

for each $\alpha\in\Natural_0^n$ with $0<|\alpha|\leq k$.
\end{proof}

\begin{lemma}\label{lem:taylor-series}
Let $E$ be a \Hlc space, $U\subsetneqq\Real^n$ be an open subset and $k\in\Natural$. Let $\gamma:\Real^n\to E$ be a partially 
$C^k$-map such that for all $\alpha\in\Natural_0^n$ with $|\alpha|\leq k$ and $y\notin U$ we have 
$\gamma(y)=0=\partderx{\alpha}\gamma(y)$. For each seminorm $q\in\contsemiE$, $x\in U$ and $\overline{x}\in\partial U$ there exists $\xi\in\interclcl{0}{1}$ such that

\begin{equation*}
	q(\gamma(x)) \leq \frac{1}{(k-1)!}\left\|x-\overline{x}\right\|_\infty^k 
				\sum_{|\alpha|=k} q\left(\partderx{\alpha}\gamma(\overline{x}+\xi(x-\overline{x}))\right).
\end{equation*}
\end{lemma}

\begin{proof}
We define the curve

\begin{equation*}
	h:\interclcl{0}{1}\to E,\quad h(t):=\gamma(\overline{x}+t(x-\overline{x})),
\end{equation*}

which is $C^k$ with derivatives

\begin{equation}\label{eq:derivatives-of-h}
	h^{(j)}(t)=\sum_{i_1,\ldots,i_j=1}^n (x_{i_1}-\overline{x}_{i_1})\cdots(x_{i_j}-\overline{x}_{i_j})
										\frac{\partial^j}{\partial x_{i_j}\cdots\partial x_{i_1}}\gamma(\overline{x}+t(x-\overline{x}))
\end{equation}

for each $j\leq k$. (We have $h^{(j)}(t)=d^{(j)}\gamma(\overline{x}+t(x-\overline{x}),x-\overline{x},\cdots,x-\overline{x})$.) Thus we have

\begin{equation}\label{eq:values-of-h}
	h(1)=\gamma(x)\quad\mbox{and}\quad h^{(j)}(0)=0
\end{equation}

for all $j\in\Natural_0$ such that $j\leq k$. By Hahn-Banach, for each $z\in E$ there is $\lambda\in E'$ such that $\lambda(z)=q(z)$ and 
$|\lambda(z')|\leq q(z')$ for all $z'\in E$. Therefore, with $z:=h(1)$ we have

\begin{equation*}
	q(h(1))= (\lambda\circ h)(1) = \sum_{j=0}^{k-1} \frac{1}{j!}(\lambda\circ h)^{(j)}(0) + \int_0^1 \frac{(1-s)^{k-1}}{(k-1)!}(\lambda\circ h)^{(k)}(s)ds
\end{equation*}

(cf. Taylor's Theorem in \cite[10.15]{WalAnalysis1}). Applying Lemma 
\ref{lem:lcx-partCk-lin-operator}, we get

\begin{equation*}
	(\lambda\circ h)^{(j)}(0) = \lambda(h^{(j)}(0)) = 0,
\end{equation*}

by \eqref{eq:values-of-h}, thus we have

\begin{equation*}
	q(h(1)) =  \int_0^1 \frac{(1-s)^{k-1}}{(k-1)!}\lambda (h^{(k)}(s))ds.
\end{equation*}

The Mean Value Theorem yields $\xi\in\interclcl{0}{1}$ such that

\begin{equation*}
	\int_0^1 \frac{(1-s)^{k-1}}{(k-1)!}\lambda (h^{(k)}(s))ds = \frac{(1-\xi)^{k-1}}{(k-1)!}\lambda (h^{(k)}(\xi)) \leq \frac{1}{(k-1)!}\lambda (h^{(k)}(\xi)),
\end{equation*}

whence we have

\begin{equation*}
	q(h(1)) \leq \frac{1}{(k-1)!}\lambda (h^{(k)}(\xi)) \leq \frac{1}{(k-1)!} q(h^{(k)}(\xi)).
\end{equation*}

Finally, we apply \eqref{eq:derivatives-of-h} and get

\begin{align*}
	&q(\gamma(x))=q(h(1))\\ 
	&\leq \frac{1}{(k-1)!}\sum_{i_1,\ldots,i_k=1}^n |x_{i_1}-\overline{x}_{i_1}|\cdots|x_{i_k}-\overline{x}_{i_k}|
											q\left(\frac{\partial^k}{\partial x_{i_k}\cdots\partial x_{i_1}}\gamma(\overline{x}+\xi(x-\overline{x}))\right)\\
	&\leq \frac{1}{(k-1)!}\left\|x-\overline{x}\right\|_\infty^k\sum_{i_1,\ldots,i_k=1}^n 
											q\left(\frac{\partial^k}{\partial x_{i_k}\cdots\partial x_{i_1}}\gamma(\overline{x}+\xi(x-\overline{x}))\right)\\
	&=\frac{1}{(k-1)!}\left\|x-\overline{x}\right\|_\infty^k 
				\sum_{|\alpha|=k} q\left(\partderx{\alpha}\gamma(\overline{x}+\xi(x-\overline{x}))\right),
\end{align*}

as asserted.
\end{proof}

Further we will work with some so-called mollifiers.

\begin{remark}\label{rem:mollifier}
We recall that a \textit{mollifier} is a function

\begin{equation*}
	\rho_\varepsilon:\Real^n\to\Real,\quad\rho_\varepsilon(x):=\frac{1}{\varepsilon^n}\rho\left(\frac{x}{\varepsilon}\right),
\end{equation*}

where $\varepsilon>0$ and

\begin{enumerate}
	\item [(i)] $\rho\in C^\infty(\Real^n,\Real)$,
	\item [(ii)] $\rho(x)\geq0$ for all $x\in\Real^n$,
	\item [(iii)] $\supp(\rho)\subseteq\overline{B_1(0)}$,
	\item [(iv)] $\int_{\Real^n}\rho(x)d\lambda_n(x)=1$.
\end{enumerate}

It is known that (i), (ii) and (iv) also hold for $\rho_\varepsilon$ and 

\begin{equation*}
	\supp(\rho_\varepsilon)\subseteq\overline{B_\varepsilon(0)}.
\end{equation*}
\end{remark}

\begin{lemma}\label{lem:estimate-der-rho}
Let $B\subseteq\Real^n$ be a Borel set and $\varepsilon>0$. For the map

\begin{equation*}
	g:\Real^n\to\interclcl{0}{1},\quad g(x):=\left(\mathds{1}_B *\rho_\varepsilon\right)(x),
\end{equation*}

where $*$ means the convolution, we have

\begin{equation*}
	\left\|\partderx{\alpha}g\right\|_\infty\leq\frac{1}{\varepsilon^{|\alpha|}}\left\|\partderx{\alpha}\rho\right\|_{L^1}
\end{equation*}

for each $\alpha\in\Natural_0^n$.
\end{lemma}

\begin{proof}
We have (cf. \cite[7.22]{WalAnalysis2})

\begin{equation*}
	\partderx{\alpha}g = \partderx{\alpha}\left(\mathds{1}_B *\rho_\varepsilon\right) = \mathds{1}_B * \partderx{\alpha}\rho_\varepsilon,
\end{equation*}

and for $x\in\Real^n$ we have

\begin{align*}
	\left|\left(\mathds{1}_B * \partderx{\alpha}\rho_\varepsilon\right)(x)\right|	
	&= \left|\int_{\Real^n}\mathds{1}_B(x-y)\partderx{\alpha}\rho_\varepsilon(y)d\lambda_n(y)\right|\\
	&\leq\int_{\Real^n}\underbrace{\left|\mathds{1}_B(x-y)\right|}_{\leq 1}\left|\partderx{\alpha}\rho_\varepsilon(y)\right|d\lambda_n(y)\\
	&\leq\int_{\Real^n}\left|\partderx{\alpha}\rho_\varepsilon(y)\right|d\lambda_n(y).
\end{align*}

Using the definition of $\rho_\varepsilon$, we obtain

\begin{equation*}
	\partderx{\alpha}\rho_\varepsilon(y) = \frac{1}{\varepsilon^n}\frac{1}{\varepsilon^{|\alpha|}}\partderx{\alpha}\rho\left(\frac{y}{\varepsilon}\right),
\end{equation*}

whence

\begin{equation*}
	\int_{\Real^n}\left|\partderx{\alpha}\rho_\varepsilon(y)\right|d\lambda_n(y) 
	= \frac{1}{\varepsilon^n}\frac{1}{\varepsilon^{|\alpha|}}\int_{\Real^n}\left|\partderx{\alpha}\rho\left(\frac{y}{\varepsilon}\right)\right|d\lambda_n(y).
\end{equation*}

After the substitution $u=\frac{y}{\varepsilon}$ we get

\begin{equation*}
	\frac{1}{\varepsilon^n}\frac{1}{\varepsilon^{|\alpha|}}\int_{\Real^n}\left|\partderx{\alpha}\rho\left(\frac{y}{\varepsilon}\right)\right|d\lambda_n(y)
	= \frac{1}{\varepsilon^{|\alpha|}}\int_{\Real^n}\left|\partderx{\alpha}\rho(u)\right|d\lambda_n(u)
	\leq\frac{1}{\varepsilon^{|\alpha|}}\left\|\partderx{\alpha}\rho\right\|_{L^1}.
\end{equation*}

Therefore, we have

\begin{equation*}
	\left\|\partderx{\alpha}g\right\|_\infty = \sup_{x\in\Real^n}\left|\left(\mathds{1}_B * \partderx{\alpha}\rho_\varepsilon\right)(x)\right|
	\leq\frac{1}{\varepsilon^{|\alpha|}}\left\|\partderx{\alpha}\rho\right\|_{L^1},
\end{equation*}

as asserted.
\end{proof}

Now, we have all the tools together to prove Proposition 
\ref{prop:smooth-compact-supp-dense-in-weighted-Ckl-assertion}.

\begin{proof}[Proof of Proposition \ref{prop:smooth-compact-supp-dense-in-weighted-Ckl-assertion}]
\underline{Step $1$.} First we show that the set

\begin{equation*}
	\begin{split}
		M:=\left\{\eta\in C_\Weights^{k,l}(U\times V,E) 	: (\exists r>0)\right.	&(\forall (x,y)\in U\times V)\\
																		&\left.\left\|(x,y)\right\|_\infty>r\Rightarrow\eta(x,y)=0\right\}
	\end{split}
\end{equation*}

is dense in $C_\Weights^{k,l}(U\times V,E)$. 
To this end, let 
$\gamma\in C_\Weights^{k,l}(U\times V,E)$ and $r_1>0$.

\underline{Step $1.1$.} We define the map

\begin{equation*}
	\eta_{r_1}:U\times V\to\interclcl{0}{1},\quad
	(x,y)\mapsto\left(\mathds{1}_{\overline{B_{r_1+\frac{1}{2}}(0)}} * \rho_{\frac{1}{2}}\right)(x),
\end{equation*}

using the function $\mathds{1}_{\overline{B_{r_1+\frac{1}{2}}(0)}}:\Real^n\to\left\{0,1\right\}$ and the mollifier 
$\rho_{\frac{1}{2}}:\Real^n\to\Real$. Note that $\eta_{r_1}(x,y)=0$ if $\left\|x\right\|_\infty>r_1+1$. 
Since $\eta_{r_1}$ is partially $C^\infty$, the map $\eta_{r_1}\cdot\gamma$ is partially $C^{k,l}$ 
(see Remark \ref{rem:part-Ck-part-Ckl-product-rules} (ii)). To see that $\eta_{r_1}\cdot\gamma\in C_\Weights^{k,l}(U\times V,E)$, let 
$f=f_1\otimes f_2\in\Weights$, $q\in\contsemiE$, $\alpha\in\Natural_0^n$ with $|\alpha|\leq k$, $\beta\in\Natural_0^m$ 
with $|\beta|\leq l$ and $(x,y)\in U\times V$. 
We have

\begin{align*}
	&	f_1(x)f_2(y)q\left(\partderx{\alpha}\partdery{\beta}(\eta_{r_1}\cdot\gamma)(x,y)\right)\\
	&	=f_1(x)f_2(y)q\left(\sum_{\tau\leq\alpha}\binom{\alpha}{\tau}
	\partderx{\alpha-\tau}\eta_{r_1}(x,y)\partderx{\tau}\partdery{\beta}\gamma(x,y)\right)\\
	&	\leq\sum_{\tau\leq\alpha}\binom{\alpha}{\tau}\left|\partderx{\alpha-\tau}\eta_{r_1}(x,y)\right|
	f_1(x)f_2(y)q\left(\partderx{\tau}\partdery{\beta}\gamma(x,y)\right),
\end{align*}

because $\partdery{\beta-\kappa}\eta_{r_1}(x,y)=0$ for all $\kappa\in\Natural_0^m$ with $\kappa<\beta$.

From Lemma \ref{lem:estimate-der-rho}, it follows that

\begin{equation*}
		\left|\partderx{\alpha-\tau}\eta_{r_1}(x,y)\right| 
		\leq 2^{|\alpha|-|\tau|}\cdot\left\|\partderx{\alpha-\tau}\rho\right\|_{L^1} =:C_1(\alpha,\tau),
\end{equation*}

whence

\begin{align*}
	\left\|\eta_{r_1}\cdot\gamma\right\|_{f,(\alpha,\beta),q}
	&\defeq \sup_{(x,y)\in U\times V} f_1(x)f_2(y)q\left(\partderx{\alpha}\partdery{\beta}(\eta_{r_1}\cdot\gamma)(x,y)\right)\\
	&\leq\sum_{\tau\leq\alpha}\binom{\alpha}{\tau}C_1(\alpha,\tau)\left\|\gamma\right\|_{f,(\tau,\beta),q}<\infty.
\end{align*}

Thus $\eta_{r_1}\cdot\gamma\in C_\Weights^{k,l}(U\times V,E)$. Now, we want to show that

\begin{equation}\label{eq:eta-r1gamma-to-gamma}
	\left\|\gamma - \eta_{r_1}\cdot\gamma\right\|_{f,(\alpha,\beta),q} \to 0
\end{equation}

as $r_1\to\infty$. The proof is by contradiction. If \eqref{eq:eta-r1gamma-to-gamma} is false, then for some $f\in\Weights$,
$q\in\contsemiE$ and $\alpha$, $\beta$ as above there exist $\delta>0$ and a sequence $(s_j)_{j\in\Natural}$ consisting of
values of $r_1$ such that $s_j\to\infty$ as $j\to\infty$ and

\begin{equation*}
	\left\|\gamma - \eta_{s_j}\cdot\gamma\right\|_{f,(\alpha,\beta),q}>\delta
\end{equation*}

for each $j\in\Natural$. Since the seminorm $\left\|\cdot\right\|_{f,(\alpha,\beta),q}$ is defined as a supremum, there exist
$(x_j,y_j)\in U\times V$ such that

\begin{equation}\label{eq:contr-first-assumption-sup-geq-delta}
	f_1(x_j)f_2(y_j)q\left(\partderx{\alpha}\partdery{\beta}(\gamma - \eta_{s_j}\cdot\gamma)(x_j,y_j)\right)>\delta.
\end{equation}

By construction, $\eta_{s_j}(x,y)=1$ if $\left\|x\right\|_\infty\leq s_j$, whence

\begin{equation*}
	q\left(\partderx{\alpha}\partdery{\beta}(\gamma - \eta_{s_j}\cdot\gamma)(x,y)\right)=0
\end{equation*}

in this case. Thus, we must have $\left\|x\right\|_\infty>s_j$, and hence $\left\|x\right\|_\infty\to\infty$ as $j\to\infty$. 
Further, we have

\begin{align*}
	&	f_1(x_j)f_2(y_j)q\left(\partderx{\alpha}\partdery{\beta}(\gamma - \eta_{s_j}\cdot\gamma)(x_j,y_j)\right)\\
	&	\leq f_1(x_j)f_2(y_j)q\left(\partderx{\alpha}\partdery{\beta}\gamma(x_j,y_j)\right)\\
	&	\mbox{\quad}+ f_1(x_j)f_2(y_j)q\left(\partderx{\alpha}\partdery{\beta}(\eta_{s_j}\cdot\gamma)(x_j,y_j)\right)\\
	&	\leq f_1(x_j)f_2(y_j)q\left(\partderx{\alpha}\partdery{\beta}\gamma(x_j,y_j)\right)\\
	&	\mbox{\quad}+ \sum_{\tau\leq\alpha}\binom{\alpha}{\tau}C_1(\alpha,\tau)
	f_1(x_j)f_2(y_j)q\left(\partderx{\tau}\partdery{\beta}\gamma(x_j,y_j)\right).
\end{align*}

Since $U$ is open and $\Weights_1$ satisfies the $o$-condition, Lemma \ref{lem:o-condition-then-Ckl-derivatives-to-zero} $(i)$
yields

\begin{equation*}
	f_1(x_j)f_2(y_j)q\left(\partderx{\alpha}\partdery{\beta}\gamma(x_j,y_j)\right)\to 0
\end{equation*}

and

\begin{equation*}
	f_1(x_j)f_2(y_j)q\left(\partderx{\tau}\partdery{\beta}\gamma(x_j,y_j)\right)\to 0,
\end{equation*}

uniformly in $y_j$, as $j\to\infty$. This contradicts \eqref{eq:contr-first-assumption-sup-geq-delta}, as required.

\underline{Step $1.2$.} Now, we define the map

\begin{equation*}
	\zeta_{r_1}:U\times V\to\interclcl{0}{1},\quad
	(x,y)\mapsto\left(\mathds{1}_{\overline{B_{r_1+\frac{1}{2}}(0)}} * \rho_{\frac{1}{2}}\right)(y),
\end{equation*}

using the function $\mathds{1}_{\overline{B_{r_1+\frac{1}{2}}(0)}}:\Real^m\to\left\{0,1\right\}$ and the mollifier 
$\rho_{\frac{1}{2}}:\Real^m\to\Real$. Thus $\zeta_{r_1}(x,y)=0$ if $\left\|y\right\|_\infty>r_1+1$. 
Proceeding similarly to Step $1.1$ (and using that $V$ is open and $\Weights_2$ satisfies the $o$-condition), we show that

\begin{equation*}
	\left\|\gamma - \zeta_{r_1}\cdot\gamma\right\|_{f,(\alpha,\beta),q} \to 0
\end{equation*}

as $r_1\to\infty$, for each weight $f=f_1\otimes f_2\in\Weights$, $q\in\contsemiE$, $\alpha\in\Natural_0^n$ with $|\alpha|\leq k$
and $\beta\in\Natural_0^m$ with $|\beta|\leq l$.

\underline{Conclusion of Step $1$.} In Step $1.1$. we approximated $\gamma\in C_\Weights^{k,l}(U\times V,E)$ with a function
$\eta\in C_\Weights^{k,l}(U\times V,E)$ such that $\eta(x,y)=0$ if $\left\|x\right\|_\infty>r$ for some $r>0$. Further, from Step $1.2$
we conclude that $\eta$ can be approximated with a function $\zeta\in C_\Weights^{k,l}(U\times V,E)$ such that $\zeta(x,y)=0$ if 
$\left\|x\right\|_\infty>r$ or $\left\|y\right\|_\infty>r$ for some $r>0$. That is $\zeta\in M$, whence $M$ is dense in
$C_\Weights^{k,l}(U\times V,E)$, as asserted.

\underline{Step $2$.} Now, let $\gamma\in M$, that is, $\gamma\in C_\Weights^{k,l}(U\times V,E)$ and $\gamma(x,y)=0$ 
if $\left\|(x,y)\right\|_\infty>r$ for some $r>0$. 

\underline{Step $2.1$.} If $U=\Real^n$, then we jump to Step $2.2$. Otherwise, for 
$\varepsilon>0$ we define the set

\begin{equation*}
	U_\varepsilon := \left\{x\in U : d(x,\partial U)\geq\varepsilon\right\}.
\end{equation*}

Further, for $r_2>0$ we define the map

\begin{equation*}
	\eta_{r_2} : U\times V\to\interclcl{0}{1},\quad(x,y)\mapsto\left(\mathds{1}_{U_{\frac{3}{4r_2}}} * \rho_{\frac{1}{4r_2}}\right)(x),
\end{equation*}

using the functions $\mathds{1}_{U_{\frac{3}{4r_2}}}:\Real^n\to\left\{0,1\right\}$ and $\rho_{\frac{1}{4r_2}}:\Real^n\to\Real$. 
(Note that $\eta_{r_2}(x,y)=0$ if $d(x,\partial U)<\frac{1}{2r_2}$.) Then $\eta_{r_2}$ is partially $C^\infty$, whence the map $\eta_{r_2}\cdot\gamma$ is partially $C^{k,l}$ 
(see Remark \ref{rem:part-Ck-part-Ckl-product-rules} (ii)). We want to show that

\begin{equation}\label{eq:eta-r2gamma-to-gamma}
	\left\|\gamma - \eta_{r_2}\cdot\gamma\right\|_{f,(\alpha,\beta),q}\to 0	
\end{equation}

as $r_2\to\infty$, for each weight $f=f_1\otimes f_2\in\Weights$, seminorm $q\in\contsemiE$ and $\alpha\in\Natural_0^n$ with $|\alpha|\leq k$, $\beta\in\Natural_0^m$ with $|\beta|\leq l$. The proof is by contradiction. If \eqref{eq:eta-r2gamma-to-gamma} is wrong, then we find
$\delta>0$ and a sequence $(s_j)_{j\in\Natural}$ of values of $r_2$ such that $s_j\to\infty$ as $j\to\infty$ and for each $j\in\Natural$
we have

\begin{equation*}
	\left\|\gamma - \eta_{s_j}\cdot\gamma\right\|_{f,(\alpha,\beta),q}>\delta.
\end{equation*}

By definition of the seminorm $\left\|\cdot\right\|_{f,(\alpha,\beta),q}$ as a supremum, there exist $(x_j,y_j)\in U\times V$ such that

\begin{equation}\label{eq:contr-second-assumption-sup-geq-delta}
	f_1(x_j)f_2(y_j)q\left(\partderx{\alpha}\partdery{\beta}\left(\gamma - \eta_{s_j}\cdot\gamma\right)(x_j,y_j)\right)>\delta.
\end{equation}

Since $\gamma\in M$, we have $\left(\gamma - \eta_{s_j}\cdot\gamma\right)(x,y)=0$ for all $(x,y)\in U\times V$ with 
$\left\|(x,y)\right\|_\infty>r$, therefore we must have $(x_j,y_j)\in\overline{B_r(0)}$. Thus, after passage to a subsequence, we
may assume $(x_j,y_j)\to(\overline{x},\overline{y})\in\overline U\times\overline V$ as $j\to\infty$. Further, we have 
$\eta_{s_j}(x,y)=1$ for all $x\in U_{\frac{1}{s_j}}$, and hence $(\gamma - \eta_{s_j}\cdot\gamma)(x,y)=0$ in this case. 
Thus we conclude that $d(x_j,\partial U)<\frac{1}{s_j}$, whence $d(\overline{x},\partial U)=0$, that is $\overline{x}\in\partial U$. 
We have

\begin{align}
	&f_1(x_j)f_2(y_j)q\left(\partderx{\alpha}\partdery{\beta}(\gamma - \eta_{s_j}\cdot\gamma)(x_j,y_j)\right)\nonumber\\
	&\leq f_1(x_j)f_2(y_j)q\left(\partderx{\alpha}\partdery{\beta}\gamma(x_j,y_j)\right)\label{eq:part-two-sum-first}\\ 
	&+ f_1(x_j)f_2(y_j)q\left(\partderx{\alpha}\partdery{\beta}(\eta_{s_j}\cdot\gamma)(x_j,y_j)\right)\label{eq:part-two-sum-second}.	
\end{align}

If $j\to\infty$, then in \eqref{eq:part-two-sum-first} we have

\begin{equation*}
	f_1(x_j)f_2(y_j)q\left(\partderx{\alpha}\partdery{\beta}\gamma(x_j,y_j)\right)\to 0
\end{equation*}

uniformly in $y_j$, by Lemma \ref{lem:o-condition-then-Ckl-derivatives-to-zero} $(i)$, since $U$ is an open subset and $\Weights_1$
satisfies the $o$-condition. We apply the Product Rule 
to \eqref{eq:part-two-sum-second} and obtain

\begin{align*}
	&f_1(x_j)f_2(y_j)q\left(\partderx{\alpha}\partdery{\beta}(\eta_{s_j}\cdot\gamma)(x_j,y_j)\right)\\
	&=f_1(x_j)f_2(y_j)q\left(\sum_{\tau\leq\alpha}\binom{\alpha}{\tau}
	\partderx{\alpha-\tau}\eta_{s_j}(x_j,y_j)\partderx{\tau}\partdery{\beta}\gamma(x_j,y_j)\right)\\
	&\leq\sum_{\tau\leq\alpha}\binom{\alpha}{\tau}\left|\partderx{\alpha-\tau}\eta_{s_j}(x_j,y_j)\right|
	f_1(x_j)f_2(y_j)q\left(\partderx{\tau}\partdery{\beta}\gamma(x_j,y_j)\right),
\end{align*}

using that $\partdery{\beta-\kappa}\eta_{s_j}(x_j,y_j)=0$ for all $\kappa\in\Natural_0^m$ with $\kappa<\beta$.

From Lemma \ref{lem:estimate-der-rho} we conclude

\begin{align*}
	&f_1(x_j)f_2(y_j)q\left(\partderx{\alpha}\partdery{\beta}(\eta_{s_j}\cdot\gamma)(x_j,y_j)\right)\\
	&\leq\sum_{\tau\leq\alpha}\binom{\alpha}{\tau}C_2(\alpha,\tau)s_j^{|\alpha|-|\tau|}
	f_1(x_j)f_2(y_j)q\left(\partderx{\tau}\partdery{\beta}\gamma(x_j,y_j)\right)
\end{align*}

with

\begin{equation*}
	C_2(\alpha,\tau):=4^{|\alpha|-|\tau|}\cdot\left\|\partderx{\alpha-\tau}\rho\right\|_{L^1}.
\end{equation*}

Now, we want to show that the term

\begin{equation}\label{eq:part-two-important-term}
	s_j^{|\alpha|-|\tau|}f_1(x_j)f_2(y_j)q\left(\partderx{\tau}\partdery{\beta}\gamma(x_j,y_j)\right)
\end{equation}

tends to $0$ as $j\to\infty$, for all $\tau\leq\alpha$ (which yields the desired contradiction). If $\tau=\alpha$, then
in \eqref{eq:part-two-important-term} we have

\begin{equation*}
	s_j^0f_1(x_j)f_2(y_j)q\left(\partderx{\alpha}\partdery{\beta}\gamma(x_j,y_j)\right)\to 0
\end{equation*}

uniformly in $y_j$, as $j\to\infty$, by Lemma \ref{lem:o-condition-then-Ckl-derivatives-to-zero} $(i)$. Otherwise,
we rewrite \eqref{eq:part-two-important-term} as

\begin{align*}
	&s_j^{|\alpha|-|\tau|}f_1(x_j)f_2(y_j)q\left(\partderx{\tau}\partdery{\beta}\gamma(x_j,y_j)\right)\\
	&=s_j^{|\alpha|-|\tau|}q\left(f_1(x_j)\partderx{\tau}\left(f_2(y_j)\partdery{\beta}\gamma(\bullet,y_j)\right)(x_j)\right)\\
	&=s_j^{|\alpha|-|\tau|}q\left(f_1(x_j)\partderx{\tau}\gamma_{\beta,y_j,f_2}(x_j)\right),
\end{align*}

where $\gamma_{\beta,y_j,f_2}$ is the map defined in Lemma \ref{lem:useful-weighted-Ckl-maps}. We have 
$\gamma_{\beta,y_j,f_2}\in C_{\Weights_1}^k(U,E)\subseteq C_{\Weights_1}^{|\alpha|}(U,E)$, whence

\begin{equation*}
	\partderx{\tau}\gamma_{\beta,y_j,f_2}\in C_{\Weights_1}^{|\alpha|-|\tau|}(U,E).
\end{equation*}

Using Remark \ref{rem:o-condition-then-derivatives-to-zero} and Lemma \ref{lem:o-condition-then-Ckl-derivatives-to-zero} $(i)$,
we see that for each $\pi\in\Natural_0^n$ with $|\pi|\leq|\alpha|-|\tau|$ we have

\begin{equation*}
	f_1(x_j)\partderx{\tau}\gamma_{\beta,y_j,f_2}(x_j)\to 0
\end{equation*}

and

\begin{equation*}
	\partderx{\pi}\left(f_1\cdot\partderx{\tau}\gamma_{\beta,y_j,f_2}\right)(x_j)\to 0
\end{equation*}

uniformly in $y_j$, as $j\to\infty$. Hence, Proposition \ref{prop:partially-Ck-extensions-partially-Ck} implies that the extension 

\begin{equation*}
	\widetilde{\left(f_1\cdot\partderx{\tau}\gamma_{\beta,y_j,f_2}\right)} :\Real^n\to E,\quad
														x\mapsto\left\{
															\begin{array}{ll}	
																\left(f_1\cdot\partderx{\tau}\gamma_{\beta,y_j,f_2}\right)(x)	& x\in U\\
																0	& x\notin U\\
															\end{array}
														\right.
\end{equation*}

is partially $C^{|\alpha|-|\tau|}$ with partial derivatives

\begin{equation*}
	\partderx{\pi}\widetilde{\left(f_1\cdot\partderx{\tau}\gamma_{\beta,y_j,f_2}\right)} :\Real^n\to E,\quad
														x\mapsto\left\{
															\begin{array}{ll}
																\partderx{\pi}\left(f_1\cdot\partderx{\tau}\gamma_{\beta,y_j,f_2}\right)(x)	& x\in U\\
																0	& x\notin U.\\
															\end{array}
														\right.	
\end{equation*}

Since $d(x_j,\partial U)<\frac{1}{s_j}$ for each $j\in\Natural$, there exists $\overline{x}_j\in\partial U$ with 
$d_j:=d(x_j,\overline{x}_j)<\frac{1}{s_j}$, and $\overline{x}_j\to\overline{x}$ as $j\to\infty$. 
Lemma \ref{lem:taylor-series} yields $\xi_j\in\interclcl{0}{1}$ such that

\begin{align*}
	&s_j^{|\alpha|-|\tau|}q\left(f_1(x_j)\partderx{\tau}\gamma_{\beta,y_j,f_2}(x_j)\right)\\
	&\leq\frac{(s_jd_j)^{|\alpha|-|\tau|}}{(|\alpha|-|\tau|-1)!}\sum_{|\kappa|=|\alpha|-|\tau|}
				q\left(\partderx{\kappa}\widetilde{\left(f_1\cdot\partderx{\tau}\gamma_{\beta,y_j,f_2}\right)}(x_j')\right)\\
	&<\frac{1}{(|\alpha|-|\tau|-1)!}\sum_{|\kappa|=|\alpha|-|\tau|}
				q\left(\partderx{\kappa}\widetilde{\left(f_1\cdot\partderx{\tau}\gamma_{\beta,y_j,f_2}\right)}(x_j')\right)
\end{align*}

with $x_j':=\overline{x}_j-\xi_j(x_j-\overline{x}_j)$. Thus we have

\begin{equation*}
	q\left(\partderx{\kappa}\widetilde{\left(f_1\cdot\partderx{\tau}\gamma_{\beta,y_j,f_2}\right)}(x_j')\right)\to 0
\end{equation*}

if $j\to\infty$ (i.e., $x_j,\overline{x}_j\to\overline{x}$, whence $x_j'\to\overline{x}$).

Altogether, we have 

\begin{equation*}
		f_1(x_j)f_2(y_j)q\left(\partderx{\alpha}\partdery{\beta}(\gamma - \eta_{s_j}\cdot\gamma)(x_j,y_j)\right)\to 0
\end{equation*}

as $j\to\infty$, which contradicts \eqref{eq:contr-second-assumption-sup-geq-delta}, as desired.

\underline{Step $2.2$.} If $V=\Real^m$, then we go over to the conclusion of Step $2$. Otherwise, we define the set 
$V_\varepsilon$ and the partially $C^\infty$-map 

\begin{equation*}
	\zeta_{r_2}:U\times V\to\interclcl{0}{1},\quad(x,y)\mapsto\left(\mathds{1}_{V_{\frac{3}{4r_2}}} * \rho_{\frac{1}{4r_2}}\right)(y)
\end{equation*}

as in Step $2.1$ (note that $\zeta_{r_2}(x,y)=0$ if $d(y,\partial V)<\frac{1}{2r_2}$) and proceed similarly. This yields

\begin{equation*}
	\left\|\gamma - \zeta_{r_2}\cdot\gamma\right\|_{f,(\alpha,\beta),q}\to 0
\end{equation*}

as $r_2\to\infty$, for each weight $f=f_1\otimes f_2\in\Weights$, seminorm $q\in\contsemiE$ and $\alpha\in\Natural_0^n$ with $|\alpha|\leq k$, $\beta\in\Natural_0^m$ with $|\beta|\leq l$.

\underline{Conclusion of Step $2$.} We conclude that the space $C_c^{k,l}(U\times V,E)$ is dense in $M$. In fact, if $U=\Real^n$ and $V=\Real^m$,
then the assertion is obviously true. Consider the case $U\subsetneqq\Real^n$ and $V\subsetneqq\Real^m$. In Step $2.1$ we approximated
$\gamma\in M$ with a partially $C^{k,l}$-map $\eta$ such that

\begin{align*}
	\supp(\eta)\subseteq\left\{(x,y)\in U\times V \right.&: (\exists r,s>0)\\
	& \left.\left\|(x,y)\right\|_\infty\leq r\mbox{ and }d(x,\partial U)\geq s\right\}.
\end{align*}

In Step $2.2$ we showed that $\eta$ can be approximated with a partially $C^{k,l}$-map $\zeta$ such that

\begin{align*}
	\supp(\zeta)\subseteq\left\{(x,y)\right.&\in U\times V : (\exists r,s>0)\\
	& \left.\left\|(x,y)\right\|_\infty\leq r\mbox{ and }d((x,y),\partial(U\times V))\geq s\right\}=:K.
\end{align*}

Furthermore, we have

\begin{equation}\label{eq:K-as-closed-bounded-subset}
\begin{split}
	K=\left\{(x,y)\right.\in \overline{U}\times\overline{V} &: (\exists r,s>0)\\
	& \left.\left\|(x,y)\right\|_\infty\leq r\mbox{ and }d((x,y),\partial(U\times V))\geq s\right\}.
\end{split}	
\end{equation}

(In fact, for $(x,y)\in(\overline{U}\times\overline{V})\backslash(U\times V)=\overline{U\times V}\backslash(U\times V)=\partial(U\times V)$
we have $d((x,y),\partial(U\times V))=0$, hence $(x,y)$ is not an element of the right-hand side of \eqref{eq:K-as-closed-bounded-subset}.)
Thus $K$ is a closed subset of $(\overline{U}\times\overline{V})\cap\overline{B_r(0)}$, hence $K$ is bounded and consequently compact.
Therefore $\zeta\in C_c^{k,l}(U\times V,E)$.

If $U\subsetneqq\Real^n$ and $V=\Real^m$, then Step $2.1$ yields an approximation of $\gamma\in M$ with a partially $C^{k,l}$-map
$\eta$ such that

\begin{align*}
	\supp(\eta)\subseteq\left\{(x,y)\in U\times\Real^m \right.&: (\exists r,s>0)\\
	& \left.\left\|(x,y)\right\|_\infty\leq r\mbox{ and }d(x,\partial U)\geq s\right\}\\
	=\left\{(x,y)\in\overline{U}\times\Real^m \right.&: (\exists r,s>0)\\
	& \left.\left\|(x,y)\right\|_\infty\leq r\mbox{ and }d(x,\partial U)\geq s\right\}.
\end{align*}

(We see again that if $(x,y)\in(\overline{U}\times\Real^m)\backslash(U\times\Real^m)=\overline{U\times\Real^m}\backslash(U\times\Real^m)
=\partial(U\times\Real^m)=(\partial U\times\Real^m)\cup(\overline{U}\times\partial\Real^m)=\partial U\times\Real^m$, then $d(x,\partial U)=0$.)
Thus $\supp(\eta)$ is a subset of a closed subset of $(\overline{U}\times\Real^m)\cap\overline{B_r(0)}$, and we have
$\eta\in C_c^{k,l}(U\times\Real^m,E)$. In the case $U=\Real^n$ and $V\subsetneqq\Real^m$, the argumentation is similar, using Step $2.2$.
\end{proof}

\begin{proposition}\label{prop:Ckl-cp-supp-in-image}
Let $E$ be a \Hlc space, $U\subseteq\Real^n$ and $V\subseteq\Real^m$ be locally convex subsets with dense interior, and $k,l\in\Natzeroinfty$.
For the set of weights $\Weights_1$ on $U$ we assume that

\begin{enumerate}
	\item [(i)] each weight is bounded on compact subsets of $U$,
	\item [(ii)] for each compact subset $K\subseteq U$ there exists a weight $f_K\in\Weights_1$ such that $\inf_{x\in K}f_K(x)>0$,
\end{enumerate}

and likewise for the set of weights $\Weights_2$ on $V$. Then we have $C_c^{k,l}(U\times V,E)\subseteq im(\Psi)$, where $\Psi$ is the topological embedding

\begin{equation*}
	\Psi:C_{\Weights_1}^k(U,C_{\Weights_2}^l(V,E))\to C_\Weights^{k,l}(U\times V,E),\quad\gamma\mapsto\gamma^\wedge
\end{equation*}

defined in Theorem \ref{thm:weighted-Ck-exp-top-emb}
\end{proposition}

\begin{proof}
Let $\gamma\in C_c^{k,l}(U\times V,E)$. We need to show that $\gamma^\vee\in C_{\Weights_1}^k(U,C_{\Weights_2}^l(V,E))$,
then by construction of $\Psi$ we have $\Psi(\gamma^\vee)=(\gamma^\vee)^\wedge=\gamma$, as required.

Let

\begin{equation*}
	\pi_1:U\times V\to U,\quad\pi_2:U\times V\to V
\end{equation*}

be the coordinate projections onto the first and second component, respectively. We set $K:=\supp(\gamma)$ and
$K_1:=\pi_1(K)$, $K_2:=\pi_2(K)$. Since $K$ is compact, so is the set $K_1\times K_2$ and $K\subseteq K_1\times K_2$,
whence $\gamma\in C_{K_1\times K_2}^{k,l}(U\times V,E)$. Therefore, we have $\gamma^\vee\in C_{K_1}^k(U,C_{K_2}^l(V,E))$,
by Lemma \ref{lem:gamma-check-kl-cp-supp}. We know that $C_{K_2}^l(V,E)\subseteq C_{\Weights_2}^l(V,E)$ and the 
inclusion map

\begin{equation*}
	i:C_{K_2}^l(V,E)\to C_{\Weights_2}^l(V,E)
\end{equation*}

is continuous and linear, by Lemma \ref{lem:Ck-K-supp-in-part-Ck-weighted-incl-cont}. Thus 
$\gamma^\vee(x)=\gamma_x\in C_{\Weights_2}^l(V,E)$ and $\gamma^\vee=i\circ\gamma^\vee:U\to C_{\Weights_2}^l(V,E)$ is
partially $C^k$, 
by Lemma \ref{lem:lcx-partCk-lin-operator}. Hence
$\gamma^\vee\in C_{K_1}^k(U,C_{\Weights_2}^l(V,E))\subseteq C_{\Weights_1}^k(U,C_{\Weights_2}^l(V,E))$, as desired.
\end{proof}

Now, we prove the Exponential Law for spaces of weighted differentiable functions
with values in locally convex spaces.

\begin{theorem}[\textbf{Exponential Law for spaces of weighted differentiable functions}]\label{thm:weighted-Ckl-exp-law-comb}
Let $E$ be a \Hlc space, $U\subseteq\Real^n$ and $V\subseteq\Real^m$ be open subsets, and $k,l\in\Natzeroinfty$.
For the set of weights $\Weights_1\subseteq C^k(U,\zeroinfty)$ on $U$ we assume that

\begin{itemize}
	\item [(i)]	$\Weights_1$ satisfies the $o$-condition,
	\item [(ii)]	for each $f\in\Weights_1$ and $\alpha\in\Natural_0^n$ with $|\alpha|\leq k$ there exists 
					$g\in\Weights_1$ such that
					
					\begin{equation*}
						\left|\partderx{\alpha}f(x)\right|\leq g(x)
					\end{equation*}
					
					for all $x\in U$,
\end{itemize}

and likewise for the set of weights $\Weights_2\subseteq C^l(V,\zeroinfty)$ on $V$. Then the linear map 

\begin{equation*}
	\Psi:C_{\Weights_1}^k(U,C_{\Weights_2}^l(V,E))\to C_\Weights^{k,l}(U\times V,E),\quad\gamma\mapsto\gamma^\wedge,
\end{equation*}

where $\Weights=\Weights_1\otimes\Weights_2$, is a homeomorphism.
\end{theorem}

\begin{proof}
\underline{Step $1$.} First we assume that the space $E$ is complete. By Proposition \ref{prop:weighted-Ck-space-complete}, the space $C_{\Weights_2}^l(V,E)$ is complete, whence also
the space $C_{\Weights_1}^k(U,C_{\Weights_2}^l(V,E))$ is complete (we recall that the condition in Proposition 
\ref{prop:weighted-Ck-space-complete} is satisfied, since the weights are assumed continuous, see Remark 
\ref{rem:weights-cont-then-inf-geq-0}). The map
$\Psi$ is a topological embedding, by Theorem \ref{thm:weighted-Ck-exp-top-emb}, thus $im(\Psi)$ is complete, hence it is
closed in $C_\Weights^{k,l}(U\times V,E)$. By Proposition \ref{prop:Ckl-cp-supp-in-image} we know that
$C_c^{k,l}(U\times V,E)\subseteq im(\Psi)$, and in Proposition \ref{prop:smooth-compact-supp-dense-in-weighted-Ckl-assertion} we have
shown that $C_c^{k,l}(U\times V,E)$ is dense in $C_\Weights^{k,l}(U\times V,E)$. Thus $im(\Psi)$ is dense in 
$C_\Weights^{k,l}(U\times V,E)$, whence

\begin{equation*}
	im(\Psi)=\overline{im(\Psi)}=C_\Weights^{k,l}(U\times V,E),
\end{equation*}

which shows that $\Psi$ is surjective, hence a homeomorphism (being a topological embedding).

\underline{Step $2$.} Now we show the surjectivity of $\Psi$ in the general case. To this end, let $\widetilde{E}$ be the completion of $E$ and 
$\gamma\in C_\Weights^{k,l}(U\times V,E)\subseteq C_\Weights^{k,l}(U\times V,\widetilde{E})$. In Step $1$ we have shown that the map

\begin{equation*}
	\widetilde{\Psi}:C_{\Weights_1}^k(U,C_{\Weights_2}^l(V,\widetilde{E}))\to C_\Weights^{k,l}(U\times V,\widetilde{E}),
	\quad\eta\mapsto\eta^\wedge,
\end{equation*}

is a homeomorphism. Hence there exists $\eta\in C_{\Weights_1}^k(U,C_{\Weights_2}^l(V,\widetilde{E}))$ with

\begin{equation}\label{eq:tilde-Psi-eq-gamma}
	\widetilde{\Psi}(\eta)=\eta^\wedge=\gamma.
\end{equation}

If we can show that $\eta\in C_{\Weights_1}^k(U,C_{\Weights_2}^l(V,E))$, then we have $\Psi(\eta)=\gamma$,
and the assertion is proven. It suffices to prove that

\begin{equation}\label{eq:part-der-eta-in-E}
	\partdery{\beta}\left(\partderx{\alpha}\eta(x)\right)(y)\in E
\end{equation}

for all $x\in U$, $y\in V$ and $\alpha\in\Natural_0^n$ with $|\alpha|\leq k$ and $\beta\in\Natural_0^m$
with $|\beta|\leq l$. From Lemma \ref{lem:lcx-part-Ck-in-image} we then conclude that

\begin{equation*}
	\partderx{\alpha}\eta(x)\in C_{\Weights_2}^l(V,E)
\end{equation*}

for all $x\in U$ and $\alpha$ as above, and using the lemma again, we get

\begin{equation*}
	\eta\in C_{\Weights_1}^k(U,C_{\Weights_2}^l(V,E)),
\end{equation*}

as required. But using \eqref{eq:tilde-Psi-eq-gamma}, we have

\begin{equation*}
	\partdery{\beta}\left(\partderx{\alpha}\eta(x)\right)(y)
	=\partdery{\beta}\left(\partderx{\alpha}\gamma(x,\bullet)\right)(y)
	=\partderx{\alpha}\partdery{\beta}\gamma(x,y)\in E,
\end{equation*}

thus \eqref{eq:part-der-eta-in-E} holds and the proof is finished.
\end{proof}

A useful consequence of the Exponential Law is the following:

\begin{corollary}\label{cor:weighted-Ckl-exp-law}
Let $E$ be a \Hlc space, $U\subseteq\Real^n$ and $V\subseteq\Real^m$ be open subsets, and $k,l\in\Natzeroinfty$.
For the set of weights $\Weights_1\subseteq C^k(U,\zeroinfty)$ on $U$ we assume that

\begin{itemize}
	\item [(i)]	$\Weights_1$ satisfies the $o$-condition,
	\item [(ii)]	for each $f\in\Weights_1$ and $\alpha\in\Natural_0^n$ with $|\alpha|\leq k$ there exists 
					$g\in\Weights_1$ such that
					
					\begin{equation*}
						\left|\partderx{\alpha}f(x)\right|\leq g(x)
					\end{equation*}
					
					for all $x\in U$,
\end{itemize}

and likewise for the set of weights $\Weights_2\subseteq C^l(V,\zeroinfty)$ on $V$. Then the linear map 

\begin{equation*}
\begin{split}
	\Psi:C_{\Weights_1}^k(U,C_{\Weights_2}^l(V,E))&\to C_{\Weights_2}^l(V,C_{\Weights_1}^k(U,E)),\\
	\quad\Psi(\gamma)&:=(y\mapsto(x\mapsto\gamma(x)(y)))
\end{split}
\end{equation*}

is a homeomorphism.
\end{corollary}

\begin{proof}
By Theorem \ref{thm:weighted-Ckl-exp-law-comb}, the map

\begin{equation*}
	\Psi_1:C_{\Weights_1}^k(U,C_{\Weights_2}^l(V,E))\to C_{\Weights_1\otimes\Weights_2}^{k,l}(U\times V,E)
\end{equation*}

is a homeomorphism, and obviously so is

\begin{equation*}
\begin{split}
	\Psi_2:C_{\Weights_1\otimes\Weights_2}^{k,l}(U\times V,E)&\to 
	C_{\Weights_2\otimes\Weights_1}^{l,k}(V\times U,E),\\
	\quad\Psi_2((x,y)\mapsto\gamma(x)(y))&:=((y,x)\mapsto\gamma(x)(y)).
\end{split}
\end{equation*}

Finally, using the homeomorphism

\begin{equation*}
	\Psi_3:C_{\Weights_2\otimes\Weights_1}^{l,k}(V\times U,E)\to C_{\Weights_2}^l(V,C_{\Weights_1}^k(U,E)),
\end{equation*}

we construct

\begin{equation*}
	\Psi:=\Psi_3\circ\Psi_2\circ\Psi_1:
	C_{\Weights_1}^k(U,C_{\Weights_2}^l(V,E))\to C_{\Weights_2}^l(V,C_{\Weights_1}^k(U,E)),
\end{equation*}

and see that

\begin{align*}
	\Psi(\gamma)
	&=(\Psi_3\circ\Psi_2\circ\Psi_1)(x\mapsto(y\mapsto\gamma(x)(y)))\\
	&=(\Psi_3\circ\Psi_2)((x,y)\mapsto\gamma(x)(y))\\
	&=\Psi_3((y,x)\mapsto\gamma(x)(y))
	=(y\mapsto(x\mapsto\gamma(x)(y))),
\end{align*}

and $\Psi$ is a homeomorphism, by construction.
\end{proof}

Using the fact that a convex, compact subset $K\subseteq\Real^m$  with nonempty interior has dense interior (whence
differentiability is defined on such sets), we get
corresponding results for spaces of weighted $C^{k,l}$-functions on products of open and compact subsets, in particular:

\begin{proposition}\label{prop:Ckl-compact-supp-dense-in-weighted-Ckl-cp-op-and-op-cp}
Let $E$ be a \Hlc space, $U\subseteq\Real^n$ be an open subset,
$K\subseteq\Real^m$ be a convex, compact subset (with $K^\circ\neq\emptyset$) and $k,l\in\Natzeroinfty$. Let $\Weights_1\subseteq C^k(U,\zeroinfty$ be a set of weights on $U$ as in Theorem 
\ref{thm:weighted-Ckl-exp-law-comb} and $\Weights_2$ be an arbitrary set of weights on $K$. Then

\begin{itemize}
	\item[(i)] the space $C_c^{k,l}(U\times K,E)$ is dense in $C_{\Weights_1\otimes\Weights_2}^{k,l}(U\times K,E)$
\end{itemize}	

 and
	 
\begin{itemize}	 
	\item[(ii)] the space $C_c^{l,k}(K\times U,E)$ is dense in $C_{\Weights_2\otimes\Weights_1}^{l,k}(K\times U,E)$.
\end{itemize}

\end{proposition}

\begin{proof}
To prove $(i)$, we use the same arguments as in Step $1.1$ in the proof of Proposition 
\ref{prop:smooth-compact-supp-dense-in-weighted-Ckl-assertion} and
approximate $\gamma\in C_{\Weights_1\otimes\Weights_2}^{k,l}(U\times K,E)$ with a function 
$\eta\in C_{\Weights_1\otimes\Weights_2}^{k,l}(U\times K,E)$ such that 
$\eta(x,y)=0$ if $\left\|x\right\|_\infty>r$ for some $r>0$. If $U=\Real^n$, then the proof is finished. 
Otherwise, we use Step $2.1$ and 
approximate $\eta$ with a partially $C^{k,l}$-function $\zeta$ such that

\begin{align*}
	\supp(\zeta)&\subseteq\left\{(x,y)\in U\times K : (\exists r,s>0)\left\|x\right\|_\infty\leq r\mbox{ and }d(x,\partial U)\geq s\right\}\\
	&=\left\{(x,y)\in \overline{U}\times K : (\exists r,s>0)\left\|x\right\|_\infty\leq r\mbox{ and }d(x,\partial U)\geq s\right\}.
\end{align*}

This is a compact subset (see the conclusion of Step 2 for details), whence $\zeta\in C_c^{k,l}(U\times K,E)$, as required.

To prove $(ii)$, we proceed similarly, using Step $1.2$ and Step $2.2$ in the proof of Proposition 
\ref{prop:smooth-compact-supp-dense-in-weighted-Ckl-assertion}.
\end{proof}

\begin{theorem}\label{thm:weighted-Ckl-exp-law-cp-op-and-op-cp}
Let $E$ be a \Hlc space, $U\subseteq\Real^n$ be an open subset and $K\subseteq\Real^m$ be a convex, compact subset (with $K^\circ\neq\emptyset$). Let $k,l\in\Natzeroinfty$, $\Weights_1\subseteq C^k(U,\zeroinfty)$ be a set of weights on $U$
as in Theorem \ref{thm:weighted-Ckl-exp-law-comb} and $\Weights_2$ be a set of weights on $K$ such that

\begin{itemize}
	\item[(i)]	there is a weight $h_K\in\Weights_2$ such that $\inf_{y\in K}h_K(y)>0$,
	\item[(ii)]	each weight $h\in\Weights_2$ is bounded on $K$.
\end{itemize}

Then each of the linear maps

\begin{equation*}
	\Psi_1:C_{\Weights_1}^k(U,C_{\Weights_2}^l(K,E))\to C_{\Weights_1\otimes\Weights_2}^{k,l}(U\times K,E),
	\quad\gamma\mapsto\gamma^\wedge,
\end{equation*}

and

\begin{equation*}
	\Psi_2:C_{\Weights_2}^l(K,C_{\Weights_1}^k(U,E))\to C_{\Weights_2\otimes\Weights_1}^{l,k}(K\times U,E),
	\quad\gamma\mapsto\gamma^\wedge
\end{equation*}

is a homeomorphism.
\end{theorem}

\begin{proof}
The proposition can be proven similarly to Theorem \ref{thm:weighted-Ckl-exp-law-comb}, using Proposition
\ref{prop:Ckl-compact-supp-dense-in-weighted-Ckl-cp-op-and-op-cp} instead of Proposition 
\ref{prop:smooth-compact-supp-dense-in-weighted-Ckl-assertion}.
\end{proof}

\begin{corollary}\label{cor:weighted-Ckl-exp-law-cp-open}
Let $E$ be a \Hlc space, $U\subseteq\Real^n$ be an open subset and $K\subseteq\Real^m$ be a convex, compact subset (with $K^\circ\neq\emptyset$). Let $k,l\in\Natzeroinfty$, and $\Weights_1$, $\Weights_2$ be sets of weights
on $U$ and $K$, respectively, as is Theorem \ref{thm:weighted-Ckl-exp-law-cp-op-and-op-cp}.

Then the linear map 

\begin{equation*}
\begin{split}
	\Psi:C_{\Weights_1}^k(U,C_{\Weights_2}^l(K,E))&\to C_{\Weights_2}^l(K,C_{\Weights_1}^k(U,E)),\\
	\quad\Psi(\gamma)&:=(y\mapsto(x\mapsto\gamma(x)(y)))
\end{split}
\end{equation*}

is a homeomorphism.
\end{corollary}

\begin{proof}
The assertion can be proven similarly to Corollary \ref{cor:weighted-Ckl-exp-law}, using Theorem 
\ref{thm:weighted-Ckl-exp-law-cp-op-and-op-cp} for the
construction of the homeomorphism $\Psi$.
\end{proof}

\begin{remark}\label{rem:weighted-Ck-space-eq-Ck-space}
Setting

\begin{equation*}
	\Weights_1:=C_c^\infty(U,\zeroinfty),
\end{equation*}

for a locally convex subset $U\subseteq\Real^n$, we obtain

\begin{equation*}
	C_{\Weights_1}^k(U,E)=C^k(U,E)
\end{equation*}

as topological vector spaces, for all $k\in\Natzeroinfty$. In fact, since each $f\in\Weights_1$ is continuous, the condition in Lemma
\ref{lem:k-inclusion-map-continuous} is satisfied (see Remark \ref{rem:weights-cont-then-inf-geq-0}),
whence the inclusion map

\begin{equation*}
	C_{\Weights_1}^k(U,E)\to C^k(U,E)
\end{equation*}

is continuous. Conversely, if $\gamma\in C^k(U,E)$, $f\in\Weights_1$, $\alpha\in\Natural_0^n$ with
$|\alpha|\leq k$ and $q\in\contsemiE$, then

\begin{align*}
	\left\|\gamma\right\|_{f,\alpha,q}
	&\defeq\sup_{x\in U}f(x)q\left(\partderx{\alpha}\gamma(x)\right)
	=\sup_{x\in\supp(f)}f(x)q\left(\partderx{\alpha}\gamma(x)\right)\\
	&\leq r\sup_{x\in\supp(f)}q\left(\partderx{\alpha}\gamma(x)\right)
	=r\left\|\gamma\right\|_{\supp(f),\alpha,q},
\end{align*}

where

\begin{equation*}
	r:=\max_{x\in\supp(f)}f(x)<\infty.
\end{equation*}

Thus $\gamma\in C_{\Weights_1}^k(U,E)$ and the inclusion map

\begin{equation*}
	C^k(U,E)\to C_{\Weights_1}^k(U,E)
\end{equation*}

is continuous. 

Therefore, using Corollary \ref{cor:weighted-Ckl-exp-law}, we know that

\begin{equation*}
	C^k(U,C_{\Weights_2}^l(V,E))\cong C_{\Weights_2}^l(V,C^k(U,E))
\end{equation*}

for open subsets $U\subseteq\Real^n$, $V\subseteq\Real^m$, a suitable set of weights $\Weights_2$ on $V$
and $k,l\in\Natzeroinfty$.

Moreover, using Corollary \ref{cor:weighted-Ckl-exp-law-cp-open}, we have

\begin{equation*}
	C^k(K,C_{\Weights_2}^l(V,E))\cong C_{\Weights_2}^l(V,C^k(K,E))
\end{equation*}

where $K\subseteq\Real^n$ is a convex, compact subset with $K^\circ\neq\emptyset$ and $V$, $\Weights_2$ as above.
\end{remark}

\section{Regularity of the Lie group $C_\Weights^l(U,H)$}

In this section we show that the Exponential Law for spaces of weighted differentiable functions (in particular,
Corollary \ref{cor:weighted-Ckl-exp-law-cp-open}) can be used to prove $C^k$-regularity of certain Lie groups.
Throughout the section, we will use the following definitions:

\begin{definition}\label{def:lie-group-on-lcx-space}
A \textit{Lie group}\index{Lie group} modeled on a \Hlc space $E$ is a group $G$, equipped with a $C^\infty$-manifold
structure modeled on $E$ which turns the group multiplication

\begin{equation*}
	\mu:G\times G\to G,\quad(g,h)\mapsto gh
\end{equation*}

and the inversion

\begin{equation*}
	\iota:G\to G,\quad g\mapsto g^{-1}
\end{equation*}

into smooth maps. 

We always denote by $e_G$ the neutral element of $G$ and by $\mathfrak{g}:=L(G):=T_{e_g}(G)$ the
Lie algebra of $G$.
\end{definition}

\begin{definition}\label{def:lie-group-Ck-reg}
Let $G$ be a Lie group modeled on a \Hlc space $E$ and $g\in G$. Using the tangent map of the left translation $\lambda_g:G\to G,x\mapsto gx$, the product of $g$ and $v\in T_x(G)$ for some $x\in G$ is defined as

\begin{equation*}
	g.v:=(T_x\lambda_g)(v)\in T_{gx}(G).
\end{equation*}

The Lie group $G$ is called \textit{$C^k$-semiregular}\index{Lie Group!$C^k$-semiregular}
(with $k\in\Natzeroinfty$) if for each $C^k$-curve $\gamma:\interclcl{0}{1}\to\mathfrak{g}$ there exists a (unique) 
$C^{k+1}$-curve $\Evol_G(\gamma):=\eta:\interclcl{0}{1}\to G$ such that

\begin{equation*}
	\eta(0)=e_G\mbox{ and } \eta'=\eta.\gamma.
\end{equation*}

The Lie group $G$ ia called \textit{$C^k$-regular}\index{Lie Group!$C^k$-regular} if $G$ is $C^k$-semiregular 
and the map

\begin{equation*}
	\evol_G:C^k(\interclcl{0}{1},\mathfrak{g})\to G,\quad\gamma\to\Evol_G(\gamma)(1)
\end{equation*}

is smooth.
\end{definition}

\begin{remark}\label{rem:lie-group-Ck-reg-iff}
From \cite[Lemma 3.1]{GlSemireg} follows, that a Lie group is $C^k$-regular if and only if
$G$ is $C^k$-semiregular and the map

\begin{equation*}
	\Evol_G:C^k(\interclcl{0}{1},\mathfrak{g})\to C^{k+1}(\interclcl{0}{1},G),\quad
	\gamma\mapsto\Evol_G(\gamma)
\end{equation*}

is smooth.
\end{remark}

For the theory of locally convex Lie groups the reader is referred to \cite{Milnor}, \cite{Neeb}
and \cite{GlNeeb}. Recall that if $H$ is a locally convex Lie group, then we can turn 
$C^l(\interclcl{0}{1},H)$ into a Lie group modeled on $C^l(\interclcl{0}{1},L(H))$.
Moreover, if $H$ is $C^k$-regular, then so is $C^l(\interclcl{0}{1},H)$ (more general, the last facts are 
true for Lie groups $C^l(K,H)$, where $K$ is a compact manifold, see \cite{GlSemireg}). 
By \cite[Rem. 13.4]{GlSemireg}, each $C^k$-regular Lie group $G$ is $C^l$-regular, for $l\geq k$.
Thus $C^\infty$-regularity, simply called \textit{regularity}\index{Lie group!regular}, going back to Milnor 
(\cite{Milnor}), is the weakest concept. Constructions of Lie groups of weighted smooth Lie group-valued mappings
can be found in \cite{Boseck-Czichowski-Rudolph}. In \cite{BWalter} B.Walter
describes the construction of Lie groups $C_\Weights^l(U,H)$ modeled on the space 
$C_\Weights^l(U,L(H))$, where $U$ is an open subset of a finite-dimensional space $X$ and $\Weights$ is a set of weights on $U$ such that the $o$-condition is satisfied and $\mathds{1}_U\in\Weights$. (In \cite{BWalter}, the 
author works with spaces $C_\Weights^l(U,H)^\bullet$, but these spaces coincide with $C_\Weights^l(U,H)$ if
the $o$-condition is satisfied by $\Weights$.) Moreover, the author shows that if $U$ 
is an open subset of a normed space, $\mathds{1}_U\in\Weights$ and $H$ is a Banach Lie group, then the Lie group $C_\Weights^l(U,H)$ is regular.

\medskip
From \cite{GlSemireg}, we derive the following fact:

\begin{lemma}\label{lem:Lie-group-Ck-regular-iff}
Let $G$ be a Lie group with Lie algebra $\mathfrak{g}$, $\Omega\subseteq C^k(\interclcl{0}{1},\mathfrak{g})$ 
be an open $0$-neighborhood and the map

\begin{equation*}
	\Theta:\Omega\to C^{k+1}(\interclcl{0}{1},G)
\end{equation*}

be smooth. If there exists a family $(\alpha_j)_{j\in J}$ of smooth homomorphisms

\begin{equation*}
	\alpha_j:G\to G_j
\end{equation*}

to Lie groups $G_j$ such that

\begin{equation*}
	\alpha_j\circ\Theta(\gamma) = \Evol_{G_j}(T_{e_G}\alpha_j\circ\gamma)
\end{equation*}

for each $\gamma\in\Omega$ and the maps $T_{e_G}\alpha_j$ separate points on $\mathfrak{g}$, then $G$ is $C^k$-regular and
$\Theta=\Evol_G$.
\end{lemma}

Now, we show the $C^k$-regularity of the Lie group $C_\Weights^l(U,H)$, checking that in the following situation all the hypotheses of Lemma \ref{lem:Lie-group-Ck-regular-iff} are satisfied.

\begin{theorem}\label{thm:weighted-Lie-group-Ck-regular}
Let $E$ be a \Hlc space and $H$ be a $C^k$-regular Lie group modeled on $E$, with $k\in\Natzeroinfty$. Let 
$\Weights\subseteq C^l(U,\zeroinfty)$ be a set of weights on an open subset $U\subseteq\Real^n$ such that

\begin{enumerate}
	\item [(i)] 	$\mathds{1}_U\in\Weights$,
	\item [(ii)]	$\Weights$ satisfies the $o$-condition,
	\item [(iii)]	for each $f\in\Weights$ and $\alpha\in\Natural_0^n$ with $|\alpha|\leq l$ there exists $g\in\Weights$ such that
					
					\begin{equation*}
						\left|\partderx{\alpha}f(x)\right|\leq g(x)
					\end{equation*}
					
					for all $x\in U$.
\end{enumerate}

Then the Lie group $G:=C_\Weights^l(U,H)$ is $C^k$-regular for each $l\in\Natzeroinfty$.
\end{theorem}

\begin{proof}
Let $\mathfrak{h}$ be the Lie algebra of $H$. For an open $e_H$-neighborhood $V\subseteq H$ and an open $0$-neighborhood $W\subseteq~\mathfrak{h}$ there exists a
$C^\infty$-diffeomorphism $\phi:V\to W$ such that

\begin{equation*}
	d\phi\big|_\mathfrak{h}=\id_\mathfrak{h}\mbox{ and }\phi(e_H)=0.
\end{equation*}

Consider the evaluation maps

\begin{equation*}
	\varepsilon_x:G=C_\Weights^l(U,H)\to H,\quad\gamma\mapsto\gamma(x)
\end{equation*}

and

\begin{equation*}
	\ev_x:\mathfrak{g}\cong C_\Weights^l(U,\mathfrak{h})\to\mathfrak{h},\quad\gamma\mapsto\gamma(x)
\end{equation*}

for $x\in U$. We have a commutative diagram

\begin{equation*}
\begin{xy}
	\xymatrix{
	C_\Weights^l(U,V)\ar[r]^{C_\Weights^l(U,\phi)}\ar[d]_{\varepsilon_x\big|_{C_\Weights^l(U,V)}}	
	&	C_\Weights^l(U,W)\ar[d]^{\ev_x\big|_{C_\Weights^l(U,W)}}\\
	V\ar[r]_\phi			&	W
	}
\end{xy}
\end{equation*}

where

\begin{equation*}
	C_\Weights^l(U,\phi):C_\Weights^l(U,V)\to C_\Weights^l(U,W)
\end{equation*}

is a $C^\infty$-diffeomorphism and $\ev_x$ is continuous and linear, hence
$C^\infty$ (cf. \cite{GlNeeb}). Therefore

\begin{equation*}
	\varepsilon_x\big|_{C_\Weights^l(U,V)}=\phi^{-1}\circ\ev_x\big|_{C_\Weights^l(U,W)}\circ C_\Weights^l(U,\phi)
\end{equation*}

and $\varepsilon_x$ is $C^\infty$. Moreover, the diagram

\begin{equation*}
\begin{xy}
	\xymatrix{
         \mathfrak{g}\ar[rr]^{dC_\Weights^l(U,\phi)\big|_\mathfrak{g}}\ar[d]_{T_{e_G}\varepsilon_x}
		&&   C_\Weights^l(U,\mathfrak{h})\ar[d]^{\ev_x}\\
         \mathfrak{h}\ar[rr]_{d\phi\big|_\mathfrak{h}=\id_\mathfrak{h}}
		&&   \mathfrak{h}   
	}
\end{xy}	
\end{equation*}

commutes, thus we identify each $T_{e_G}\varepsilon_x$ with $\ev_x$ 
(as $dC_\Weights^l(U,\phi)\big|_\mathfrak{g}$ is an isomorphism), which separates points on 
$\mathfrak{g}$. 

The Lie group $H$ was assumed $C^k$-regular, thus the map

\begin{equation*}
	\Evol_H:C^k(\interclcl{0}{1},\mathfrak{h})\to C^{k+1}(\interclcl{0}{1},H)
\end{equation*}

is $C^\infty$ (by Remark \ref{rem:lie-group-Ck-reg-iff}), and we have $\Evol_H(0)(t)=e_H$ for each $t\in\interclcl{0}{1}$. Therefore, there exists an open 
$0$-neighborhood $Q$ in $C^k(\interclcl{0}{1},\mathfrak{h})$ such that

\begin{equation*}
	\Evol_H(Q)\subseteq C^{k+1}(\interclcl{0}{1},V).
\end{equation*}

Setting

\begin{equation*}
	\Omega:=\Psi_1^{-1}(C_\Weights^l(U,Q)),
\end{equation*}

where 

\begin{equation*}
	\Psi_1:C^k(\interclcl{0}{1},C_{\Weights}^l(U,\mathfrak{h}))\to C_{\Weights}^l(U,C^k(\interclcl{0}{1},\mathfrak{h}))
\end{equation*}

is the homeomorphism from Corollary \ref{cor:weighted-Ckl-exp-law-cp-open}, we obtain an
open $0$-neighborhood $\Omega$ in $C^k(\interclcl{0}{1},C_\Weights^l(U,\mathfrak{h}))$. We now construct the $C^\infty$-map

\begin{equation*}
	\Theta:\Omega\to C^{k+1}(\interclcl{0}{1},C_\Weights^l(U,H))
\end{equation*}

as follows.

The map

\begin{equation*}
	\Theta_1:=\Psi_1\big|_\Omega:\Omega\to C_\Weights^l(U,Q)
\end{equation*}

is a restriction of a continuous and linear map, hence $C^\infty$. Further, the maps $C^{k+1}(\interclcl{0}{1},\phi)$ and
$\Evol_H\big|_Q$are $C^\infty$, whence the map

\begin{equation*}
	C_\Weights^l(U,C^{k+1}(\interclcl{0}{1},\phi)\circ\Evol_H\big|_Q):
	C_\Weights^l(U,Q)\to C_\Weights^l(U,C^{k+1}(\interclcl{0}{1},W))
\end{equation*}

is $C^\infty$ (see \cite{BWalter}), and we set

\begin{equation*}
	\Theta_2:=i\circ C_\Weights^l(U,C^{k+1}(\interclcl{0}{1},\phi)\circ\Evol_H\big|_Q):
	C_\Weights^l(U,Q)\to C_\Weights^l(U,C^{k+1}(\interclcl{0}{1},\mathfrak{h})).
\end{equation*}

with the inclusion map $i:C_\Weights^l(U,C^{k+1}(\interclcl{0}{1},W))\to 
C_\Weights^l(U,C^{k+1}(\interclcl{0}{1},\mathfrak{h}))$. By Corollary 
\ref{cor:weighted-Ckl-exp-law-cp-open}, the flip

\begin{equation*}
	\Psi_2:C_\Weights^l(U,C^{k+1}(\interclcl{0}{1},\mathfrak{h}))\to
	C^{k+1}(\interclcl{0}{1},C_\Weights^l(U,\mathfrak{h}))
\end{equation*}

is $C^\infty$ (being continuous and linear). Now, we have $im(\Psi_2\circ\Theta_2\circ\Theta_1)
\subseteq C^{k+1}(\interclcl{0}{1},C_\Weights^l(U,W))$, which is open
in $C^{k+1}(\interclcl{0}{1},C_\Weights^l(U,\mathfrak{h}))$, thus we obtain the $C^\infty$-map

\begin{equation*}
	\Theta_3:=(\Psi_2\circ\Theta_2\circ\Theta_1)\big|^{C^{k+1}(\interclcl{0}{1},C_\Weights^l(U,W))}:
	\Omega\to C^{k+1}(\interclcl{0}{1},C_\Weights^l(U,W)).
\end{equation*}

With the $C^\infty$-diffeomorphism 

\begin{equation*}
	C^{k+1}(\interclcl{0}{1},C_\Weights^l(U,\phi)^{-1}):C^{k+1}(\interclcl{0}{1},C_\Weights^l(U,W))\to
	C^{k+1}(\interclcl{0}{1},C_\Weights^l(U,V))
\end{equation*}

and the inclusion map $j:C^{k+1}(\interclcl{0}{1},C_\Weights^l(U,V))\to C^{k+1}(\interclcl{0}{1},C_\Weights^l(U,H))$ 
we get

\begin{multline*}
	\Theta_4:=j\circ C^{k+1}(\interclcl{0}{1},C_\Weights^l(U,\phi)^{-1}):
	C^{k+1}(\interclcl{0}{1},C_\Weights^l(U,W))\\ \to C^{k+1}(\interclcl{0}{1},C_\Weights^l(U,H))
\end{multline*}

and finally construct the $C^\infty$-map

\begin{equation*}
	\Theta:=\Theta_4\circ\Theta_3:\Omega\to C^{k+1}(\interclcl{0}{1},C_\Weights^l(U,H)).
\end{equation*}

It remains to show that

\begin{equation*}
	\varepsilon_x\circ\Theta(\gamma)=\Evol_H(\ev_x\circ\gamma)
\end{equation*}

for each $\gamma\in\Omega$ and $x\in U$. By construction of $\Theta$ we have

\begin{align*}
	\Theta(\gamma)
	&=(\Theta_4\circ\Theta_3)(t\mapsto(x\mapsto\gamma(t)(x)))\\
	&=(\Theta_4\circ\Psi_2\circ\Theta_2\circ\Theta_1)(t\mapsto(x\mapsto\gamma(t)(x)))\\
	&=(\Theta_4\circ\Psi_2\circ\Theta_2)(x\mapsto(t\mapsto\gamma(t)(x)))\\
	&=(\Theta_4\circ\Psi_2)(x\mapsto(s\mapsto\phi(\Evol_H(t\mapsto\gamma(t)(x))(s))))\\
	&=\Theta_4(s\mapsto(x\mapsto\phi(\Evol_H(t\mapsto\gamma(t)(x))(s))))\\
	&=(s\mapsto(x\mapsto\Evol_H(t\mapsto\gamma(t)(x))(s))),
\end{align*}

hence for $s\in\interclcl{0}{1}$ we see that

\begin{align*}
	(\varepsilon_x\circ\Theta(\gamma))(s)
	&=\Evol_H(t\mapsto\gamma(t)(x))(s)=\Evol_H(\ev_x\circ\gamma)(s),
\end{align*}

as required, and we obtain the desired result using Lemma \ref{lem:Lie-group-Ck-regular-iff}.
\end{proof}


\begin{thebibliography}{50}
\bibitem{Alz}
	Alzaareer, H.,
	\emph{Lie groups of mappings on non-compact spaces and manifolds},
	Paderborn, Univ., Diss., 2013,
	\url{nbn-resolving.de/urn:nbn:de:hbz:466:2-11572}

\bibitem{Bastiani}
	Bastiani, A.,
	\emph{Applications diff\'{e}rentiables et vari\'{e}t\'{e}s diff\'{e}rentiables de dimension infinie},
	J. Analyse Math. 13 (1964), 1 - 114

\bibitem{Bierstedt}
	Bierstedt, K.-D.,
	\emph{Gewichtete R\"aume stetiger vektorwertiger Funktionen und das injektive Tensorprodukt I},
	 J. Reine Angew. Math. 259 (1973), 186 - 210

\bibitem{Boseck-Czichowski-Rudolph}
	Boseck, H., Czichowski, G., Rudolph, K.-P.,
	Analysis on Topological Groups - General Lie Theory,
	Teubner-Texte zur Mathematik, Bd.37,
	Teubner Verlagsgesellschaft Leipzig, 1981

\bibitem{BourbGenTop510}
	Bourbaki, N.,
	General Topology 5-10,
	Springer-Verlag, 1998

\bibitem{Garnir-DeWilde-SchmetsIII}
	Garnir, H.G., De Wilde, M., Schmets, J.,
	Analyse fonctionnelle. Tome III: Espaces fonctionnels usuels,
	Birkh\"auser Verlag Basel, 1973

\bibitem{GlLieGrStr}
	Gl\"ockner, H.,
	\emph{Lie group structures on quotient groups and universal complexifications for infinite-dimensional Lie groups},
	J. Funct. Anal. 194 (2002), No.2, 347 - 409

\bibitem{GlSemireg} 
	Gl\"ockner, H., 
	\emph{Regularity properties of infinite-dimensional Lie groups, and semiregularity},
	preprint, \url{arXiv:1208.0715}
	
\bibitem{GlExpLaws}
	Gl\"ockner, H.,
	\emph{Exponential laws for ultrametric partially differentiable functions and applications},
	 extended preprint version \url{arXiv: 1209.1384} (cf. publication in p-Adic Numbers Ultrametric Anal. Appl. 5 (2013), no. 2, pp. 122 -159)
	
\bibitem{GlNeeb}
	Gl\"ockner, H., Neeb, K.-H.,
	\emph{Infinite-dimensional Lie groups},
	book in preparation
	
\bibitem{Keller}
	Keller, H.H.,
	Differential Calculus in Locally Convex Spaces,
	Lecture Notes in Mathematics 417,
	Springer-Verlag, 1974

\bibitem{KelleyGenTop}
	Kelley, J.L.,
	General Topology,
	Springer-Verlag, 1975

\bibitem{Kriegl-Michor-Rainer}
	Kriegl A., Michor P.W., Rainer A.,
	\emph{The exponential law for spaces of test functions and diffeomorphism groups},
	Indagationes Mathimaticae (2015) \url{dx.doi.org/10.1016/j.indag.2015.10.006}
	
\bibitem{Milnor}
	Milnor, J.,
	\emph{Remarks on infinite dimensional Lie groups},
	Relativit\'{e}, groupes et topologie II (Les Houches, 1983), 1007-1055

\bibitem{Nachbin}
	Nachbin, L.,
	\emph{Elements of approximation theory},
	Van Nostrand Mathematical Studies, No. 14, Princeton, 1967

\bibitem{Neeb}
	Neeb, K.-H.,
	\emph{Towards a Lie theory of locally convex groups},
	Jpn. J. Math. 1 (2006), no. 2, 291 - 468.

\bibitem{NikMT}
	Nikitin, N.,
	\emph{Exponential laws for weighted function spaces and regularity of weighted mapping groups},
	master's thesis,
	University of Paderborn, 2015 (advisor: Helge Gl\"ockner)
	
\bibitem{Prolla}
	Prolla, J.,
	\emph{Weighted spaces of vector-valued continuous functions},
	Ann. Mat. Pura Appl. (4) 89 (1971), 145 - 157
	
\bibitem{SchTopVec}
	Sch\"afer, H. H.,
	Topological Vector Spaces,
	Springer-Verlag, 1971
	
\bibitem{Summers}
	Summers, W.H.,
	\emph{A representation theorem for biequicontinuous completed tensor products of weighted spaces},
	Trans. Amer. Math. Soc. 146 (1969) 121 - 131
	
\bibitem{BWalter}
	Walter, B.,
	\emph{Weighted diffeomorphism groups of Banach spaces and weighted mapping groups},
	Diss. Math. 484, 2012

\bibitem{WalAnalysis1}
	Walter, W.,
	Analysis 1,
	Springer-Verlag, 2004

\bibitem{WalAnalysis2}
	Walter, W.,
	Analysis 2,
	Springer-Verlag, 1992
\end{thebibliography}
\end{document}